\documentclass[a4paper,12pt]{amsart}

\usepackage{hyperref}
\hypersetup{
    colorlinks=true,
    linkcolor=dodger,
    filecolor=dodger,
    urlcolor=dodger,
    citecolor=dodger,
}

\usepackage[top=3cm,left=2.5cm,right=2.5cm,bottom=3cm]{geometry}
\usepackage{relsize} %relative size of fonts
\usepackage{lineno}
\usepackage{moreenum} % more enumerate options, like greek, Greek
\usepackage{amsmath,amsthm,pb-diagram,amssymb,comment}
\usepackage{amsfonts}
\usepackage{graphicx,color}
\usepackage{enumitem, fancyhdr, dsfont, multicol}
\usepackage{thmtools,cleveref}
\usepackage[normalem]{ulem}
\usepackage{stmaryrd}   % more math symbols
\usepackage{textcomp}   % text companion fonts, like yen
\usepackage{mathtools}
\usepackage{fontawesome}
\usepackage{tikz,float}
\usepackage{mleftright}   % para ajustar paréntesis.. en realidad no veo pa qué

\setlist{itemsep=1ex}

    \DeclareMathOperator{\dom}{{\rm dom}}

    \newcommand{\thzfc}{\mathrm{ZFC}}

    \newcommand{\Bwf}{\mathcal{B}}

    \newcommand{\Ewf}{\mathcal{E}}
    
    \newcommand{\Gwf}{\mathcal{G}}
    
    \newcommand{\Iwf}{\mathcal{I}}

    \newcommand{\Mwf}{\mathcal{M}}
    \newcommand{\Nwf}{\mathcal{N}}

    \newcommand{\Swf}{\mathcal{S}}

    \newcommand{\bfrak}{\mathfrak{b}}
    \newcommand{\cfrak}{\mathfrak{c}}
    \newcommand{\dfrak}{\mathfrak{d}}
    
    \newcommand{\hfrak}{\mathfrak{h}}
    \newcommand{\gfrak}{\mathfrak{g}}
    
    \newcommand{\pfrak}{\mathfrak{p}}
    \newcommand{\rfrak}{\mathfrak{r}}
    \newcommand{\sfrak}{\mathfrak{s}}

    \newcommand{\menos}{\smallsetminus}
    
    \DeclareMathOperator{\pts}{\mathcal{P}}
    
    \newcommand{\frestr}{{\upharpoonright}}
    \DeclareMathOperator{\add}{{\rm add}}
    \DeclareMathOperator{\cov}{{\rm cov}}
    \DeclareMathOperator{\non}{{\rm non}}
    \DeclareMathOperator{\cof}{{\rm cof}}
    \DeclareMathOperator{\limdir}{{\rm limdir}}

    \newcommand{\Aor}{\mathbb{A}}
    \newcommand{\Bor}{\mathbb{B}}
    
    \newcommand{\Dor}{\mathbb{D}}
    \newcommand{\Eor}{\mathbb{E}}

    \newcommand{\Loc}{\mathbb{LOC}}
    
    \newcommand{\Por}{\mathbb{P}}
    \newcommand{\Pbb}{\mathbb{P}}
    \newcommand{\Qor}{\mathbb{Q}}

    \newcommand{\Qnm}{\dot{\mathbb{Q}}}

    \newcommand{\Ncal}{\mathcal{N}}
    \newcommand{\Mcal}{\mathcal{M}}
    
    \newcommand{\SNwf}{\mathcal{SN}}
    \newcommand{\Q}{\mathbb{Q}}
    \newcommand{\R}{\mathbb{R}}

    \DeclareMathOperator{\cf}{{\rm cf}}

    \newcommand{\imp}{\mathrel{\mbox{$\Rightarrow$}}}

    \newcommand{\sii}{\mathrel{\mbox{$\Leftrightarrow$}}}

    \newcommand{\la}{\langle}
    \newcommand{\ra}{\rangle}

   \newcommand{\Sbf}{\mathbf{S}}

\newcommand{\Dbf}{\mathbf{D}}

\DeclareMathOperator{\supp}{\mathrm{supp}}

\DeclareMathOperator{\seq}{\mathrm{seq}}
\DeclareMathOperator{\sqw}{\mathrm{seq}_{<\omega}}
\newcommand{\Fr}{\mathrm{Fr}}
\newcommand{\Rbf}{\mathbf{R}}

\newcommand{\Lc}{\mathbf{Lc}}
\newcommand{\aLc}{\mathbf{aLc}}
\DeclareMathOperator{\Lb}{\mathbf{Lb}}
\newcommand{\Hcal}{\mathcal{H}}
\newcommand{\Scal}{\mathcal{S}}
\newcommand{\id}{\mathrm{id}}
\newcommand{\Ed}{\mathbf{Ed}}

\DeclareMathOperator{\hgt}{\mathrm{ht}}
\DeclareMathOperator{\Seqw}{\mathrm{seq}_{<\omega}}

\DeclareMathOperator{\gsupp}{\mathrm{gsupp}}
\newcommand{\leqT}{\mathrel{\mbox{$\preceq_{\mathrm{T}}$}}}
\newcommand{\eqT}{\mathrel{\mbox{$\cong_{\mathrm{T}}$}}}
\newcommand{\supcof}{\mathrm{supcof}}
\newcommand{\supcov}{\mathrm{supcov}}
\newcommand{\minadd}{\mathrm{minadd}}
\newcommand{\minnon}{\mathrm{minnon}}
\newcommand{\minLc}{\mathrm{minLc}}

\newcommand{\baireincr}{\omega^{\uparrow\omega}}

\newcommand{\Cbf}{\mathbf{C}}

\newcommand{\set}[2]{\left\{#1 :\, #2\right\}}
\newcommand{\Seq}[2]{\la #1 :\, #2\ra}
\newcommand{\largeset}[2]{\left\{#1 :\, #2\right\}}

\newcommand{\Fn}{\mathrm{Fn}}
\newcommand{\gen}{\mathrm{gn}}

\newcommand{\balc}{\mathfrak{b}^{\mathrm{aLc}}}
\newcommand{\dalc}{\mathfrak{d}^{\mathrm{aLc}}}

\newcommand{\onebf}{\mathbf{1}}
\newcommand{\zerobf}{\mathbf{0}}
\newcommand{\DS}{\mathbf{DS}}

\DeclareMathOperator{\scf}{{\rm scf}}

\newcommand\subsetdot{\mathrel{\ooalign{$\subset$\cr
  \hidewidth\hbox{$\cdot\mkern3mu$}\cr}}}

\newcommand{\leqtr}{\triangleleft}
\newcommand{\nleqtr}{\ntriangleleft}
\newcommand{\nsqsubset}{\not\sqsubset}
\newcommand{\nsqsubseteq}{\not\sqsubseteq}

\definecolor{sub0}{RGB}{29,32,137}
\definecolor{sub1}{RGB}{1,71,157}
\definecolor{sub2}{RGB}{1,104,183}
\definecolor{sub3}{RGB}{0,160,234}
\definecolor{sug}{RGB}{0,154,68}
\definecolor{suy}{RGB}{208,219,1}

\newcommand{\symdf}{\mathrel{\mbox{\scriptsize $\triangle$}}}

\definecolor{dodger}{rgb}{0.0,0.5,1.0}
\definecolor{carrotorange}{rgb}{0.93, 0.57, 0.13}

\title{Separating cardinal characteristics of the strong measure zero ideal}

\author{J\"org Brendle}
\address{Graduate School of System Informatics, Kobe University,
Rokko--dai 1--1, Nada--ku, 657--8501 Kobe, Japan}
\email{\href{mailto:brendle@kobe-u.ac.jp}{brendle@kobe-u.ac.jp}}

\author{Miguel A.~Cardona}
\address{Einstein Institute of Mathematics,
The Hebrew University of Jerusalem. Givat Ram, Jerusalem, 9190401, Israel}
\email{\href{mailto:miguel.cardona@mail.huji.ac.il}{miguel.cardona@mail.huji.ac.il}}
\urladdr{\url{https://sites.google.com/mail.huji.ac.il/miguel-cardona-montoya/home-page}}

\author{Diego A.~Mej\'ia}
\address{Graduate School of System Informatics, Kobe University,
Rokko--dai 1--1, Nada--ku, 657--8501 Kobe, Japan}
\email{\href{mailto:damejiag@people.kobe-u.ac.jp}{damejiag@people.kobe-u.ac.jp}}
\urladdr{\url{https://researchmap.jp/mejia?lang=en}}

\thanks{The first author was partially supported by Grant-in-Aid for Scientific Research (C) 18K03398, Japan Society for the Promotion of Science; the second author was partially supported by the Slovak Research and Development Agency under Contract No.~APVV-20-0045 and by Pavol Jozef \v{S}af\'arik University at a postdoctoral position; and the third author was supported by the Grants-in-Aid for Scientific Research (C) 23K03198, Japan Society for the Promotion of Science.}

\subjclass[2020]{03E17, 03E35, 03E40}

\keywords{Strong measure zero sets, cardinal characteristics of the continuum, forcing iteration theory}

\begin{document}

\makeatletter
\def\@roman#1{\romannumeral #1}
\makeatother

\newcounter{enuAlph}
\renewcommand{\theenuAlph}{\Alph{enuAlph}}

\numberwithin{equation}{section}
\renewcommand{\theequation}{\thesection.\arabic{equation}}

\theoremstyle{plain}
  \newtheorem{theorem}[equation]{Theorem}%[section]
  \newtheorem{corollary}[equation]{Corollary}
  \newtheorem{lemma}[equation]{Lemma}
  \newtheorem{mainlemma}[equation]{Main Lemma}
  \newtheorem*{mainthm}{Main Theorem}
  \newtheorem{prop}[equation]{Proposition}
  \newtheorem{clm}[equation]{Claim}
  \newtheorem{fact}[equation]{Fact}
  \newtheorem{exer}[equation]{Exercise}
  \newtheorem{question}[equation]{Question}
  \newtheorem{problem}[equation]{Problem}
  \newtheorem{conjecture}[equation]{Conjecture}
  \newtheorem{assumption}[equation]{Assumption}
  \newtheorem*{thm}{Theorem}
  \newtheorem{teorema}[enuAlph]{Theorem}
  \newtheorem*{corolario}{Corollary}
\theoremstyle{definition}
  \newtheorem{definition}[equation]{Definition}
  \newtheorem{example}[equation]{Example}
  \newtheorem{remark}[equation]{Remark}
  \newtheorem{notation}[equation]{Notation}
  \newtheorem{context}[equation]{Context}

  \newtheorem*{defi}{Definition}
  \newtheorem*{acknowledgements}{Acknowledgements}

\def\sectionautorefname{Section}
\def\subsectionautorefname{Subsection}
%\maketitle

%\tableofcontents

\begin{abstract}
Let $\mathcal{SN}$ be the $\sigma$-ideal of the strong measure zero sets of reals. We present general properties of forcing notions that allow to control of the additivity of $\mathcal{SN}$ after finite support iterations. This is applied to force that the four cardinal characteristics associated with $\mathcal{SN}$ are pairwise different: 
\[\mathrm{add}(\mathcal{SN})<\mathrm{cov}(\mathcal{SN})<\mathrm{non}(\mathcal{SN})<\mathrm{cof}(\mathcal{SN}).\]
Furthermore, we construct a forcing extension satisfying the above and Cicho\'n's maximum (i.e.\ that the non-dependent values in Cicho\'n's diagram are pairwise different).
\end{abstract}

\maketitle

\section{Introduction}\label{SecIntro}

%Strong measure zero sets have been introduced by Borel when he deal with describing the measure zero sets. Since then strong measure zero sets have been receiving a lot of attention in Set Theory.

%A set $X\subseteq \mathbb R$ has \emph{strong measure zero} if, for every sequence
%$(\epsilon_n)_{n\in\omega}$ of positive reals, there are open intervals %$(I_n)_{n\in\omega}$ such that $\Lb(I_n)\leq\varepsilon_n$ and $A\subseteq\bigcup_{n\in\omega} I_n$, where $\Lb$ denotes the Lebesgue measure on $\R$.

%He conjectured that each subset of the real line that has strong measure zero is countable, which is known as Borel's Conjecture (BC). Sierpi\'nski ~\cite{S} showed that the Continuum Hypothesis implies that BC fails, and  Laver~\cite{Laver} proved the consistency of BC with ZFC using forcing. So BC cannot be proven or refuted in ZFC. 

In this paper, we provide answers to two open questions addressed in~\cite{CMR,cardona} related to the consistency of the cardinal characteristics associated with the strong measure zero ideal, in particular concerning its additivity number. %These cardinals are  presented in~\autoref{SecPre}. 

Denote by $\SNwf$ the ideal of strong measure zero subsets of $\R$ (or of the Cantor space $2^\omega$). The cardinal characteristics associated with $\SNwf$ have also been interesting objects of research, in particular, when related to the cardinals in Cicho\'n's diagram as well as other cardinal characteristics.  %In this context,~\autoref{Cichonwith_SN} illustrates the provable inequalities in $\thzfc$ among cardinal characteristics of $\SNwf$ with cardinals in Cicho\'n's diagram.
%
%
%
% it is well-known in $\thzfc$ that 
%
% \begin{enumerate}[label=(\faThemeisle\arabic*)]
%     \item $\add(\Nwf)\leq\add(\SNwf)$ (Carlson~\cite{Carlson}),
%     \item $\cov(\Mwf)\leq\non(\SNwf)$ and $\add(\Mwf)=\min\{\bfrak,\non(\SNwf)\}$ (Miller~\cite{Mi1982}),
%     \item\label{itemc} $\cof(\SNwf)\leq 2^\dfrak$ (Osuga, see~\cite{O08}).
%     %\item $\cof(\SNwf)\leq\cov\left(\left([\supcof]^{<\minadd}\right)^\dfrak\right)$ (the second author and third author, see~\cite{CM23}).
% \end{enumerate}
%
Recently, the \emph{Yorioka ideals} $\Iwf_f$, parametrized by increasing functions $f\in \omega^\omega$, had been playing an important role in understanding the combinatorics of $\SNwf$ and its cardinal characteristics. 
%the second and third authors~\cite{CM23} obtained an upper bound of $\cof(\SNwf)$ in ZFC in terms of $\sigma$-ideals $\Iwf_f$ parametrized by increasing functions $f\in\omega^\omega$, which are known as \emph{Yorioka ideals} (see~\autoref{DefYorio}).  
%(Yorioka~\cite{Yorioka} introduced these ideals to characterize $\SNwf$, concretely, $\SNwf=\bigcap\{\Ical_f:\, f\in\omega^\omega\text{ increasing}\}$ and $\Iwf_f\subseteq\Nwf$.) 
%The second and third authors proved that $\cof(\SNwf)\leq\cov\left(\left([\supcof]^{<\minadd}\right)^\dfrak\right)$. 
%These cardinals have also is related to cardinals in Cicho\'n's diagram (see~\autoref{Cichonwith_SN}).
%
% appeared in many contexts and gave the following well-known characterizations.
% 
%
% \begin{theorem}[{\cite{Mi1982,O08, CM,CM23}}]
% \mbox{}
% \begin{enumerate}[label=\rm(\arabic*)]
%     \item $\non(\SNwf)=\minnon$.
%    
%     \item $\add(\Mwf)=\min\{\bfrak,\non(\SNwf)\}$.
%    
%     \item $\cof(\Mwf)=\max\{\dfrak,\supcov\}$.
%    
%     \item $\add(\Nwf)\leq\minadd\leq\add(\Mwf)$ and $\cof(\Mwf)\leq\supcof\leq\cof(\Nwf)$.
%    
%     \item $\add(\Nwf)\leq\minadd\leq\add(\SNwf)$.
%
%     %\item $\cof(\SNwf)\leq\cov\left(\left([\supcof]^{<\minadd}\right)^\dfrak\right).$
% \end{enumerate}
% \end{theorem}
%
\autoref{Cichonwith_SN} illustrates the provable inequalities among the cardinal characteristics associated with $\SNwf$ and $\Iwf_f$, and Cicho\'n's diagram~\cite{Mi1982,BJ,O08,KO08, CM,CM23}. (See \autoref{SecPre} for definitions.)

\begin{figure}[ht!]
\begin{center}
  \includegraphics[width=\linewidth]{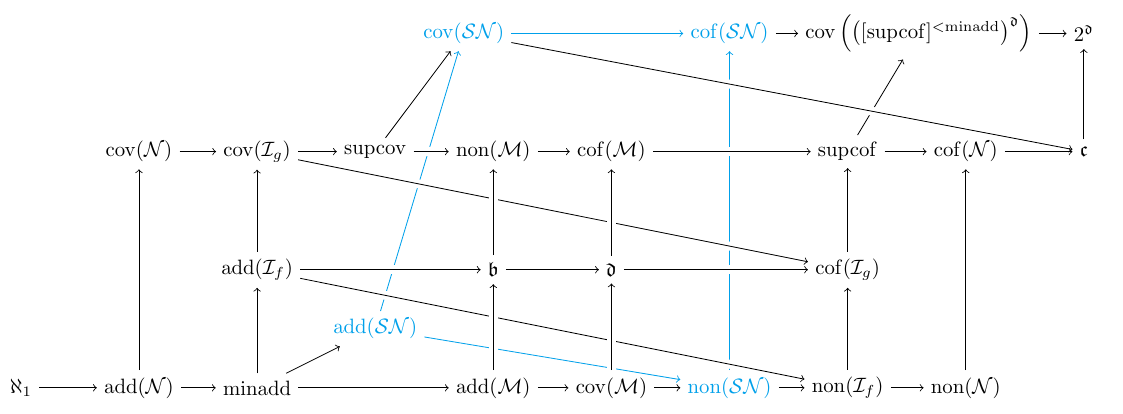}
 \caption{Cicho\'n's diagram and the cardinal characteristics associated with $\SNwf$ and the Yorioka ideals.
 An arrow  $\mathfrak x\rightarrow\mathfrak y$ means that (provably in ZFC) 
    $\mathfrak x\le\mathfrak y$. Moreover, $\add(\Mwf)=\min\{\bfrak,\non(\SNwf)\}$ and $\cof(\Mwf)=\max\{\dfrak,\supcov\}$, and $\minnon = \non(\SNwf)$. Two $\Iwf_f$ and $\Iwf_g$ are illustrated for arbitrary $f$ and $g$ to emphasize that the covering of any Yorioka ideal is an upper bound of the additivities of all Yorioka ideals, likewise for the cofinality and uniformity.}
  \label{Cichonwith_SN}
\end{center}
\end{figure}

As to consistency results, by using finite support (FS, for short) iterations of ccc forcings,
Pawlikowski~\cite{paw85} constructed a model satisfying $\bfrak=\aleph_1<\add(\SNwf)=\cfrak=\aleph_2$, while Judah and Shelah~\cite{JS89} constructed a model satisfying $\add(\SNwf)=\aleph_1<\add(\Mwf)=\cfrak=\aleph_2$ by iterating with precaliber $\aleph_1$ posets. Motivated by the latter, Pawlikowski~\cite{P90} proved that precaliber $\aleph_1$ posets do not increase $\cov(\SNwf)$, which shows that $\cov(\SNwf) = \aleph_1$ in Judah's and Shelah's model. Later, Goldstern, Judah, and Shelah~\cite{GJS} 
used a countable support (CS) iteration of proper forcings to prove the consistency of $\cov(\Mwf)=\bfrak=\aleph_1<\non(\SNwf)=\dfrak=\cfrak=\aleph_2$ as well as the consistency of $\cof(\Mwf)=\aleph_1<\add(\SNwf)=\aleph_2=\cfrak$. Modifying the latter result, Judah and Bartoszyinski~\cite[Thm.~8.4.11]{BJ} obtained a model where  $\dfrak=\aleph_1<\add(\SNwf)=\cov(\Nwf)=\aleph_2$. Almost 9 years later, Yorioka introduced his ideals to show that no inequality between $\cof(\SNwf)$ and $2^{\aleph_0}$ can be proved in $\thzfc$.

A natural question in the context of the above consistency results is whether $\cov(\SNwf)\leq\cof(\Nwf)$. The answer to this question is negative since %the second and third authors and Rivera-Madrid~\cite{CMR} showed that 
$\cof(\Nwf)=\aleph_1<\cov(\SNwf)=\aleph_2$ holds in Sacks model.\footnote{Marczewski~\cite{Marcz} proved that every strong measure zero set is in the Marczewski ideal $s^0$, and it is clear that $\cov(s^0)=\aleph_2$ in Sacks model, while $\cof(\Ncal)=\aleph_1$ follows by the well-known Sacks property (see e.g.~\cite[Model~7.6.2]{BJ}). Another proof of $\cov(\SNwf)=\aleph_2$ in Sacks model, indicated to us by the referee, follows from Miller's result~\cite{Milmap} stating that, in this model, every subset of the Cantor space of size $\cfrak$ maps continuously onto the Cantor space. A more recent proof for iterations of tree forcings in general is due by the authors with Rivera-Madrid~\cite{CMR}.} This was used to prove the consistency of \begin{equation}\tag{\faLeaf}
\add(\SNwf)=\non(\SNwf)<\cov(\SNwf)<\cof(\SNwf),\label{3SNtwo}\end{equation} which was the first consistency result where more than two cardinal characteristics associated with $\SNwf$ are pairwise different. On the other hand, the second author~\cite{cardona} generalized Yorioka's characterization of $\cof(\SNwf)$ and, using finite support iteration techniques, proved the consistency of \begin{equation}\tag{\faPagelines}\add(\SNwf)=\cov(\SNwf)<\non(\SNwf)<\cof(\SNwf).\label{3SNone}\end{equation}
%
%These two latest results state possible ways to separate the four cardinal characteristics associated with $\SNwf$. 
Furthermore, employing the technique of matrix iterations with vertical support restrictions
of the third author~\cite{mejiavert}, the second and third authors~\cite{CM23} refined the model of~\eqref{3SNone} and obtained that $\non(\SNwf)$ may be singular in such model.

The main challenge in both results is to separate $\add(\SNwf)$ from the rest. So we ask:

\begin{question}[{\cite[Question~5.1]{CMR}, see also~\cite[Sec.~5]{cardona}}]\label{Mainadd}
Is it consistent with $\thzfc$ that \[\add(\SNwf)<\min\{\cov(\SNwf),\non(\SNwf)\}?\]
\end{question}

So far, in the context of FS iterations, it is only known that precaliber $\aleph_1$ posets do not increase $\add(\SNwf)$ (as shown by Judah and Shelah~\cite{JS89}), but this is not enough to solve the question because such posets do not increase $\cov(\SNwf)$
(as discovered by Pawlikowski~\cite{P90}). 

We solve the previous question in this paper: we present a new property that helps us to force the additivity of $\SNwf$ small. Concretely, we introduce a new 
Polish relational system $\Rbf^f_\Gwf$ for an increasing function $f\in\omega^\omega$ and a countable family of increasing functions $\Gwf\subseteq\omega^\omega$, and prove that FS iterations of $\Rbf^f_\Gwf$-good posets force $\add(\SNwf)$ small.

\begin{teorema}[\autoref{thm:Paddsn}]\label{mainpresaddSN}
Let $\theta$ be an uncountable regular cardinal, $\lambda=\lambda^{\aleph_0}$ a cardinal and let $\pi=\lambda\delta$ (ordinal product) for some ordinal $0<\delta<\lambda^+$. Assume $\theta\leq \lambda$ and $\cf(\pi)\geq \omega_1$. If $\Por$ is a FS iteration of length $\pi$ of non-trivial ccc $\theta$-$\Rbf^f_\Gwf$-good posets of size ${\leq}\lambda$ 
then $\Por$ forces
$\add(\SNwf)\leq\theta$ and $\lambda\leq\cof(\SNwf)$. 
\end{teorema}

The terms ``Polish relational system" and ``goodness" are reviewed in \autoref{SecPres}. These are old notions developed by the first author~\cite{Br} and Judah and Shelah~\cite{JS}, but with many recent updates as in~\cite[Sec.~4]{CM}.

We prove that instances of $\Rbf^f_\Gwf$-good posets are (posets whose completions are) Boolean algebras with a strictly positive finitely additive measure $\mu$. %satisfying the following version of the \emph{Lebesgue Density Theorem:} there is some countable set $S$ such that, for any non-zero $a$ and any $\varepsilon>0$, there is some $s\in S$ such that $\mu(a\wedge s)>\mu(s)(1-\varepsilon)$ (\autoref{mainlemma}). 
We remark that Kamburelis~\cite{Ka} proved that such posets do not increase $\add(\Nwf)$ (see \autoref{ExmPrs}~\ref{ExmPrsd}). 
Examples are random forcing, any of its subalgebras, and the $\sigma$-centered posets. %(see~\autoref{exm:densepfam}). 

As a direct consequence, \autoref{Mainadd} is solved:

\begin{corollary}
Let $\Por$ be the FS iteration of length $\nu$ of random forcing $\Bor$, where $\nu=\nu^{\aleph_0}$ is a cardinal. Then, in the $\Por$-extension, 
$\add(\SNwf)=\aleph_1$ and $\cov(\SNwf)=\non(\SNwf)=\nu$.    
\end{corollary}

%\autoref{exm:densepfam} together with the Corollary give a complete solution to~\autoref{Mainadd}

\begin{proof}
   It is well-known that $\Por$ forces $\cov(\Nwf)=\cov(\Mwf)= \cfrak =\nu$. 
   Since $\cov(\Nwf)\leq\cov(\SNwf)$ and $\non(\Mwf)\leq\non(\SNwf)$, $\Por$ forces $\cov(\SNwf)=\non(\SNwf)=\nu$. The equality $\add(\SNwf)=\aleph_1$ is forced thanks to~\autoref{mainpresaddSN}.
\end{proof}

Using \autoref{mainpresaddSN}, we also prove that the four cardinal characteristics ($\add$, $\cov$, $\non$ and $\cof$) associated with $\SNwf$ can be forced pairwise different, which strengthens~(\ref{3SNone}) and solves~{\cite[Question~5.2]{CMR}}. This is obtained by iterating (with finite support) restrictions of random forcing.

\begin{teorema}[{\autoref{thm4SN}}]\label{4SN}
There is a ccc poset that forces
   \[\add(\SNwf)<\cov(\SNwf)<\non(\SNwf)<\cof(\SNwf).\]    
\end{teorema}

%\Miguel{Motivar el siguiente teorema.}
Furthermore, we can prove the previous consistency result by forcing, simultaneously, \emph{Cicho\'n's maximum}, i.e.\ that all the non-dependent values in Cicho\'n's diagram are pairwise different.

\begin{teorema}[{\autoref{cor1:4SNmax}}]\label{Main4SN}
% Let be $\theta_1\leq\theta_2\leq\theta_3\leq\theta_4\leq\theta_5 \leq \theta_6 = \theta_6^{<\theta_6}$ be uncountable regular cardinals, and let $\nu$ be a cardinal such that $\nu^{\theta_6}=\nu$. Then there is a cofinality-preserving poset that forces the constellation in \autoref{Fig4SN}.
Let $\lambda = \lambda^{\aleph_0}$ be a cardinal and, for $1\leq i\leq 5$, let $\lambda^\bfrak_i$ and $\lambda^\dfrak_i$ be  uncountable regular cardinals
such that $\lambda^\bfrak_i\leq \lambda^\bfrak_j \leq \lambda^\dfrak_j \leq \lambda^\dfrak_i\leq\lambda$ for any $i\leq j$,  and assume that $\cof\left(([\lambda^\dfrak_1]^{<\lambda^\bfrak_1})^{\lambda^\dfrak_4}\right) = \lambda^\dfrak_1$.\footnote{The latter assumption implies $\lambda^\dfrak_4<\lambda^\dfrak_1$.} Then there is a ccc poset forcing the constellation in~\autoref{cichonmaxSN2}.
\end{teorema}

\begin{figure}[ht!]
\begin{center}
  \includegraphics[width=\linewidth]{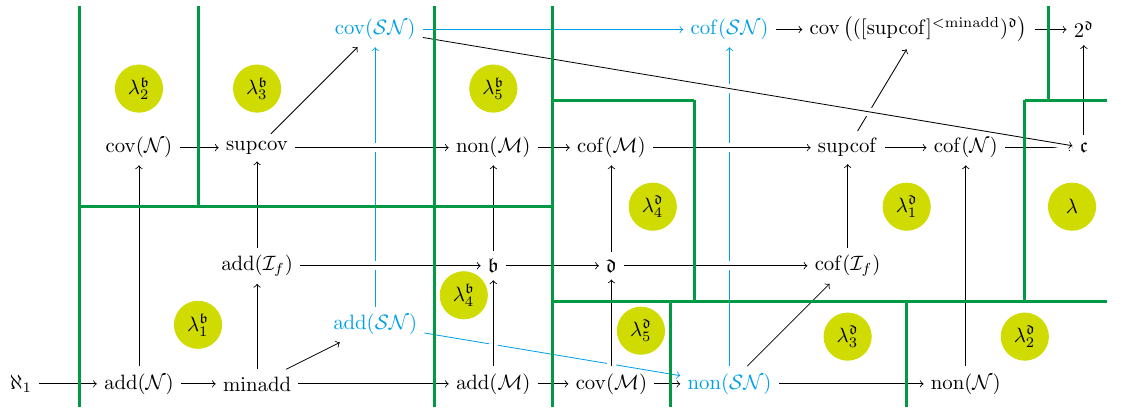}
 \caption{Cicho\'n's maximum as forced in \autoref{Main4SN}, with the four cardinal characteristics associated with $\SNwf$ pairwise different. Here $\Iwf_f$ is illustrated for any $f$, so their four cardinal characteristics are pairwise different, too. 
 The value of $2^\dfrak$ is the ground-model value of $\lambda^{\lambda^\dfrak_4}$.}
  \label{cichonmaxSN2}
\end{center}
\end{figure}

The assumption $\cof\left(([\lambda^\dfrak_1]^{<\lambda^\bfrak_1})^{\lambda^\dfrak_4}\right) = \lambda^\dfrak_1$ is feasible since it can be easily forced by adding $\lambda^\dfrak_4$-Cohen reals (on $2^{\lambda^\dfrak_4}$); and the value forced to $2^\dfrak$ can also be larger than (or equal to) $\cfrak$. More details are presented at the beginning of \autoref{SecConstr}.

In the direction of the previous results, the third author~\cite{mejiamatrix} used a forcing matrix iteration to force that $\add(\Nwf)<\cov(\Nwf)<\non(\Nwf)<\cof(\Nwf).$
Afterward, the second and third authors~\cite{CM} established the consistency of 
$\add(\Iwf_f)<\cov(\Iwf_f)<\non(\Iwf_f)<\cof(\Iwf_f)$ for any $f$ above some fixed $f^*$ again using a forcing matrix iteration. Later, Goldstern, Kellner, and Shelah~\cite{GKScicmax} used four strongly compact cardinals to obtain the consistency of Cicho\'n's maximum. In this model, $\add(\Mwf)<\non(\Mwf)<\cov(\Mwf)<\cof(\Mwf)$. This result was improved by the authors~\cite{BCM}, who developed a new method of two-dimensional iterations with ultrafilters (called~\emph{${<}\mu$-uf-extendable matrix iterations}, see~\autoref{Defmatuf}) to show the consistency with ZFC of $\add(\Mwf)<\non(\Mwf)<\cov(\Mwf)<\cof(\Mwf)$ and $\add(\Iwf_f)<\cov(\Iwf_f)<\non(\Iwf_f)<\cof(\Iwf_f)$ for all $f$, without large cardinals.  
Later, Goldstern, Kellner,  the third author, and Shelah~\cite{GKMS} proved, by intersecting ccc forcings with $\sigma$-closed models, that no large cardinals are needed for the consistency of Cicho\'n's maximum.   

More recently, the ${<}\mu$-uf-extendable matrix iterations have been used by the second author~\cite{Car4E} to obtain the consistency of $\add(\Ewf)<\non(\Ewf)<\cov(\Ewf)<\cof(\Ewf)$ with ZFC, where $\Ewf$ is the ideal generated by the $F_\sigma$ measure zero subsets of the reals. On the other hand, the first author~\cite{Brshatt} introduced a powerful technique called~\emph{Shattered iterations} to obtain a model of ZFC satisfying $\aleph_1<\cov(\Mwf)<\non(\Mwf)$, in which $\add(\Nwf)<\non(\Nwf)<\cov(\Nwf)<\cof(\Nwf)$ holds. The consistency of $\add(\Mwf)<\cov(\Mwf)<\non(\Mwf)<\cof(\Mwf)$ and $\add(\Ewf)<\cov(\Ewf)<\non(\Ewf)<\cof(\Ewf)$ are still open.

The proof of~\autoref{Main4SN} has a similar flow as the proof of Cicho\'n's maximum~\cite{GKMS}: we first force the constellation with $\cov(\Mwf)=\cfrak$, while separating everything we need on the left side, and then apply the method of intersections with $\sigma$-closed models for the final forcing construction. The first step uses iterations with ultrafilter limits as in~\cite{GMS,GKScicmax,BCM}, but we follow the two-dimensional version from~\cite{BCM}; the second step's method is original by~\cite{GKMS} and it is reviewed in~\cite{Brmodtec,CM22}, but our presentation is closer to~\cite{CM22}.

In the resulting constellation (\autoref{cichonmaxSN2}), $\cov(\SNwf)$ and $\non(\SNwf)$ are different from everything else, but $\add(\SNwf) = \add(\Nwf)$ and $\cof(\SNwf) = \cof(\Nwf)$. The problem about forcing in addition $\add(\Nwf) < \add(\SNwf)$ is discussed in \autoref{sec:Q}. 
Although $\add(\Iwf_f) = \add(\Nwf)$ and $\cof(\Iwf_f)= \cof(\Nwf)$ for all $f$, not all $\cov(\Iwf_f)$ may be forced to be the same, likewise for $\non(\Iwf_f)$ (see \autoref{rem:covIf} and~\ref{remcovnon}). Despite this, the four cardinal characteristics associated with the Yorioka ideals are pairwise different. 

The constellations we force can be adapted to methods from~\cite{GKMSnoreal,GKMSsplit} to separate, in addition, cardinal characteristics like the \emph{pseudo-intersection number} $\pfrak$, the \emph{distributivity number} $\hfrak$, the \emph{group-wise density number} $\gfrak$, the \emph{splitting number} $\sfrak$, the \emph{reaping number} $\rfrak$ and some cardinal characteristics associated with variations of Martin's axiom. Since the details can be understood from the references, we do not work with those cardinals in this paper.

%To prove~\autoref{Main4SN}, we construct a suitable ${<}\mu$-uf-extendable matrix iteration where \autoref{mainpresaddSN} can be applied.

\subsection*{Structure of the paper} In~\autoref{SecPre}, we review all the essentials related to relational systems and the Tukey order, measure and category on $\prod b$, and related forcing notions and their properties. We review in~\autoref{SecPres} the second's and third's author~\cite{CM} preservation theory of unbounded families, which is a generalization of Judah's and Shelah’s~\cite{JS} and the first author’s~\cite{Br} preservation theory.

\autoref{sec:good} is devoted to our general goodness support theory, which formalizes a technical device introduced by the first author and Switzer~\cite[Subsec.~4.4]{BS22}. In~\autoref{preaddSN}, we introduce the relational system $\Rbf^f_\Gwf$ and prove~\autoref{mainpresaddSN} using the goodness theory introduced in~\autoref{sec:good}. 
%that they behave well in this preservation theory. This allows us to prove~\autoref{mainpresaddSN}. 
We present applications of~\autoref{mainpresaddSN} in \autoref{SecConstr}, where we prove~\autoref{4SN} along with several constellations of Cicho\'n's diagram. In \autoref{sec:subm} we review the method of forcing-intersection with $\sigma$-closed submodels and show new results about its effect on $\SNwf$, which works to prove \autoref{Main4SN} and several variants. 
Discussions and open questions are presented in \autoref{sec:Q}.

%\begin{enumerate}
 %   \item[(A1)$_{\Iwf}$] $\add(\Iwf)<\cov(\Iwf)<\non(\Iwf)<\cof(\Iwf)$, and
    
 %   \item[(A2)$_{\Iwf}$] $\add(\Iwf)<\non(\Iwf)<\cov(\Iwf)<\cof(\Iwf)$. 
%\end{enumerate}

%The following questions have been raised by the second and third authors in~\cite{cardona,CMR}, respectively. 

%\begin{question}\label{Mainq}
%Are each one of the following statements consistent with ZFC?    
%\begin{enumerate}[label=  \rm (Q\arabic*)]
%    \item\label{Mainqtwo} $\twomainP_\SNwf$.   
%    \item\label{Mainqthree} $\onemainP_\SNwf$.
%\end{enumerate}
%\end{question}

\section{Relational systems and cardinal characteristics}\label{SecPre}

This section summarizes the notation and preliminary facts we need to present our main theory and results. \autoref{Ebhdense} and~\ref{EDsepfam} are facts original to this paper. 
%
%This section contains some definitions and basic facts~\rojo{$\ldots$}
%
%
%\subsection{Notation}\label{subsec:not} 
We use standard set-theoretic notation as in~\cite{kunen80,Ke2,jech,kunen}. 
%
%For any set $A$, $\id_A$ denotes the identity function on $A$, and $\id:=\id_\omega$. 
Given a formula $\phi$,
\begin{itemize}
    \item[{}] $\forall^\infty\, n<\omega\colon \phi$ means that all but finitely many natural numbers satisfy $\phi$; and
    \item[{}] $\exists^\infty\, n<\omega\colon \phi$ means that infinitely many natural numbers satisfy $\phi$.
\end{itemize}

% For $x,y:\omega\to\mathrm{On}$, we write \azul{(1) and (2) can appear more generally in the examples of relational systems}
% \begin{enumerate}[label= \rm (\arabic*)]
%     \item $x\leq y$ if $x(n)\leq y(n)$ for all $i<\omega$
    
%     \item $x\leq^* y$ if $\forall^\infty n\, :x(n)\leq y(n)$, which is read \emph{$x$ is dominated by $y$}. Likewise, $x<^* y$ is defined.
    
% \end{enumerate}

%A \emph{slalom} is a function $\varphi\colon \omega\to[\omega]^{<\omega}$. For any functions $x$, $y$ and $\varphi$ with domain $\omega$, we write $x\in^*\varphi$ if $\forall^\infty\, n<\omega\colon x(n)\in\varphi(n)$, which is read \emph{$\varphi$ localizes $x$} and write $x=^\infty y$ if $x(n)=y(n)$ for infinitely many $n$. The relation $x\neq^\infty y$ expresses that \emph{$x$ and $y$ are eventually different}. 

%We use the notation $x\leq y$ quite indiscriminately when $x$ and $y$ are functions with the same domain, meaning that $x(i)$ is below $y(i)$, with respect to some (partial) order on the $i$-th coordinate, for all $i\in\dom x$. When $x$ and $y$ has domain some limit ordinal $\gamma$ (like $\omega)$, we write $x\leq^* y$ when $\exists\, \alpha<\gamma\ \forall \beta\geq\gamma\colon x(\beta)\leq y(\beta)$. %\azul{context is missing}

\subsection{Relational systems and Tukey order}
 \ 

We closely follow the presentation in~\cite[Sect.~1]{CM22}, which is based on~\cite{Vojtas,BartInv,blass}. 
%
%\begin{definition}\label{def:relsys}
We say that $\Rbf=\la X, Y, \sqsubset\ra$ is a \textit{relational system} if it consists of two non-empty sets $X$ and $Y$ and a relation $\sqsubset$.% contained in $X\times Y$.
\begin{enumerate}[label=(\arabic*)]
    \item A set $F\subseteq X$ is \emph{$\Rbf$-bounded} if $\exists\, y\in Y\ \forall\, x\in F\colon x \sqsubset y$. 
    \item A set $D\subseteq Y$ is \emph{$\Rbf$-dominating} if $\forall\, x\in X\ \exists\, y\in D\colon x \sqsubset y$. 
\end{enumerate}

We associate two cardinal characteristics with this relational system $\Rbf$: 
\begin{align*}
    \bfrak(\Rbf) & :=\min\{|F|:\, F\subseteq X  \text{ is }\Rbf\text{-unbounded}\} \text{ the \emph{unbounding number of $\Rbf$}, and}\\
    \dfrak(\Rbf) & :=\min\{|D|:\, D\subseteq Y \text{ is } \Rbf\text{-dominating}\} \text{ the \emph{dominating number of $\Rbf$}.}
\end{align*}
%\end{definition}

As in~\cite{GKMS,CM22}, we also look at relational systems given by directed preorders. 
%
%\begin{definition}\label{examSdir}
We say that $\la S,\leq_S\ra$ is a \emph{directed preorder} if it is a preorder (i.e.\ $\leq_S$ is a reflexive and transitive relation on $S$) such that 
\[\forall\, x, y\in S\ \exists\, z\in S\colon x\leq_S z\text{ and }y\leq_S z.\] 
A directed preorder $\la S,\leq_S\ra$ is seen as the relational system $S=\la S, S,\leq_S\ra$, and their associated cardinal characteristics are denoted by $\bfrak(S)$ and $\dfrak(S)$. The cardinal $\dfrak(S)$ is actually the \emph{cofinality of $S$}, typically denoted by $\cof(S)$ or $\cf(S)$.
%\end{definition}

The typical example is $\omega^\omega$ with the relation $x\leq^* y$ iff $x(n)\leq y(n)$ for all but finitely many $n<\omega$. Its associated cardinal characteristics are the classical \emph{bounding number} $\bfrak$ and the \emph{dominating number} $\dfrak$.

%\begin{fact}\label{basicdir}
%If a directed preorder $S$ has no maximum element then $\bfrak(S)$ is infinite and regular, and $\bfrak(S)\leq\cf(\dfrak(S))\leq\dfrak(S)\leq|S|$. Even more, if $L$ is a linear order without maximum then $\bfrak(L)=\dfrak(L)=\cof(L)$.
%\end{fact}

The cardinal characteristics associated with an ideal can be characterized by relational systems as well. 
%
%\begin{example}\label{exm:Iwf}
For $\Iwf\subseteq\pts(X)$, define the relational systems: %(containing $[X]^{<\aleph_0}$).
\begin{enumerate}[label=(\arabic*)]
    \item $\Iwf:=\la\Iwf,\Iwf,\subseteq\ra$, which is a directed preorder when $\Iwf$ is closed under unions (e.g.\ an ideal).
    
    \item $\Cbf_\Iwf:=\la X,\Iwf,\in\ra$.
\end{enumerate}
%\end{example}

\begin{fact}
If $\Iwf$ is an ideal on $X$ containing $[X]^{<\aleph_0}$, then
\begin{multicols}{2}
\begin{enumerate}[label= \rm (\alph*)]
    \item $\bfrak(\Iwf)=\add(\Iwf)$. %:=\min\set{|\Jwf|}{\Jwf\subseteq\Iwf,\,\bigcup\Jwf\notin\Iwf}$. 
    
    \item $\dfrak(\Iwf)=\cof(\Iwf)$.%:=\min\set{|\Jwf|}{\Jwf\subseteq\Iwf,\ \forall\, A\in\Iwf\ \exists\, B\in \Jwf\colon A\subseteq B}$.
    
    \item $\dfrak(\Cbf_\Iwf)=\cov(\Iwf)$. %:=\min\set{|\Jwf|}{\Jwf\subseteq\Iwf,\,\bigcup\Jwf=X}$. 
    
    \item $\bfrak(\Cbf_\Iwf)=\non(\Iwf)$.%:=\min\set{|A|}{A\subseteq X,\,A\notin\Iwf}$. 
\end{enumerate}
\end{multicols}
\end{fact}

% \begin{example}\label{ex:ideal<theta}
% Let $\theta$ be an infinite cardinal and $X$ a set of size $\geq\theta$. Then $[X]^{<\theta}$ is an ideal. 
% Its additivity and uniformity numbers are easy to determine:
% \[\add([X]^{<\theta})=\cf(\theta) \text{ and } \non([X]^{<\theta})=\theta.\]

% For the covering number, we obtain
%     \[\cov([X]^{<\theta})=\left\{
%     \begin{array}{ll}
%         |X| & \text{if $|X|>\theta$,} \\
%         \cf(\theta) & \text{if $|X|=\theta$.}
%     \end{array}\right.\]
% Therefore $\cov([X]^{<\theta})=|X|$ whenever $\theta$ is regular, which is our case of interest.

% The cofinality number is more interesting. Our case of interest is $|X|^{<\theta}=|X|$, for which we have $\cof([X]^{<\theta})=|X|$. More generally, when $\theta$ is regular, $\cof([\theta]^{<\theta})=\theta$, otherwise $\theta<\cof([\theta]^{<\theta})\leq\theta^{<\theta}$; if $\kappa\geq\theta$ is a cardinal, then $\cof([\kappa^+]^{<\theta})=\kappa^+\cdot\cof([\kappa]^{<\theta})$; and, whenever $\lambda>\theta$ is a limit cardinal, $\cof([\lambda]^{<\theta})=\sup\{\cof([\mu]^{<\theta}):\, \theta\leq\mu<\lambda\}$ if $\cf(\lambda)\geq\theta$, otherwise $\lambda<\cof([\lambda]^{<\theta})\leq\lambda^{<\theta}$.

% Under Shelah's Strong Hypothesis\footnote{The failure of this hypothesis requires large cardinals.} it follows that
% \[\cof([X]^{<\theta})=\left\{
%     \begin{array}{ll}
%         |X| & \text{if $\cf(|X|)\geq\theta$,} \\
%         |X|^+ & \text{otherwise.}
%     \end{array}\right.\]
% \end{example}

Inequalities between cardinal characteristics associated with relational systems can be determined by the dual of a relational system and also via Tukey connections. %, which we introduce below.
%
%\begin{definition}\label{def:dual}
Fix a relational system $\Rbf=\la X,Y,\sqsubset\ra$. The \emph{dual of $\Rbf$} is the relational system $\Rbf^\perp:=\la Y,X,\sqsubset^\perp\ra$ where $y \sqsubset^\perp x$ iff $\neg(x \sqsubset y)$. 
%\end{definition}
%
Note that $\dfrak(\Rbf^\perp)=\bfrak(\Rbf)$ and $\bfrak(\Rbf^\perp)=\dfrak(\Rbf)$.

%\begin{definition}\label{def:Tukey}
Let %$\Rbf=\la X,Y,\sqsubset\ra$ and 
$\Rbf'=\la X',Y',\sqsubset'\ra$ be another relational system. We say that $(\Psi_-,\Psi_+)\colon\Rbf\to\Rbf'$ is a \emph{Tukey connection from $\Rbf$ into $\Rbf'$} if 
 $\Psi_-\colon X\to X'$ and $\Psi_+\colon Y'\to Y$ are functions such that  \[\forall\, x\in X\ \forall\, y'\in Y'\colon \Psi_-(x) \sqsubset' y' \Rightarrow x \sqsubset \Psi_+(y').\]
The \emph{Tukey order} between relational systems is defined by
$\Rbf\leqT\Rbf'$ iff there is a Tukey connection from $\Rbf$ into $\Rbf'$. \emph{Tukey equivalence} is defined by $\Rbf\eqT\Rbf'$ iff $\Rbf\leqT\Rbf'$ and $\Rbf'\leqT\Rbf$. 
%\end{definition}

Recall that $\Rbf\leqT\Rbf'$ implies $(\Rbf')^\perp\leqT\Rbf^\perp$, $\dfrak(\Rbf)\leq\dfrak(\Rbf')$ and $\bfrak(\Rbf')\leq\bfrak(\Rbf)$. 

% \begin{example}
%  The following Tukey connections are easy to obtain.
    
%     \begin{enumerate}[label=(\arabic*)]
%         \item For any ideal $\Iwf$ on $X$, if $[X]^{<\aleph_0}$ then $\Cbf_\Iwf\leqT\Iwf$ and $\Cbf^\perp_\Iwf\leqT\Iwf$. These determine some of the inequalities in \autoref{diag:idealI}.

%         \begin{figure}[ht]
%   \begin{center}
%     \includegraphics[scale=1.0]{diagaddetc2.pdf}
%     \caption{Diagram of the cardinal characteristics associated with an ideal $\Iwf$.}
%     \label{diag:idealI}
%   \end{center}
% \end{figure}
    
%     \item If $\theta'\leq\theta$ are infinite cardinals and $\theta\leq|X|\leq|X'|$, then $\Cbf_{[X]^{<\theta}}\leqT \Cbf_{[X']^{<\theta'}}$ and $[X]^{<\theta}\leqT[X']^{<\theta}$.
    
%     % \item For any cardinal $\mu$, $\Cbf_{[\mu]^{<\mu}}\leqT\mu \leqT \Cbf^\perp_{[\mu]^{<\mu}} \leqT[\mu]^{<\mu}$. In the case when $\mu$ is regular, $[\mu]^{<\mu}\leqT \Cbf_{[\mu]^{<\mu}}$ so the relational systems $\mu$, $[\mu]^{<\mu}$ and $\Cbf_{[\mu]^{<\mu}}$ are Tukey equivalent, and $\add([\mu]^{<\mu})=\cof([\mu]^{<\mu})=\mu$.
    
%     \end{enumerate}
% \end{example}

In our forcing applications, we show that some cardinal characteristics have certain values (in a generic extension) by forcing a Tukey connection between their relational systems and some simple relational systems like $\Cbf_{[\lambda]^{<\theta}}$ and $[\lambda]^{<\theta}$ for some cardinals $\theta\leq\lambda$ with $\theta$ uncountable regular. 
 
For instance, if $\Rbf$ is a relational system and we force $\Rbf\eqT\Cbf_{[\lambda]^{<\theta}}$, then we obtain $\bfrak(\Rbf)=\non([\lambda]^{<\theta})=\theta$ and $\dfrak(\Rbf)=\cov([\lambda]^{<\theta})=\lambda$, the latter when either $\theta$ is regular or $\lambda>\theta$. %In the case when $\lambda^{<\theta}=\lambda$, we obtain the same values when forcing $\Rbf\eqT[\lambda]^{<\theta}$, also because of the following result.
%
%
% \begin{lemma}[{\cite[Fact~3.8]{CM23}}]\label{lem:thetaid}
% Let $\theta$ be an infinite cardinal and $X$ a set of size ${\geq}\theta$. 
% Then $\Cbf_{[X]^{<\theta}}\eqT [X]^{<\theta}$ iff $\cof([X]^{<\theta})=|X|$ and $\theta$ is regular.
% \end{lemma}
%
%
This discussion motivates the following characterizations of the Tukey order between $\Cbf_{[X]^{<\theta}}$ and other relational systems.

\begin{lemma}[{\cite[Lemma~1.16]{CM22}}]\label{lem:TukeyCtheta}
Let $\theta$ be an infinite cardinal, $I$ a set of size ${\geq}\theta$ and let $\Rbf=\la X,Y,\sqsubset\ra$ be a relational system. Then: 
\begin{enumerate}[label=\rm(\alph*)]
    \item\label{Ctheta:a} $\Rbf\leqT\Cbf_{[X]^{<\bfrak(\Rbf)}}$. In particular, if $|X|\geq\theta$ then 
    $\Rbf\leqT\Cbf_{[X]^{<\theta}}$ iff $\bfrak(\Rbf)\geq\theta$.
    %$\forall\, A \in[X]^{<\theta}\ \exists\, y_A\in Y\ \forall\, x\in A\colon x\sqsubset y_A$, i.e.\ any subset of $X$ of size ${<}\theta$ is $\Rbf$-bounded. 
    
    \item If $\bfrak(\Rbf) = |X|$ then $\dfrak(\Rbf) \leq \cf(|X|)$.
    %Assume $\bfrak(\Rbf)\geq\theta$. %If $|X|>\theta$ then $\dfrak(\Rbf)\leq|X|$. Otherwise, 
    %If $|X|=\theta$ then $\dfrak(\Rbf)\leq\cf(\theta)$.
    
    \item\label{c:str} $\Cbf_{[I]^{<\theta}}\leqT\Rbf$ iff $\exists\,\{ x_i:\, i\in I\} \subseteq X\ \forall\, y\in Y\colon |\{i\in I:\, x_i\sqsubset y\}|<\theta$. In this case:
    
    \begin{enumerate}[label=\rm (\roman*),topsep=1ex]
        \item $\bfrak(\Rbf)\leq\theta$;
        \item $|I|\leq\dfrak(\Rbf)$ whenever $|I|>\theta$, otherwise $\cf(\theta)\leq\dfrak(\Rbf)$.
    \end{enumerate}    
\end{enumerate}
\end{lemma}

A family $\{ x_i:\, i\in I\}$ as in~\ref{c:str} has a special role for forcing different values to cardinal characteristics. They are also called \emph{strong witnesses of $\bfrak(\Rbf)$} (as in~\cite{GKScicmax,GKMS}).

\begin{definition}\label{def:strwit}
    Let $\Rbf=\la X,Y,\sqsubset\ra$ be a relational system, and let $\theta$ be a cardinal. We say that a family $\{ x_i:\, i\in I\} \subseteq X$ is \emph{(strongly) $\theta$-$\Rbf$-unbounded}\footnote{``Strongly $\theta$-$\Rbf$-unbounded" and ``$\theta$-$\Rbf$-unbounded" are different notions in~\cite{CM}, but since the latter is not used in this paper, we use ``$\theta$-$\Rbf$-unbounded" for ``strongly $\theta$-$\Rbf$-unbounded".} if $|I|\geq\theta$ and, for any $y\in Y$, $|\set{i\in I}{x_i \sqsubset y}|<\theta$. 
\end{definition}

So \autoref{lem:TukeyCtheta}~\ref{c:str} states that $\Cbf_{[I]^{<\theta}} \leqT \Rbf$ iff there is some $\theta$-$\Rbf$-unbounded family indexed by $I$.

We finish this subsection with the following special types of objects, usually added as generic objects of a forcing notion.

\begin{definition}\label{def:domreal}
Let $\Rbf=\la X,Y,\sqsubset\ra$ be a relational system and let $M$ be a set. %(commonly a model).
\begin{enumerate}[label=(\arabic*)]
    \item Say that $y\in Y$ is \emph{$\Rbf$-dominating over $M$} if $\forall\, x \in  X\cap M\colon x \sqsubset y$.
    \item Say that $x$ is \emph{$\Rbf$-unbounded over $M$} if $\forall\, y \in  Y\cap M\colon \neg(x \sqsubset y)$, i.e.\ it is $\Rbf^\perp$-dominating over $M$.
\end{enumerate}
\end{definition}

\subsection{Measure and category}
\ 

Given a sequence $b = \Seq{b(n)}{n\in \omega}$ of non-empty sets and an $h\colon \omega\to\omega$ define: 
\begin{align*}
\prod b&:=\prod_{n<\omega}b(n),\\
\Scal(b,h)&:=\prod_{n<\omega} [b(n)]^{\leq h(n)},\\
\seq_{<\omega}(b)&:=\bigcup_{n<\omega}\prod_{k<n}b(k).
\end{align*}
For each $s\in\seq_{<\omega}(b)$, define
\[[s]:=[s]_b:=\set{x\in\prod b}{ s\subseteq x}.\]

As a topological space, $\prod b$ has the product topology with each $b(n)$ endowed with the discrete topology. Note that $\set{[s]_b}{s\in \Seqw b}$ forms a base of clopen sets for this topology. When each $b(n)$ is countable we have that $\prod b$ is a Polish space, and $\prod b$ is perfect iff $|b(n)|\geq 2$ for infinitely many $n$. 

Recall that $T\subseteq\sqw(b)$ is a \emph{tree} if $\forall\, t\in T\ \forall\, s\in\sqw(b)\colon s\subseteq t \imp s\in T$. Denote by $[T]:=\{x\in\prod b:\, \forall\, n<\omega \colon x{\upharpoonright}n\in T\}$ the set of infinite branches of $T$. Note that $F\subseteq\prod b$ is closed iff $F=[T]$ for some (well-pruned) tree $T\subseteq\sqw(b)$.

%We now review the Lebesgue measure on $\prod b$ when each $b(n)$ is either a finite set or $\omega$. For any ordinal $\eta$ which is either finite or $\omega$, the probability measure $\mu_\eta$ on the power set of $\eta$ is defined by:
% \begin{itemize}
% \item when $\eta$ is finite, $\mu_{\eta}$ is the measure such that, for all $k\in \eta$,
% $\mu_{\eta}(\{k\})=\frac{1}{|\eta|}$, and 
% \item when $\eta=\omega$, $\mu_{\omega}$ is the measure such that, for  $k<\omega$,
% $\mu_{\omega}(\{k\})=\frac{1}{2^{k+1}}$.
% \end{itemize}
% Denote by $\Lb_b$ the product measure of $\la \mu_{b(n)}:\, n<\omega\ra$, which we call \emph{the Lebesgue measure on $\prod b$}, so $\Lb_b$ is a  probability measure on the Borel $\sigma$-algebra of $\prod b$. More concretely, $\Lb_{b}$ 
% is the unique measure on the Borel $\sigma$-algebra such that, for any $s\in \Seqw b$, $\Lb_b([s])=\prod_{i<|s|}\mu_{b(i)}(\{s(i)\})$. In particular, denote by $\Lb$, $\Lb_{2}$ and $\Lb_{\omega}$ the Lebesgue measure on $\R$, on $2^\omega$, and on $\omega^\omega$, respectively. 

\begin{remark}\label{ex:tukeysmall}
 %\Diego{ser\'a que quitamos este remark?}  
 Let $X$ be a Polish space and $\Bwf(X)$ the $\sigma$-algebra of Borel subsets of $X$.
    \begin{enumerate}[label= \rm (\arabic*)]

    \item If $X$ is perfect then $\Mwf(X)\eqT \Mwf(\R)$ and $\Cbf_{\Mwf(X)}\eqT \Cbf_{\Mwf(\R)}$, where $\Mwf(X)$ denotes the ideal of meager subsets of $X$ (see~\cite[Ex.~8.32 \&~Thm.~15.10]{Ke2}). Therefore, the cardinal characteristics associated with the meager ideal are independent of the perfect Polish space used to calculate it. When the space is clear from the context, we write $\Mwf$ for the meager ideal.
    
    \item Assume that $\mu\colon\Bwf(X)\to [0,\infty]$ is a $\sigma$-finite measure such that $\mu(X)>0$ and every singleton has measure zero.
    Denote by $\Nwf(\mu)$ the ideal generated by the $\mu$-measure zero sets, which is also denoted by $\Nwf(X)$ when the measure on $X$ is clear.
    Then $\Nwf(\mu)\eqT \Nwf(\Lb)$ and $\Cbf_{\Nwf(\mu)}\eqT \Cbf_{\Nwf(\Lb)}$ where $\Lb$ denotes the Lebesgue measure on $\R$ (see~\cite[Thm.~17.41]{Ke2}). Therefore, the four cardinal characteristics associated with both measure zero ideals are the same. 
    %
    %When $b=\la b(n):\, n<\omega\ra$, each $b(n)$ is either finite or $\omega$, and $\prod b$ is perfect, we have that $\Lb_b$ satisfies the properties of $\mu$ above. 
    %
    When the measure space is understood, we just write $\Nwf$ for the null ideal.
\end{enumerate}
\end{remark}

\subsection{Strong measure zero sets}
\ 

For combinatorial purposes, we use the notion of strong measure zero in $\prod b$ when each $b(n)$ is finite and $\prod b$ is perfect.

\begin{definition}\label{defsmz}
For $\sigma \in (\seq_{<\omega}(b))^\omega$, define the \emph{height of $\sigma$}, $\hgt_\sigma\colon \omega\to\omega$, by $\hgt_\sigma(n):=|\sigma(n)|$ for all $n$.

A set $X\subseteq \prod b$ has \emph{strong measure zero} (in $\prod b$) if
\[\forall\, f\in\omega^\omega\ \exists\, \sigma\in(\seq_{<\omega}(b))^\omega\colon f\leq \hgt_\sigma \text{ and }X\subseteq\bigcup_{i<\omega}[\sigma(i)].\]
Denote by $\SNwf(\prod b)$ the collection of strong measure zero subsets of $\prod b$.
\end{definition}

We have that $\SNwf(\prod b) \eqT\SNwf([0,1]) \eqT\SNwf(\R)$ and $\Cbf_{\SNwf(\prod b)} \eqT \Cbf_{\SNwf(\R)} \eqT \Cbf_{\SNwf([0,1])}$, see details in~\cite{CMR}. Therefore, the cardinal characteristics associated with the strong measure zero ideal do not depend on such a space.

Denote
\[[\sigma]_\infty:=[\sigma]_{b,\infty}=\set{x \in \prod b}{\exists^{\infty}\, n < \omega\colon \sigma(n) \subseteq x}=\bigcap_{n<\omega} \bigcup_{k \geqslant n}[\sigma(k)].\]
The following characterization of $\SNwf(\prod b)$, in terms of the sets above, is quite practical.

\begin{lemma}\label{charSN}
    Let $X\subseteq\prod b$ and let $D\subseteq\omega^\omega$ be a dominating family. Then $X\subseteq\prod b$ has strong measure zero in $\prod b$ iff  
    \[\forall\ f\in D\ \exists\sigma\in (\seq_{<\omega}(b))^\omega\colon f\leq^*\hgt_\sigma\textrm{\ and\ } X\subseteq[\sigma]_\infty.\]
\end{lemma}

From now on, we work with $\SNwf:=\SNwf(2^\omega)$. 
%Now we introduce Yorioka ideals, which were introduced to study the cofinality of $\SNwf$. This seems to be inspired by the previous lemma.

\begin{definition}[Yorioka~{\cite{Yorioka}}]\label{DefYorio}
Set $\omega^{\uparrow\omega}:=\{d\in\omega^{\omega}:\, d\text{\ is increasing}\}$.
For $f\in\omega^{\uparrow\omega}$, define the \emph{Yorioka ideal}
\[\Iwf_f:=\{A\subseteq 2^\omega:\, \exists\, \sigma\in(2^{<\omega})^\omega\colon f\ll \hgt_\sigma\text{\ and }A\subseteq[\sigma]_\infty\},\]
where the relation $x \ll y$  denotes $\forall\, k<\omega\ \exists\, m_k<\omega\ \forall\, i\geq m_k\colon x(i^k)\leq y(i)$.
\end{definition}

The reason why it is used $\ll$ instead of $\leq^*$ is that the latter would not yield an ideal, as proved by Kamo and Osuga~\cite{KO08}.
Yorioka~\cite{Yorioka} has proved (indeed) that $\Iwf_f$ is a $\sigma$-ideal when $f$ is increasing. By~\autoref{charSN} it is clear that $\SNwf=\bigcap\{\Iwf_f:\, f\in \omega^{\uparrow\omega}\}$ and $\Iwf_f \subseteq \Nwf$ for any $f\in\baireincr$.

%Concerning the cardinal characteristics associated with Yorioka ideals, it is clear that $f\leq^* g$ implies $\Iwf_g\subseteq\Iwf_f$, so $\Cbf_{\Iwf_f}\leqT \Cbf_{\Iwf_g}$, which implies $\cof(\Iwf_f) \leq \cof(\Iwf_g)$ and $\non(\Iwf_g) \leq \non(\Iwf_f)$. Although it is unclear whether similar claims hold for additivities and cofinalities, it is known that $\add(\Nwf) \leq \add(\Iwf_f) \leq \add(\Iwf_\id)$ and $\cof(\Iwf_\id) \leq \cof(\Iwf_f) \leq \cof(\Nwf)$, see \cite[Subsec.~3.2]{CM}.

Denote
\begin{align*}
    \minadd & :=\min\set{\add(\Iwf_f)}{f\in\omega^{\uparrow\omega}}, &
    \supcov & :=\sup\set{\cov(\Iwf_f)}{ f\in\omega^{\uparrow\omega}},\\
    \minnon & :=\min\set{\non(\Iwf_f)}{ f\in\omega^{\uparrow\omega}}, &
    \supcof & :=\sup\set{\cof(\Iwf_f)}{ f\in\omega^{\uparrow\omega}}.
\end{align*}
It is known from Miller that $\minnon = \non(\SNwf)$.
On the other hand, $\add(\Iwf_\id)$ is the largest of the additivities of Yorioka ideals and $\cof(\Iwf_\id)$ is the smallest of the cofinalities, similarly for $\non(\Iwf_\id)$ and $\cof(\Iwf_\id)$, where $\id$ denotes the identity function on $\omega$. See~\cite[Sec.~3]{CM} for a summary of the cardinal characteristics associated with the Yorioka ideals.

We present some results from~\cite{CM23} that help calculate $\cof(\SNwf)$. To do this, we first review products and powers of relational systems, as developed in~\cite[Sec.~3]{CM23}.

\begin{definition}\label{def:idpow}
Let $w$ be a set and
let $\Rbf_i = \la X_i,Y_i,\sqsubset_i\ra$ be a relational system for any $i\in w$. 
Define the \emph{product relational system}
\[\prod_{i\in w} \Rbf_i = \Big\la\prod_{i\in w}X_i,\prod_{i\in w}Y_i,\sqsubset\Big\ra,\]
where $x \sqsubset y$ iff $x_i \sqsubset_i y_i$ for all $i\in w$. 
When all $\Rbf_i$ are the same $\Rbf$, we denote the product by $\Rbf^w$ and call it a 
\emph{power of $\Rbf$.}

Given an ideal $\Iwf$ on $X$, define $\Iwf^{(w)}$ as the ideal on $X^w$ generated by the sets of the form $\prod_{i\in w}A_i$ with $\Seq{A_i}{i\in w}\in\Iwf^{w}$. Denote $\add(\Iwf^w):=\bfrak(\Iwf^{(w)})$, $\cof(\Iwf^w):=\dfrak(\Iwf^{(w)})$, $\non(\Iwf^w):=\bfrak(\Cbf_{\Iwf^{(w)}})$ and $\cov(\Iwf^w):=\dfrak(\Cbf_{\Iwf^{(w)}})$.
\end{definition}

It is not hard to show that 
\[\bfrak\left(\prod_{i\in w} \Rbf_i\right) = \min_{i\in w}\bfrak(\Rbf_i) \text{ and } \sup_{i\in w}\dfrak(\Rbf_i)\leq \dfrak\left(\prod_{i\in w} \Rbf_i\right) \leq \prod_{i\in w} \dfrak(\Rbf_i).\]
As a consequence:
\begin{fact}\label{fct:idpow}
Let $w$ be a set and let $\Iwf$ be an ideal on $X$. Then:
\begin{enumerate}[label=\rm (\alph*)]
    \item $\Iwf^w\eqT\Iwf^{(w)}$. 
    
    \item $\Cbf_\Iwf^w\eqT\Cbf_{\Iwf^{(w)}}$.
    
    \item $\add(\Iwf^w)=\add(\Iwf)$ and $\non(\Iwf^w)=\non(\Iwf)$.
    
    \item $\cov(\Iwf)\leq\cov(\Iwf^w)\leq \cov(\Iwf)^{|w|}$ and $\cof(\Iwf)\leq\cof(\Iwf^w)\leq \cof(\Iwf)^{|w|}$.
\end{enumerate}
\end{fact}

\begin{theorem}[{\cite[Thm.~4.8]{CM23}}]\label{new_upperb}
$\SNwf\leqT\Cbf_{[\supcof]^{<\minadd}}^\dfrak$. In particular, 
\[\cof(\SNwf)\leq\cov\left(\left([\supcof]^{<\minadd}\right)^\dfrak\right).\]
\end{theorem}

For the lower bound of $\cof(\SNwf)$, we use the directed preorder $\la \lambda^\lambda,\leq\ra$ for a (regular) cardinal $\lambda$, where $\leq$ denotes poitwise inequality. Here $\bfrak(\lambda^\lambda,\leq)=\cf(\lambda)$ and $\dfrak(\lambda^\lambda,\leq)=\dfrak_\lambda$ is the typical dominating number of $\lambda^\lambda$.

\begin{theorem}[{\cite[Cor.~4.25]{CM23}}]\label{cor:lowSN}
Assume $\cov(\Mwf)=\dfrak$ and $\non(\SNwf)=\supcof=\mu$. If $\lambda:=\cf(\dfrak)=\cf(\mu)$ then $\la \lambda^\lambda,\leq\ra\leqT\SNwf$, in particular, $\add(\SNwf)\leq \lambda$ and $\dfrak_\lambda\leq \cof(\SNwf)$. Moreover, $\dfrak_\lambda\neq\mu$ and $\mu<\cof(\SNwf)$.
\end{theorem}

\subsection{Forcing}\label{Subsforcing} 
\ 

We review the elements and notation of forcing theory we use throughout this text.
When $\Por_\pi=\la\Por_\xi,\Qnm_\xi:\, \xi<\pi\ra$ is a FS (finite support) iteration, we denote $V_\pi:=V[G]$ when $G$ is $\Por_\pi$-generic over $V$ and, for $\xi<\pi$, denote $V_\xi:=V[\Por_\xi\cap G]$.

Concerning FS iterations, we recall the following useful result to force statements of the form $\Rbf\leqT\Cbf_{[X]^{<\theta}}$ for the following type of relational systems.

\begin{definition}\label{def:defrel}
We say that $\Rbf=\la X, Y, \sqsubset\ra$ is a \textit{definable relational system of the reals} if both $X$ and $Y$ are non-empty and analytic in Polish spaces $Z$ and $W$, respectively, and $\sqsubset$ is analytic in $Z\times W$.\footnote{In general, we need that $X$, $Y$ and $\sqsubset$ are definable and that the statements ``$x\in X$", ``$y\in Y$" and ``$x\sqsubset y$" are absolute for the arguments we need to carry on.} The interpretation of $\Rbf$ in any model corresponds to the interpretation of $X$, $Y$ and $\sqsubset$.
\end{definition}

\begin{theorem}[{\cite[Thm.~2.12]{CM22}}]\label{itsmallsets}
Let $\Rbf=\la X,Y,\sqsubset\ra$ be a definable relational system of the reals, $\theta$ an uncountable regular cardinal, and let  $\Por_\pi=\la\Por_\xi,\Qnm_\xi:\,  \xi<\pi\ra$ be a FS iteration of $\theta$-cc posets with $\cf(\pi)\geq\theta$. Assume that, for all $\xi<\pi$ and any $A\in[X]^{<\theta}\cap V_\xi$, there is some $\eta\geq\xi$ such that $\Qnm_\eta$ adds an $\Rbf$-dominating real over $A$. Then $\Por_\pi$ forces $\Rbf\leqT\Cbf_{[X]^{<\theta}}$, i.e.\ $\theta\leq\bfrak(\Rbf)$. %and $\dfrak(\Rbf)\leq|X|$.
\end{theorem}

To force $\cov(\Nwf)<\cov(\SNwf)$ in our main results, we use forcing notions that modify the following relational system.

\begin{definition}
 Let $b$ be a function with domain $\omega$ such that $b(i)\neq\emptyset$ for all $i<\omega$, and let $h\in\omega^\omega$. Denote $\aLc(b,h):=\la\Scal(b,h),\prod b,\not\ni^\infty\ra$ the relational system where $x\in^\infty\varphi$ iff $\exists^\infty\, n<\omega\colon x(n)\in \varphi(n)$  ($\aLc$ stands for \emph{anti-localization}). The  \emph{anti-localization cardinals} are defined by $\balc_{b,h}:=\bfrak(\aLc(b,h))$ and $\dalc_{b,h}:=\dfrak(\aLc(b,h))$.
\end{definition}

It is known that, whenever $h\geq^* 1$, $\balc_{\omega,h}=\non(\Mwf)$ and $\dalc_{\omega,h}=\cov(\Mwf)$ (see \autoref{ExmPrs}~\ref{ExmPrsa}), where $\omega$ denotes the constant sequence $\omega$.
We now introduce a forcing to increase $\balc_{b,h}$, which typically adds \emph{eventually different reals}. We present a modification of a poset from Kamo and Osuga~\cite{KO} (see~\cite{CM, BCM}).

\begin{definition}[ED forcing]\label{DefEDforcing}
  Let $b=\la b(n):\, n<\omega\ra$ be a sequence of non-empty sets and $h\in\omega^\omega$ such that $\lim_{i\to\infty}\frac{h(i)}{|b(i)|}$ \sloppy $=0$ (when $b(i)$ is infinite, interpret $\frac{h(i)}{|b(i)|}$ as $0$). Define the \emph{$(b,h)$-ED (eventually different real) forcing} $\Eor^h_b$ as the poset whose conditions are of the form $p=(s^p,\varphi^p)$ such that, for some $m:=m_{p}<\omega$,  
         \begin{enumerate}[label= \rm (\roman*)]
            \item $s^p\in\seq_{<\omega}(b)$, $\varphi^p\in\Swf(b,m h)$, and
            \item $m h(i)<b(i)$ for every $i\geq|s^p|$,
         \end{enumerate}
      ordered by $(t,\psi)\leq(s,\varphi)$ iff $s\subseteq t$, $\forall\, i<\omega\colon \varphi(i)\subseteq\psi(i)$, and $t(i)\notin\varphi(i)$ for all $i\in|t|\menos|s|$.
      
      When $G$ is $\Eor_{b}^{h}$-generic over $V$, we denote the generic real by $e:=\bigcup_{p\in G}s^p$, which we usually refer to as the \emph{eventually different generic real (over $V$)}.
      
      Denote $\Eor_{b}:=\Eor_b^1$ and $\Eor:=\Eor_\omega$. For $p\in\Eor_b^h$ define \[C(p):=\set{x\in\prod b}{s^p\subseteq x\text{ and }\forall\,  i\geq|s^p|\colon x(i)\notin\varphi^p(i)},\]
      which is a closed subset of $\prod b$.
\end{definition}
 
% Given $p\in\Eor_b^h$, notice that 
% \begin{equation*}
%     \Lb_b(C(p))\geq\Lb_b([s^p])\cdot\prod_{i\geq|s^p|}\bigg(1-\frac{m^p h(i)}{b(i)}\bigg)
% \end{equation*}

When each $b(n)$ is finite and $\prod b$ is perfect, we consider \emph{the Lebesgue measure on $\prod b$}, denoted by $\Lb_b$, as the (completion) of the product measure of the uniform probability measure of $b(n)$ for all $n<\omega$. In this case, 
the set $C(p)$ has positive measure under the condition 
$\sum_{i\in\omega}\frac{h(i)}{|b(i)|}<\infty$.

\begin{lemma}\label{Ebhdense}
Let $b=\la b(n):\, n<\omega\ra$ be a sequence of finite sets such that $\prod b$ is perfect, and $h\in\omega^\omega$ satisfying $\sum_{i\in\omega}\frac{h(i)}{|b(i)|}<\infty$. Then:
\begin{enumerate}[label= \rm (\alph*)]
    \item\label{it:densebh} The set $D_b^h$ defined below is a dense subset of $\Eor_b^h$:
    \[D_b^h:=\largeset{p\in\Eor_b^h}{\sum_{i\geq|s^p|}\frac{|\varphi^p(i)|}{|b(i)|}<1}.\]
    \item\label{it:posmeas} $\Lb_b(C(p))>0$ for any $p\in \Eor_b^h$.
\end{enumerate}
\end{lemma}
\begin{proof}
Item~\ref{it:densebh} follows directly from the assumption $\sum_{i\in\omega}\frac{h(i)}{|b(i)|}<\infty$. We now show~\ref{it:posmeas}. If $p\in \Eor^h_b$ then there is some $q\in D^h_b$ stronger than $p$ by~\ref{it:densebh}, so
\[\Lb_b\left([s^q]\menos C(q)\right)\leq \Lb_b([s^q])\sum_{i\geq|s^q|}\frac{|\varphi^q(i)|}{|b(i)|}<\Lb_b([s^q]),\]
i.e.\ $\Lb_b(C(q))>0$. On the other hand, $q\leq p$ implies $C(q)\subseteq C(p)$, so $\Lb_b(C(p))>0$.% ($C(p)$ is measurable because it is closed in $\prod p$).
\end{proof}

Motivated by~\cite[Lemma~1.19]{KST}, we have the following lemma for $\Eor_b^h$. Here $[X]_\Nwf$ denotes the equivalence class of the equivalence relation $\sim_\Nwf$ defined by $X\sim_\Nwf Y$ iff $X\symdf Y \in \Nwf$ (where $\symdf$ denotes the symmetric difference).

\begin{lemma}\label{EDsepfam}
Let $b,h\in\omega^\omega$ as in \autoref{Ebhdense} and define  $\Qor^h_b:=\largeset{[C(p)]_\Nwf}{p\in \Eor_b^h}$, which is a subposet of $\Bwf(\prod b)/\Nwf(\prod b)$ (random forcing). Denote by $\Bor^h_b$ the subalgebra of $\Bwf(\prod b)/\Nwf(\prod b)$ generated by $\Qor^h_b$. Then:
\begin{enumerate}[label= \rm (\alph*)]
    \item\label{it:Qbh} $\Qor^h_b$ is dense in $\Bor^h_b$ and forcing equivalent with $\Eor^h_b$.
    \item For $s\in\seq_{<\omega}(b)$, let $p^*_s:=(s,\overline{\emptyset})\in \Eor^h_b$, where 
        $\overline{\emptyset}:=\la\emptyset,\emptyset,\ldots\ra$. Then
    \vspace{-7pt}
    \begin{multicols}{2}
    \begin{enumerate}[label= \rm (\roman*)]
        \item $C(p^*_s)=[s]$ and
        \item $[C(p^*_s)]_\Nwf\in\Q^h_b$.
    \end{enumerate}
    \end{multicols}    
\end{enumerate}
\end{lemma}
\begin{proof}
We only show~\ref{it:Qbh}.
Note that, for $p,q\in \Eor^h_b$, $q\leq p$ implies $C(q)\subseteq C(p)$. Also, if $C(p)\cap C(q)\neq\emptyset$ then $p$ and $q$ are compatible in $\Eor^h_b$. Therefore, the map $p\mapsto [C(p)]_\Nwf$ is a dense embedding from $\Eor^h_b$ onto $\Qor^h_b$, so $\Eor^h_b$ is forcing equivalent with $\Qor^h_b$.

To conclude~\ref{it:Qbh}, it remains to show that $\Qor^h_b$ is dense in $\Bor^h_b$. Let $a\in\Bor^h_b$ be non-zero, so there are some finite $F\subseteq \Eor^h_b$ and some $e\colon F\to \{0,1\}$ such that $[\emptyset]_\Nwf<\left[\bigcap_{p\in F}C(p)^{e(p)}\right]_\Nwf\leq a$, where $C(p)^0:=C(p)$ and $C(p)^1=\prod b\menos C(p)$. Then there is some $x\in\bigcap_{p\in F}C(p)^{e(p)}$. Let $F_0:=\set{p\in F}{e(p)=0}$. For $p\in F\menos F_0$ choose some $i_p\in\omega$ witnessing that $x\notin C(p)$, concretely, either $i_p<|s^p|$ and $x(i_p)\neq s^p(i_p)$, or $i_p\geq|s^p|$ and $x(i_p)\in\varphi^p(i_p)$. Define $m^*:=\sum_{p\in F_0}m^p$ and choose $n^*<\omega$ larger than $\max\left(\{|s^p|:\, p\in F_0\}\cup\{i_p:\, p\in F\menos F_0\}\right)$ such that $m^*\,  h(i)<b(i)$ for all $i\geq n^*$. So set
$q:=(s^q,\varphi^q)$ where $s^q:=x\frestr n^*$  and $\varphi^q(i):=\bigcup_{p\in F_0}\varphi^p(i)$. Then $q\in\Eor^h_b$, it is stronger than all $p\in F_0$ and incompatible with all $p\in F\menos F_0$ (because of the choice of $i_p$). Therefore, $C(q)\subseteq C(p)$ for all $p\in F_0$ and $C(q)\cap C(p)=\emptyset$ for all $p\in F\menos F_0$, i.e.\ $C(q)\subseteq \bigcap_{p\in F}C(p)^{e(p)}$, so $[C(q)]_\Nwf\leq a$.
\end{proof}

We have that $\Eor$ is $\sigma$-centered, but the same does not hold for $\Eor^h_b$ when $b\in\omega^\omega$. However, the following linkedness property, due to Kamo and Osuga, is satisfied by such posets.

\begin{definition}[{\cite{KO}}]\label{link}
Let $\rho,\varrho \in \omega^\omega$. A forcing notion $\Pbb$ is \textit{$(\rho,\varrho)$-linked} if there exists a sequence $\langle Q_{n,j}:\, n<\omega,\ j<\rho(n)\rangle$ of subsets of $\Pbb$ such that
\begin{enumerate}[label= \rm (\roman*)]
\item\label{it:rhopi1} $Q_{n,j}$ is $\varrho(n)$-linked for all $n<\omega$ and $j<\rho(n)$, and 
\item\label{it:rhopi2} $\forall\, p\in \Pbb\ \forall^{\infty}\, n<\omega\ \exists\, j<\rho(n)\colon p\in Q_{n,j}$.
\end{enumerate}
\end{definition}

\begin{lemma}[{\cite[Lemma~2.21]{CM}}]\label{genlink}
   Let $b,h\in\omega^\omega$ with $b\geq 1$. Let $\varrho,\rho\in\omega^\omega$ and assume that there is a non-decreasing function $f\in\omega^\omega$ going to infinity and an $m^*<\omega$  such that, for all but finitely many $k<\omega$,
   \begin{enumerate}[label= \rm (\roman*)]
       \item $k\varrho(k)h(i)<b(i)$ for all $i\geq f(k)$ and
       \item $k\prod_{i=m^*}^{f(k)-1}((\min\{k,f(k)\}-1)h(i)+1)\leq\rho(k)$.
   \end{enumerate}
   Then $\Eor^h_b$ is $(\rho,\varrho)$-linked.
\end{lemma}

\section{Review of Preservation Theory}\label{SecPres}

We review the preservation theory of unbounded families presented in~\cite[Sect.~4]{CM}. This a generalization of Judah's and Shelah's~\cite{JS} and the first author's~\cite{Br} preservation theory.

\begin{definition}\label{def:Prs}
We say that $\Rbf=\langle X,Y,\sqsubset\rangle$ is a \textit{generalized Polish relational system (gPrs)} if
\begin{enumerate}[label=\rm(\arabic*)]
\item $X$ is a Perfect Polish space,
\item $Y=\bigcup_{e\in \Omega}Y_e$ where $\Omega$ is a non-empty set and, for some Polish space $Z$, $Y_e$ is non-empty and analytic in $Z$ for all $e\in \Omega$, and
\item \label{def:Prsc}$\sqsubset=\bigcup_{n<\omega}\sqsubset_{n}$ where $\langle\sqsubset_{n}: n<\omega\rangle$  is some increasing sequence of closed subsets of $X\times Z$ such that, for any $n<\omega$ and for any $y\in Y$,
$(\sqsubset_{n})^{y}=\set{x\in X}{x\sqsubset_{n}y }$ is closed nowhere dense.
\end{enumerate}
If $|\Omega|=1$, we just say that $\Rbf$ is a \emph{Polish relational system (Prs)}.
\end{definition}

\begin{remark}\label{Prsremark}
By~\autoref{def:Prs}~\ref{def:Prsc}, $\Cbf_{\Mcal(X)} \leqT \Rbf$. Therefore, $\bfrak(\Rbf)\leq \non(\Mcal)$ and $\cov(\Mcal)\leq\dfrak(\Rbf)$.
\end{remark}

For the rest of this section, fix a gPrs $\Rbf=\langle X,Y,\sqsubset\rangle$ and an infinite cardinal $\theta$.

\begin{definition}\label{def:good}
A poset $\Por$ is \textit{$\theta$-$\Rbf$-good} if, for any $\Por$-name $\dot{h}$ for a member of $Y$, there is a non-empty set $H\subseteq Y$ (in the ground model) of size ${<}\theta$ such that, for any $x\in X$, if $x$ is $\Rbf$-unbounded over  $H$ then $\Vdash x\nsqsubset \dot{h}$.

We say that $\Por$ is \textit{$\Rbf$-good} if it is $\aleph_1$-$\Rbf$-good.  %\red{Moreover, Say that $\Por$ is \textit{${<}\theta$-$\Rbf$-good} when it is $\theta_0$-$\Rbf$-good for all $\theta_0<\theta$.}
\end{definition}

Notice that $\theta<\theta_0$
implies that any $\theta$-$\Rbf$-good poset is $\theta_0$-$\Rbf$-good. Also, if $\Por \lessdot\Qor$ and $\Qor$ is $\theta$-$\Rbf$-good, then $\Por$ is $\theta$-$\Rbf$-good.

Good posets allow us to preserve the Tukey order as follows.

\begin{lemma}[{\cite[Lemma~4.7~(b)]{CM}}]\label{goodpres}
Assume that $\theta$ is regular uncountable and assume that $\Por$ is a $\theta$-$\Rbf$-good poset. If $\mu$ is a cardinal preserved by $\Por$, $\cf(\mu)\geq\theta$ and $|I|\geq\mu$, then $\Por$ preserves $\mu$-$\Rbf$-unbounded families indexed by $I$. In particular,
if $\Cbf_{[I]^{<\mu}}\leqT\Rbf$ (in the ground model), then $\Por$ forces that $\Cbf_{[I]^{<\mu}}\leqT\Rbf$.
\end{lemma}
% \begin{proof}
% Let $\{ x_i:\, i\in I \}$ be a $\theta$-$\Rbf$-unbounded family. We show that $\Por$ forces $|\{i\in I:\, x_i\sqsubset y\}|<\mu$ for all $y\in Y$. Let $\dot y$ be a $\Por$-name of a member of $Y$ and choose $H$ as in \autoref{def:good}. Let $B:=\bigcup_{y'\in H}\{i\in I:\, x_i\sqsubset y'\}$, so $|B|<\mu$. Since $\Por$ forces $x_i\sqsubset\dot y\Rightarrow i\in B$, then $\Por$ forces $|\{i\in I:\, x_i\sqsubset \dot y\}|<\mu$.
% \end{proof}

Small posets are automatically good.

\begin{lemma}[{\cite[Lemma~4.10]{CM}}]\label{smallgoodness}
If $\theta$ is a regular cardinal then any poset of size ${<}\theta$
is $\theta$-$\Rbf$-good. In particular, the Cohen forcing is $\Rbf$-good.
\end{lemma}

For two posets $\Por$ and $\Qor$, we write $\Por\subsetdot\Qor$ when $\Por$ is a complete suborder of $\Qor$, i.e.\ the inclusion map from $\Por$ into $\Qor$ is a complete embedding.

\begin{definition}[Direct limit]\label{def:limdir}
We say that $\la\Por_i:\, i\in S\ra$ is a \emph{directed system of posets} if $S$ is a directed preorder and, for any $j\in S$, $\Por_j$ is a poset and $\Por_i\subsetdot\Por_j$ for all $i\leq_S j$.

For such a system, we define its \emph{direct limit} $\limdir_{i\in S}\Por_i:=\bigcup_{i\in S}\Por_i$ ordered by
\[q\leq p \sii \exists\, i\in S\colon p,q\in\Por_i\text{ and }q\leq_{\Por_i} p.\]
\end{definition}

The Cohen reals added along an iteration are usually used as witnesses for Tukey connections, as they form strong witnesses. For example:

\begin{lemma}[{\cite[Lemma~4.14]{CM}}]\label{lem:strongCohen}
Let $\mu$ be a cardinal with uncountable cofinality and let $\la\Por_{\alpha}:\, \alpha<\mu\ra$ be a $\subsetdot$-increasing sequence of $\cf(\mu)$-cc posets such that  $\Por_\mu=\limdir_{\alpha<\mu}\Por_{\alpha}$. If $\Por_{\alpha+1}$ adds a Cohen real $\dot{c}_\alpha\in X$ over $V^{\Por_\alpha}$ for any $\alpha<\mu$, then $\Por_{\mu}$ forces that $\{\dot{c}_\alpha:\alpha<\mu\}$ is a $\mu$-$\Rbf$-unbounded family. In particular, $\Por_\mu$ forces that $\mu\leqT\Cbf_\Mwf \leqT \Rbf$.
\end{lemma}

\begin{example}\label{ExmPrs}
The following are Prs's that describe cardinal characteristics in Cicho\'n's diagram.
  \begin{enumerate}[label= \rm (\arabic*)]
     \item\label{ExmPrsa} Consider the Polish relational system $\Ed:=\la\omega^\omega,\omega^\omega,\neq^\infty\ra$ where $x =^\infty y$ means that $x(n)=y(n)$ for infinitely many $n$ (so $x \neq^\infty y$ means that $x$ and $y$ are eventually different). 
     By~\cite[Thm.~2.4.1 \& Thm.~2.4.7]{BJ}, $\bfrak(\Ed)=\non(\Mwf)$ and $\dfrak(\Ed)=\cov(\Mwf)$, moreover, $\Ed\eqT \aLc(\omega,h)$ for any $h\geq^* 1$ (see e.g.~\cite[Thm.~3.17]{CMlocalc}).  
     
     \item\label{ExmPrsb}  The relational system $\Dbf:=\la\omega^\omega,\omega^\omega,\leq^*\ra$ is Polish. Typical examples of $\Dbf$-good sets are $\Eor^h_b$ and random forcing. More generally, $\sigma$-$\Fr$-linked posets are $\Dbf$-good (see~\cite{mejiavert,BCM}).

     % \item\label{ExmPrsc} \Diego{Never used, it can be omitted.} Define $\Omega_n:=\set{a\in [2^{<\omega}]^{<\aleph_0}}{ \Lb_2(\bigcup_{s\in a}[s])\leq 2^{-n}}$ (endowed with the discrete topology). Put $\Omega:=\prod_{n<\omega}\Omega_n$ with the product topology, which is a perfect Polish space. For every $x\in \Omega$ denote $N_{x}:=\bigcap_{n<\omega}\bigcup_{k\geq n}\bigcup_{s\in x(k)}[s]$, which is clearly a Borel null set in $2^{\omega}$.

     % Define the Prs $\Cn:=\la \Omega, 2^\omega,  R\ra$ where $x R z$ iff $z\notin N_{x}$. Recall that any null set in $2^\omega$ is a subset of $N_{x}$ for some $x\in \Omega$, so $\Cn\eqT\Cbf_\Nwf^\perp$. Hence, $\bfrak(\Cn)=\cov(\Nwf)$ and $\dfrak(\Cn)=\non(\Nwf)$.

     % The first author~\cite{Br} showed that any $\nu$-centered poset is $\nu^+$-$\Cn$-good.

     \item\label{ExmPrsd}  %For each $k<\omega$ let $\id^k\colon \omega\to\omega$ such that $\id^k(i)=i^k$ for all $i<\omega$, and $\Hcal:=\{\id^{k+1}:\, k<\omega\}$. 
     For $\Hcal\subseteq\omega^\omega$ non-empty and countable,
     let $\Lc^*_\Hcal:=\la\omega^\omega, \Scal(\omega, \Hcal), \in^*\ra$ be the Polish relational system where \[\Scal(\omega, \Hcal):=\set{\varphi\colon \omega\to[\omega]^{<\aleph_0}}{ \exists\, h\in\Hcal\ \forall\, i<\omega\colon |\varphi(i)|\leq h(i)}\]
     and $x\in^*\varphi$ is defined by $\forall^\infty\, n<\omega\colon x(n)\in\varphi(n)$.
     
     %Denote by $\id$ the identity function on $\omega$.
     As a consequence of~\cite{BartInv}, if $\Hcal=\{\id^{k+1}:k<\omega\}$ where 
      $\id^k(i):=i^k$ for all $i<\omega$, then $\Lc^*_\Hcal\eqT\Nwf$, so $\bfrak(\Lc^*_\Hcal)=\add(\Nwf)$ and $\dfrak(\Lc^*_\Hcal)=\cof(\Nwf)$. Denote this particular case by $\Lc^*$.
      
      Any $\mu$-centered poset is $\mu^+$-$\Lc^*$-good (see~\cite{Br,JS}) so, in particular, $\sigma$-centered posets are $\Lc^*$-good. Besides,  Kamburelis~\cite{Ka} showed that any Boolean algebra with a strictly positive finitely additive measure is $\Lc^*$-good (in particular, subalgebras of random forcing). 
  \end{enumerate}
\end{example}

In \autoref{preaddSN} we show that Boolean algebras with a strictly positive finitely additive measure do not increase $\add(\SNwf)$, which strengthens Kaburelis' result cited above.

We fix the following notation, which we use as in the following example. For functions $x,y\colon D\to \omega$, we define the functions $x+y$, $x\cdot y$ and $x^y$ from $D$ to $\omega$ in the natural way, i.e.\ $(x+y)(k):= x(k)+y(k)$, $(x\cdot y)(k):= x(k)\cdot y(k)$ and $(x^y)(k):=x(k)^{y(k)}$. We expand the notation when using a constant $c\in\omega$, i.e.\ maps $x+c$, $c\cdot x$, $c^x$ and $x^c$ are defined similarly (like $\id^k$ in \autoref{ExmPrs}~\ref{ExmPrsd}). For the product, we may omit the symbol ``$\cdot$".

\begin{example}[{\cite[Ex.~4.19]{CM}}]\label{KOpre}
Kamo and Osuga~{\cite{KO}} define a gPrs with parameters $\varrho,\rho\in\omega^\omega$, which we denote by $\aLc^*(\varrho,\rho)$. For the purposes of this paper, it is just enough to review its properties. Assume that $\varrho>0$ and $\rho\geq^* 1$.
% Fix a family $\Dwf\subseteq\omega^\omega$ of size $\aleph_1$ of non-decreasing functions which satisfies
% \begin{enumerate}[label= \rm (\arabic*)]
% \item $\forall\, e \in \Dwf\colon e \leq \id$,
% \item $\forall\, e \in \Dwf\colon \lim_{n\to+\infty}e(n)=+\infty$ and $\lim_{n\to+\infty}(n-e(n))=+\infty$,
% \item $\forall\, e \in \Dwf\ \exists\, e'\in\Dwf\colon e+1\leq^{*} e'$ and
% \item $\forall\, \Dwf' \in [\Dwf]^{\leq \aleph_0}\ \exists\, e\in \Dwf\ \forall\, e'\in\Dwf'\colon e'\leq^{*}e$.
% \end{enumerate}
% For $\varrho,\rho \in \omega^\omega$ such that $\varrho>0$ and $\rho\geq^*1$, we define %$\hat{\Swf}(b,h)=\hat{\Swf}_{\Dwf}(b,h)$ by
% \[\hat{\Swf}(\varrho,\rho):=\bigcup_{e\in\Dwf}\Swf(\varrho,\rho^e)=\set{\varphi \in \prod_{n<\omega}\Pwf(\varrho(n))}{\exists\, e\in \Dwf\ \forall\, n<\omega \colon |\varphi(n)|\leq \rho(n)^{e(n)}}\]
% Now define the relation $\blacktriangleright$ %in $\Swf(\varrho,\rho^{\id})\times\hat\Swf(\varrho,\rho)$ 
% by
%  \[\psi\blacktriangleright\varphi \textrm{\ iff\ }\forall^\infty\, n<\omega\colon \psi(n)\nsupseteq\varphi(n).\]  
%  Put $\aLc^*(\varrho,\rho):=\langle \Scal(\varrho,\rho^{\id}), \hat{\Scal}(\varrho,\rho), \blacktriangleright\rangle$, which is a gPrs.
\begin{enumerate}[label = \rm (\alph*)]
    \item\label{KOa} $\aLc^*(\varrho,\rho)\leqT \aLc(\varrho,\rho^{\id})$~\cite[Lem.~4.21]{CM}.
    \item\label{KOb} If $\sum_{i<\omega}\frac{\rho(i)^i}{\varrho(i)}<\infty$ then $\aLc(\varrho,\rho^{\id})\leqT\Cbf_{\Nwf}^\perp$~\cite[Lem.~2.3]{KM21}, so $\cov(\Nwf)\leq\bfrak(\aLc^*(\varrho,\rho))$ and $\dfrak(\aLc^*(\varrho,\rho))\leq\non(\Nwf)$
    \item\label{KOc} If $\varrho\not\leq^*1$ and $\rho\geq^*1$, then any $(\rho,\varrho^{\rho^{\id}})$-linked poset is $\aLc^*(\varrho,\rho)$-good (see~\cite[Lem.~10]{KO} and~\cite[Lem.~4.23]{CM}). %Here, given two functions $f, g$ the function $f^g\in \omega^{\omega}$ means $(f^g)(n)=f(n)^{g(n)}$ for all $n<\omega$. 
    \item\label{KOd} Any $\theta$-centered poset is $\theta^+$-$\aLc^*(\varrho,\rho)$-good~\cite[Lem.~4.24]{CM}.
\end{enumerate}
\end{example}

We close this section with the following preservation results for the covering of $\SNwf$. This was originally introduced by Pawlikowski~\cite{P90} and was generalized and improved by the second and third authors~\cite{CM23}. Here, we use the notion of the \emph{segment cofinality} of an ordinal $\pi$:
\[\scf(\pi):=\min\set{|c|}{c\subseteq \pi \text{ is a non-empty final segment of }\pi}.\]

\begin{theorem}[{\cite{P90},~\cite[Thm.~5.4~(c)]{CM23}}]\label{thm:cfcovSN}
    Let $\la \Por_\xi:\, \xi\leq\pi\ra$ be a $\subsetdot$-increasing sequence of posets such that $\Por_\pi = \limdir_{\xi<\pi}\Por_\xi$. Assume that $\cf(\pi)>\omega$, $\Por_\pi$ has the $\cf(\pi)$-cc and $\Por_{\xi+1}$ adds a Cohen real over the $\Por_\xi$-generic extension for all $\xi<\pi$. Then $\pi\leqT \Cbf_\Nwf^\perp$, in particular $\cov(\SNwf) \leq \cf(\pi) \leq \non(\SNwf)$.
\end{theorem}

\begin{theorem}[{\cite{P90},~\cite[Cor.~5.9]{CM23}}]\label{thm:precaliber}
Assume that $\theta\geq\aleph_1$ is regular. Let
$\Por_\pi=\la \Por_\xi,\Qnm_\xi:\, \xi<\pi\ra$ be a FS iteration of non-trivial precaliber $\theta$ posets such that $\cf(\pi)>\omega$ and $\Por_\pi$ has $\cf(\pi)$-cc,
and let $\lambda:=\scf(\pi)$. Then $\Por_\pi$ forces $\Cbf_{[\lambda]^{<\theta}}\leqT\Cbf_{\SNwf}^{\perp}$. In particular, whenever $\scf(\pi)\geq\theta$, $\Por_\pi$ forces $\cov(\SNwf)\leq\theta$ and $\scf(\pi)\leq\non(\SNwf)$.
\end{theorem}

\section{Goodness support}\label{sec:good}

We generalize and formalize a notion of \emph{support} original from the first author and Switzer \cite[Subsec.~4.4]{BS22}, which we call \emph{goodness support}. They developed this framework for the iteration of Hechler forcing to force $\Cbf_{[\R]^{<\aleph_1}}\leqT \Cbf_\Nwf^\perp$ (i.e.\ the existence of a Rothberger family for $\Nwf$) without applying \autoref{goodpres}, by using that Hechler forcing is $\aLc(2^{\id+1},1)$-good (note that $\aLc(2^{\id+1},1)\leqT \Cbf_\Nwf^\perp$). We expand their results and show that any FS iteration of $\theta$-cc $\theta$-$\Rbf$-good posets forces, basically, $\Cbf_{[\R]^{<\theta}}\leqT \Rbf$, where $\theta$ is an uncountable regular cardinal and $\Rbf$ is a gPrs.

For a FS iteration, we define the \emph{support} of a nice name of a real as follows.

%For a FS iteration, we define the notion of \emph{history} of a condition and of a name. This is a subset $H$ of the length of the iteration, which indicates that the condition (or the name) depends only on the generic objects added at $\xi\in H$.

\begin{definition}\label{def:history}
Let $\Por$ be a poset and let $Z$ be a Polish space (in the ground model). A Polish space is the completion of some metric space $\la \gamma,d \ra$ where $\gamma\leq\omega$ is an ordinal, hence 
any member of $Z$ is determined by a Cauchy sequence from $\la\gamma,d\ra$. We say that $\dot y$ is a \emph{nice $\Por$-name of a member of $Z$} if it is a nice $\Por$-name of a Cauchy sequence from $\la\gamma,d\ra$, i.e.\ it is constructed from a sequence of maximal antichains $\la A_n:\, n<\omega\ra$ such that each member of $A_n$ decides the $n$-th term of the Cauchy sequence. When using the forcing relation, or in generic extensions, we identify such a nice name with the limit of the Cauchy sequence.

When dealing with nice $\Por$-names of Cauchy sequences, we fix a well-order $H_\chi$ for some large enough regular $\chi$ to pick a \emph{code} for every nice $\Por$-name of a member of $Z$. Precisely, such a code is a sequence $\Seq{A_n,h_n}{n<\omega}$ such that each $A_n$ is a maximal antichain in $\Por$ and $h_n$ is a function with domain $A_n$ such that each $p\in A_n$ forces that $h_n(p)$ is the $n$-th term of the Cauchy sequence with limit the name that is being encoded. Something like this is required to make sense of e.g.\ \autoref{def:gsupp}.

Let $\la \Por_\xi,\Qnm_\xi:\, \xi<\pi\ra$ be a FS iteration. For any nice $\Por_\pi$-name $\dot y$ of a member of $Z$, define $\supp(\dot y) := \bigcup_{n<\omega}\bigcup_{p\in A_n} \dom p$ (where each $A_n$ is as in the previous paragraph). 
\end{definition}

% \begin{definition}\label{def:history}
% Let $\la \Por_\xi,\Qnm_\xi:\, \xi<\pi\ra$ be a FS iteration. We define the \emph{history} $H(p)$ of a condition $p\in\Por_\xi$ and the \emph{history} $H(\tau)$ of a $\Por_\xi$-name $\tau$ by recursion on $\xi\leq\pi$ as follows.
% When, for all $\zeta<\xi$,  $H(r)\subseteq\zeta$ has been defined for all $r\in\Por_\zeta$ 
% and $H(\sigma)\subseteq\zeta$ for any $\Por_\zeta$-name $\sigma$:
%     \begin{enumerate}[label=\rm (\roman*)]
%         \item when $\xi=0$, set $H(p):=\emptyset$;
%         \item when $\xi=\zeta+1$, %and $p\in\Por_{\zeta+1}$,
%            \[H(p):=\left\{\begin{array}{ll}
%                 H(p{\upharpoonright}\zeta) & \text{if $\zeta\notin\dom p$,}  \\
%                 H(p{\upharpoonright}\zeta)\cup\{\zeta\}\cup H(p(\zeta)) & \text{if $\zeta\in\dom p$}
%            \end{array}\right.\]
%            (here, $H(p(\zeta))$ is defined because $p(\zeta)$ is a $\Por_\zeta$-name);
           
%         \item when $\xi$ is limit, $H(p)$ has already been defined (because $p\in\Por_\zeta$ for some $\zeta<\xi$);
%         \item in any of the cases above,  
%         \[H(\tau):=\bigcup\{H(\sigma)\cup H(p):(\sigma,p)\in\tau\}.\]
%     \end{enumerate}
%     It is clear that $H(p)$ and $H(\tau)$ are subsets of $\xi$, and that $\dom p\subseteq H(p)$.
% \end{definition}

\begin{lemma}\label{lem:hist}
Let $\xi\leq\pi$ and let $\dot y$ be a nice $\Por_\pi$-name of a member of $Z$. Then $\dot y$ is a $\Por_\xi$-name iff $\supp(\dot y)\subseteq \xi$.
\end{lemma}

% \begin{lemma}\label{lem:hist}
% Let $\xi\leq\pi$. 
% \begin{enumerate}[label=\rm (\alph*)]
%     \item\label{hist:a} If $p\in\Por_\pi$ and $H(p)\subseteq\xi$ then $p\in\Por_\xi$. 
%     \item\label{hist:b} If $\tau$ is a $\Por_\pi$-name and $H(\tau)\subseteq\xi$ then $\tau$ is a $\Por_\xi$-name.
% \end{enumerate}
% \end{lemma}
% \begin{proof}
% \ref{hist:a} follows because $\dom p\subseteq H(p)\subseteq\xi$. Item~\ref{hist:b} can be proved by induction on $\rk(\tau)$ (using~\ref{hist:a}).
% \end{proof}

For this whole section, fix an uncountable regular cardinal $\theta$, a Polish space $Z$ and a FS iteration $\la \Por_\xi,\Qnm_\xi:\, \xi<\pi\ra$ of $\theta$-cc posets. %As usual in the construction of FS iterations, we assume that every $p\in\Por_\xi$ ($\xi\leq\pi$) is a function with finite domain contained in $\xi$ such that $p(\zeta)$ is a $\Por_\zeta$-nice name of a real for all $\zeta\in\dom p$.
In this context, we obtain:

\begin{lemma}\label{smallhist}
    For any nice $\Por_\pi$-name $\dot y$ of a member of $Z$, $|\supp(\dot y)| < \theta$.
\end{lemma}

For the rest of this section, fix a gPrs $\Rbf=\la X,Y,\sqsubset\ra$ where $Y=\bigcup_{e\in\Omega}Y_e$ is a union of non-empty analytic subsets of $Z$. We also fix a continuous surjection $f_e\colon \omega^\omega\to Y_e$ for each $e\in\Omega$.

We now aim to define the \emph{$\Rbf$-goodness support} of a name of a real in $Y$. For this, we need a couple of preliminary lemmata.

\begin{definition} Let $\Pbb$ be a forcing notion and let $\dot{z}$ be a $\Pbb$-name for a real in $\omega^\omega$. A  pair $(\bar p,g)$ is called an \textit{interpretation of $\dot{z}$ in $\Pbb$} if $g\in\omega^{\omega}$ and $\bar p=\la p_k:\, k<\omega\ra$ is a decreasing sequence of conditions in $\Por$ such that  $p_k\Vdash \dot{z}\frestr k=g\frestr k$ for any $k<\omega$.
%Say that this interpretation is \textit{below $p\in\Pbb$} if, additionally, $p_{0}\leq p$.
\end{definition}

\begin{lemma}[{\cite[Lemma~4.9]{CM}}]\label{intrlem}
Assume that $\Pbb$ is a poset, $e\in \Omega$, $f \colon\omega^\omega\to Y_e$ is a continuous function, $\dot{z}$ is a $\Pbb$-name for a real in $\omega^\omega$ and $(\bar p,g)$ is an interpretation of $\dot{z}$ in $\Pbb$. If $x\in X$, $n<\omega$ and $\neg(x \sqsubset_n f(g))$, then there is a $k<\omega$ such that $p_{k}\Vdash \neg(x \sqsubset_{n} f(\dot{z}))$.
\end{lemma}

\begin{lemma}\label{limitgoodlem}
Let $\delta\leq\pi$ be a limit ordinal, $\la \gamma_\eta:\, \eta<\cf(\delta)\ra$ increasing and cofinal in $\delta$, and let $\dot y$ be a nice $\Por_\delta$-name of a member of $Y$.\footnote{I.e.\ a nice name of a member of $Z$ which is forced to be in $Y$.} Then there is a sequence $\la W_\eta:\, \eta<\cf(\delta)\ra$ such that each $W_\eta$ is a non-empty set of nice $\Por_{\gamma_\eta}$-names of members of $Y$, $|W_\eta|<\theta$, and $\Por_\delta$ forces:
\begin{center}
    for any $\xi<\delta$ and $x\in X\cap V_\xi$, 
    if $x$ is $\Rbf$-unbounded over $W_\eta$ for all $\eta$ in some cofinal subset of $\cf(\delta)$, then $\neg(x \sqsubset \dot y)$.
\end{center}
\end{lemma}
\begin{proof}
Choose some maximal antichain $A$ in $\Por_\delta$ and some $e'\colon A\to\Omega$ ($e_r:=e'(r)$) such that $r\Vdash_\delta \dot y\in Y_{e_r}$ for all $r\in A$. In addition, pick a nice $\Por_\delta$-name $\dot{y}^r$ of a member of $Y_{e_r}$ such that $r\Vdash \dot y^r=\dot y$.

It is enough to obtain the desired $W^r_\eta$ for each $\dot y^r$ because $W_\eta=\bigcup_{r\in A}W^r_\eta$ works. So we assume wlog that $\dot y$ is a nice $\Por_\delta$-name of a member of some $Y_e$. Choose a nice $\Por_\delta$-name $\dot z$ of a member of $\omega^\omega$ such that $\Por_\delta$ forces $f_e(\dot z)=\dot y$.

Fix $\eta<\cf(\delta)$. Choose $\Por_{\gamma_\eta}$-names $\dot{\bar{p}}^\eta$ and $\dot z^\eta$ such that $\Por_{\gamma_\eta}$ forces that $(\dot{\bar{p}}^\eta,\dot z^\eta)$ is an interpretation of $\dot z$ in $\Por_\delta/\Por_{\gamma_\eta}$. Let $\dot w_\eta$ be a nice $\Por_{\gamma_\eta}$-name of $f_e(\dot z^\eta)$ and $W_\eta:=\{\dot w_\eta\}$.

We show that $\la W_\eta:\, \eta<\cf(\delta)\ra$ is as required. Let $\xi<\delta$, $n<\omega$, $p\in\Por_\delta$, let $\dot x$ be a nice $\Por_\xi$-name of a member of $X$, and assume that $p$ forces that $\neg(\dot x \sqsubset w_\eta)$ for $\eta$ in some cofinal subset of $\cf(\delta)$. It is enough to show that there is some $q\leq p$ forcing $\neg(\dot x \sqsubset_n \dot y)$.

Pick some $\xi'<\delta$ such that $\xi'\geq \xi$ and $p\in\Por_{\xi'}$, and find some $p'\leq p$ in $\Por_\delta$ and some $\eta_0 < \cf(\delta)$ such that $\xi_0:=\gamma_{\eta_0}\geq \xi'$ and $p'$ forces $\neg(\dot x \sqsubset \dot w_{\eta_0})$. Then $p_0:=p'\frestr \xi_0$ forces the same with respect to $\Por_{\xi_0}$ (by absoluteness of $\sqsubset$ because both $\dot x$ and $w_{\eta_0}$ are $\Por_{\xi_0}$-names) and $p_0\leq p$ in $\Por_{\xi_0}$.

%By strengthening $p$ if required, we can choose some $\eta_0<\cf(\delta)$ such that $p\in\Por_{\xi_0}$ where $\xi_0:=\gamma_{\eta_0}$, $\dot x$ is a $\Por_{\xi_0}$-name, and $p$ forces that $\neg(\dot x \sqsubset \dot w_{\eta_0})$. %(find $\eta_0$ large enough while strengthening $p$, and use $p\frestr \gamma_{\eta_0}$)

Let $G$ be $\Por_{\xi_0}$-generic over $V$ with $p_0\in G$, and work in $V[G]$. Then $\neg(x \sqsubset w_{\eta_0})$, so $\neg(x \sqsubset_n w_{\eta_0})$. By \autoref{intrlem}, there is some $k<\omega$ such that $p^{\eta_0}_k\Vdash_{\Por_\delta/\Por_{\xi_0}}\neg( x \sqsubset_n \dot y)$. Back in $V$, we can obtain some $q\leq p_0$ in $\Por_\delta$ forcing $\neg(\dot x \sqsubset_n \dot y)$.
\end{proof}

We are now ready to define the \emph{$\Rbf$-goodness support} $\gsupp_\Rbf(\dot y)$ for a nice name $\dot y$ of a member of $Y$. Here, we require that the iterands are (forced to be) $\theta$-$\Rbf$-good. The goodness support $\gsupp_\Rbf(\dot y)$ not only contains the support of the name, but also the support of the names involved in obtaining the \emph{good} set for $\dot y$. In more detail, recall that goodness indicates that we can choose some non-empty $W\subseteq Y$ (in the ground model) of size ${<}\theta$ such that any $x\in X$ $\Rbf$-unbounded over $W$ is forced to be unbounded by $\dot y$, so the goodness support of $\dot y$ must contain the goodness support of all (nice names of the) members of $W$. With the goodness support, we can find such $W$'s in intermediate steps with stronger features, where the support (and even goodness support) of their members are contained in $\gsupp_\Rbf(\dot y)$. This is stated concretely in \autoref{itegsupp}.

%We are now ready to define the \emph{$\Rbf$-goodness support} $\gsupp_\Rbf(\dot y)$ for a nice name $\dot y$ of a member of $Y$ (i.e., a nice name of a member of $Z$ which is forced to be in $Y$). Here, we require that the iterands are (forced to be) $\theta$-$\Rbf$-good. The goodness support $\gsupp_\Rbf(\dot y)$ not only contains the history of the name but indicates that the \emph{goodness of $\dot y$} depends on the generic sets added at $\xi\in \gsupp_\Rbf(\dot y)$. In more detail, recall that goodness indicates that there is some non-empty $W\subseteq Y$ (in the ground model) of size ${<}\theta$ such that any $x\in X$ $\Rbf$-unbounded over $W$ is forced to be unbounded by $\dot y$. With the goodness support, we can find such $W$'s in intermediate steps with stronger features, where the history (and even goodness support) of their members are contained in $\gsupp_\Rbf(\dot y)$. This is stated concretely in \autoref{itegsupp}. 

\begin{definition}\label{def:gsupp}
Assume (in addition) that $\Por_\xi$ forces that $\Qnm_\xi$ is $\theta$-$\Rbf$-good for all $\xi<\pi$. We define the \emph{$\Rbf$-goodness support $\gsupp_\Rbf(\dot y)$} for any nice $\Por_\xi$-name $\dot y$ of a member of $Y$ by recursion on $\xi\leq\pi$:
\begin{enumerate}[label=\rm (\roman*)]
    \item when $\xi=0$, $\gsupp_\Rbf(\dot y)=\emptyset$;
    \item\label{gsuppsuc} when $\xi=\zeta+1$, pick a non-empty set $W_{\dot y}$ of nice $\Por_\zeta$-names of members of $Y$ such that $|W_{\dot y}|<\theta$ and the members of $W_{\dot y}$ witness the $\Rbf$-goodness of $\dot y$ in $V_\zeta$ (with respect to $\Qnm_\zeta$), and define
    \[\gsupp_\Rbf(\dot y):=\supp(\dot y)\cup\bigcup_{\tau\in W_{\dot y}}\gsupp_\Rbf(\tau);\]
    
    \item\label{gsuppsmallcof} when $\xi$ is limit and $\cf(\xi)<\theta$, pick some increasing cofinal sequence $\la \gamma_\eta:\, \eta<\cf(\xi)\ra$ in $\xi$, for each $\eta<\cf(\xi)$ choose some non-empty set $W_{\dot y,\eta}$ of nice $\Por_{\gamma_\eta}$-names of members of $Y$ as in \autoref{limitgoodlem}, and define
    \[\gsupp_\Rbf(\dot y):=\supp(\dot y)\cup\bigcup_{\eta<\cf(\xi)}\bigcup_{\tau\in W_{\dot y,\eta}}\gsupp_\Rbf(\tau);\]
    \item when $\cf(\xi)\geq\theta$, $\dot y$ is a nice $\Por_\zeta$-name for some $\zeta<\xi$, so $\gsupp_\Rbf(\dot y)$ is already defined.
\end{enumerate}
When $\Rbf$ is understood, we write $\gsupp(\dot y)$ for the goodness support. 
It is clear that $\supp(\dot y)\subseteq \gsupp(\dot y)\subseteq\xi$.
\end{definition}

\begin{lemma}\label{smallgsupp}
Under the assumptions in \autoref{def:gsupp}, for any $\xi\leq\pi$ and for any nice $\Por_\xi$-name $\dot y$ of a member of $Y$, $|\gsupp(\dot y)|<\theta$.
\end{lemma}
\begin{proof}
Proceed by induction on $\xi$.
\end{proof}

The main use of the goodness support is illustrated in the following results, which generalize~\cite[Fact~4.7]{BS22}.

\begin{theorem}\label{itegsupp}
Let $\xi\leq\xi'\leq\delta\leq \pi$ and let $\dot y$ be a nice $\Por_\delta$-name of a member of $Y$. If $\gsupp(\dot y)\cap[\xi,\xi')=\emptyset$ then there is a non-empty set $W^*$ of nice $\Por_\xi$-names of members of $Y$ with $|W^*|<\theta$ such that $\bigcup_{\tau\in W^*}\gsupp(\tau)\subseteq\gsupp(\dot y)$ and $\Por_{\xi'}$ forces:
\begin{center}
    For any $x\in X$ $\Rbf$-unbounded over $W^*$, $\Por_\delta/\Por_{\xi'}$ forces $\neg(x\sqsubset \dot y)$.
\end{center}
\end{theorem}
\begin{proof}
When $\delta=\xi'$ we have by \autoref{lem:hist} that $\dot y$ is a $\Por_\xi$-name, so $W^*=\{\dot y\}$ works. So we prove the theorem by induction on $\delta\leq\pi$ for $\xi\leq\xi'<\delta$. The case $\delta=0$ is vacuous (no $\xi'<0$).

Assume $\delta=\zeta+1$ and $\xi\leq\xi'<\delta$ (so $\xi'\leq\zeta$). Choose $W_{\dot y}$ as in \autoref{def:gsupp}~\ref{gsuppsuc}. By induction hypothesis, since each $\tau\in W_{\dot y}$ is a nice $\Por_\zeta$-name of a member of $Y$ and $\gsupp_\Rbf(\tau)\subseteq\gsupp_\Rbf(\dot y)$ is disjoint with $[\xi,\xi')$, we obtain a non-empty set $W^\tau$ of size ${<}\theta$ of $\Por_\xi$-names of members of $Y$ such that $\bigcup_{\sigma\in W^\tau}\gsupp(\sigma)\subseteq\gsupp(\tau)$ and $\Por_{\xi'}$ forces:
\begin{center}
    For any $x\in X$ $\Rbf$-unbounded over $W^\tau$, $\Por_\zeta/\Por_{\xi'}$ forces $\neg(x\sqsubset \tau)$.
\end{center}
Let $W^*:=\bigcup_{\tau\in W_{\dot y}}W^\tau$, which is as required: in the $\Por_{\xi'}$ extension, if $x\in X$ is $\Rbf$-unbounded over $W^*$, then $\Por_\zeta/\Por_{\xi'}$ forces that $x$ is $\Rbf$-unbounded over $W_{\dot y}$, so by goodness of $\Qnm_\zeta$, $\Por_\delta/\Por_{\xi'}$ forces $\neg(x\sqsubset \dot y)$. On the other hand, it is clear that $\bigcup_{\sigma\in W^*}\gsupp(\sigma)\subseteq\bigcup_{\tau\in W_{\dot y}}\gsupp(\tau)\subseteq\gsupp(\dot y)$.

Assume now that $\delta$ is limit, $\cf(\delta)<\theta$ and $\xi\leq\xi'<\delta$. Let $\la\gamma_\eta:\, \eta<\cf(\delta)\ra$ and $\la\dot W_{\dot y,\eta}:\, \eta<\cf(\delta)\ra $ be as in \autoref{def:gsupp}~\ref{gsuppsmallcof}. For each $\eta<\cf(\delta)$ and $\tau\in W_{\dot y,\eta}$, since $\gsupp(\tau)\subseteq \gsupp(\dot y)$, by induction hypothesis applied to $\xi_\eta:=\min\{\xi,\gamma_\eta\}$ and $\xi'_\eta:=\min\{\xi',\gamma_\eta\}$, we can find a non-empty set $W^\tau$ of size ${<}\theta$ of $\Por_{\xi_\eta}$-names of members of $Y$ such that $\bigcup_{\sigma\in W^\tau}\gsupp(\sigma)\subseteq\gsupp(\tau)$ and $\Por_{\xi'_\eta}$ forces:
\begin{center}
    For any $x\in X$ $\Rbf$-unbounded over $W^\tau$, $\Por_{\gamma_\eta}/\Por_{\xi'_\eta}$ forces $\neg(x\sqsubset \tau)$.
\end{center}
Then $W^*:=\bigcup_{\eta<\cf(\delta)}\bigcup_{\tau\in W_{\dot y, \eta}}W^\tau$ works. It is clear that $\gsupp(\sigma)\subseteq\gsupp(\dot y)$ for any $\sigma\in W^*$. Now, let $G_{\xi'}$ be $\Por_{\xi'}$-generic over $V$, and work in $V_{\xi'}$. Assume that $x\in X$ is $\Rbf$-unbounded over $W^*$ (after evaluating with $G_\xi=G_{\xi'}\cap\Por_\xi$). When $\gamma_\eta>\xi'$ we obtain $\xi_\eta=\xi$ and $\xi'_\eta=\xi'$, and since $x$ is $\Rbf$-unbounded over $W^\tau$ (after evaluating) for any $\tau\in W_{\dot y,\eta}$, we get that $\Por_{\gamma_\eta}/\Por_{\xi'}$ forces $\neg(x\sqsubset \tau)$, and the same is forced by $\Por_{\delta}/\Por_{\xi'}$. Therefore, by \autoref{limitgoodlem}, since $\Por_{\delta}/\Por_{\xi'}$ forces that $x$ is $\Rbf$-unbounded over $W_{\dot y,\eta}$ for $\eta$ in some cofinal segment of $\cf(\delta)$, we conclude that $\Por_{\delta}/\Por_{\xi'}$ forces $\neg(x\sqsubset \dot y)$.

In the case $\cf(\delta)\geq \theta$, pick some $\delta_0<\delta$ such that $\xi'<\delta_0$ and $\dot y$ is a nice $\Por_{\delta_0}$-name. The conclusion follows after applying the induction hypothesis to $\delta_0$.
\end{proof}

Although the following is well known, it follows directly from the previous result.

\begin{corollary}\label{FSIgoodness}
Any FS iteration of $\theta$-cc $\theta$-$\Rbf$-good posets is $\theta$-$\Rbf$-good.
\end{corollary}
\begin{proof}
Apply \autoref{itegsupp} to $\xi=\xi'=0$ and $\delta=\pi$ (the length of the iteration).
\end{proof}

As an application of \autoref{itegsupp}, we can force $\Cbf_{[\pi]^{<\theta}}\leqT \Rbf$ directly. Typically, to force such a statement, Cohen reals are added to force the Tukey connection, afterward the iteration of $\theta$-$\Rbf$-good posets is performed, and the Tukey connection is preserved by \autoref{goodpres}. But now, thanks to the following result, we do not need the first Cohen reals to force the Tukey connection. Instead, the Cohen reals that are added at the limit steps take care of this job.

\begin{theorem}\label{Comgood}
Let $\la \Por_\xi,\Qnm_\xi:\, \xi<\pi\ra$ be a FS iteration such that $\Por_\xi$ forces that $\Qnm_\xi$ is a non-trivial $\theta$-cc $\theta$-$\Rbf$-good poset. 
Let $\set{\gamma_\alpha}{\alpha<\delta}$ be an increasing enumeration of $0$ and all limit ordinals smaller than $\pi$ (note that $\gamma_\alpha=\omega\alpha$), and for $\alpha<\delta$ let $\dot c_\alpha$ be a $\Por_{\gamma_{\alpha+1}}$-name of a Cohen real in $X$ over $V_{\gamma_\alpha}$. 

If (in $V$) $I\subseteq\delta$ and $|I|\geq \theta$ then $\Por_\pi$ forces that $\set{\dot c_\alpha}{\alpha\in I}$ is $\theta$-$\Rbf$-unbounded. In particular, if $\pi\geq\theta$ then $\Cbf_{[\pi]^{<\theta}}\leqT\Rbf$, $\bfrak(\Rbf)\leq\theta$ and $|\pi|\leq\dfrak(\Rbf)$.
\end{theorem}
\begin{proof}
In the final extension, let $c_\alpha\in X\cap V_{\gamma_{\alpha+1}}$ be the evaluation of $\dot c_\alpha$. Since it is a Cohen real over $V_{\gamma_\alpha}$, $c_\alpha$ is $\Rbf$-unbounded over $Y\cap V_{\gamma_\alpha}$.

Work in $V$ and let $\dot y$ be a nice $\Por_\pi$-name of a member of $Y$. Since $|\gsupp(\dot y)|<\theta$, we get that 
\[L:=\set{\alpha\in I}{\gsupp(\dot y)\cap[\gamma_\alpha,\gamma_{\alpha+1})\neq\emptyset} \text{ has size }{<}\theta.\]
For $\alpha\in I\menos L$, by \autoref{itegsupp} applied to $\xi=\gamma_\alpha$, $\xi'=\gamma_{\alpha+1}$ and $\delta=\pi$, we obtain that $\Por_\pi$ forces that $\neg(\dot c_\alpha \sqsubset \dot y)$. Therefore, in the final extension, $\set{\alpha\in I}{c_\alpha\sqsubset y}\subseteq L$, so it has size ${<}\theta$.

If $\pi\geq\theta$ then $|\delta|=|\pi|$, which implies $\Cbf_{[\pi]^{<\theta}}\leqT\Rbf$ when $I=\delta$.
\end{proof}

% \azul{
% The following observation will be used to prove~\autoref{thm:Paddsn}. 

% \begin{remark}\label{Rem:Comgood}
% Let $f$ and $\Gwf$ as in \autoref{def:Sf}. If we use $\Rbf^f_\Gwf$ instead of $\Rbf$ in the hypothesis of~\autoref{Comgood}, then we can prove that 
% $\Por_\pi$ forces that \[\set{(t,\dot c_\alpha)}{t\in\bigcup_{m<\omega}2^{f(m)},\ \alpha\in I} \text{ is $\theta$-$\Rbf^f_\Gwf$-unbounded}\]
% in a way totally similar to~\autoref{Comgood}.
% \end{remark}
% }

\begin{remark}\label{rem:smallH}
    Fuchino and the third author have an unpublished proof of \autoref{Comgood} that does not use goodness support (cf.~\cite{FuMe}).
\end{remark}

\section{Forcing the additivity of \texorpdfstring{$\SNwf$}{} small}\label{preaddSN}

To force $\add(\SNwf)$ small, we find a suitable Polish relational system such that FS iterations as in \autoref{Comgood} (for this relational system) force $\Cbf_{[\pi]^{<\theta}}\leqT \SNwf$, which corresponds to our main technical result \autoref{mainpresaddSN} (\autoref{thm:Paddsn}). It is consistent that $\SNwf$ cannot be represented by a definable relation system of the reals (\autoref{def:defrel}), e.g.\ CH implies that $\cof(\SNwf)>\aleph_1$ while $\dfrak(\Rbf)\leq\cfrak$ holds for any relational system of the reals $\Rbf$.\footnote{The opposite follows by Borel's Conjecture, i.e.\ $\SNwf=[\R]^{<\aleph_1}$ is a relational system of the reals. We thank the referee for this observation.} For this reason, 
we have to work more to force the desired Tukey connection. 

We motivate our main result as follows. We perform a FS iteration $\Por = \la \Por_\xi,\Qnm_\xi:\, \xi<\pi\ra$ of length $\pi$, where $\pi$ has uncountable cofinality and all iterands have the ccc. Let $\theta\leq\lambda$ be uncountable cardinals with $\theta$ regular, and assume that we have constructed $\Por$-names $\la \dot X_\beta:\, \beta<\lambda\ra$ such that $\dot X_\beta = \bigcap_{\alpha<\lambda}[\dot \sigma^\beta_\alpha]_\infty$ is in $\SNwf$ for some $\Por$-names $\dot\sigma^\beta_\alpha$ ($\alpha<\lambda$) of members of $(2^{<\omega})^\omega$. Any $\dot\sigma^\beta_\alpha$ is taken as a Cohen real (over some intermediate extension). We deduce the requirements of the iteration to guarantee that $\{\dot X_\beta:\, \beta<\lambda\}$ is forced to be $\theta$-$\SNwf$-unbounded (meaning that any $Y\in \SNwf$ only contains ${<}\theta$-many of the $X_\beta$'s). We even aim for a stronger statement: we fix some increasing function $f\in\omega^\omega$ in the ground model and aim to show that, in the final extension, for any $\tau\in 2^f$,
\[\left|\set{\beta<\lambda}{X_\beta\subseteq \bigcup_{k<\omega}[\tau(k)]}\right|<\theta.\]
In the final extension, let $T\subseteq 2^{<\omega}$ be the well-pruned tree such that $[T]=2^\omega\menos \bigcup_{k<\omega}[\tau(k)]$. Assuming $f(i+1) - f(i)\geq 2$ for infinitely many $i<\omega$, we have that $T$ is a perfect tree. Now, $T$ lives in some intermediate extension, and as any FS iteration adds Cohen reals, we have some intermediate step where we add a Cohen real $c\in[T]$, i.e.\ $c\notin \bigcup_{k<\omega}[\tau(k)]$. So it would be enough to show that $|\set{\beta<\lambda}{c\notin X_\beta}|<\theta$, i.e., that $c\in X_\beta$ for most of the $\beta$'s. Fixing $\beta$, we will have that many $\sigma^\beta_\alpha$'s are added after $c$, so $c\in[\sigma^\beta_\alpha]_\infty$ (recall that each $\sigma^\beta_\alpha$ is a Cohen real), but for the $\sigma^\beta_\alpha$'s added before $c$, to prove $c\in[\sigma^\beta_\alpha]_\infty$ we require to show that the set $\set{\sigma^\alpha_\beta(n)}{n\geq n_0}$ is dense in $T$ for all $n_0<\omega$ (hence the generic set producing $c$ intersects it). In the following lemma, we dissect the elements we need to guarantee such density. We remark that the choice of $\beta$ should guarantee that no $\sigma^\beta_\alpha$ is added at the same stage as $c$.

\begin{lemma}\label{magiclem}
    Let $f\colon \omega\to\omega$ be an increasing function and let $\tau\in 2^f$ and $\sigma\in(2^{<\omega})^\omega$. Assume that, for all $t\in\bigcup_{i<\omega}2^{f(i)}$,
    \begin{equation}\label{eq:dense}
    \exists^\infty\, n<\omega\colon \sigma(n)\supseteq t \text{ and } \forall\, k<\omega\colon |t|<f(k)\leq|\sigma(n)| \imp \sigma(n){\upharpoonright}f(k)\neq \tau(k).\tag{\faBolt}
    \end{equation}
    Then, $\set{\sigma(n)\in T}{n\geq n_0}$ is dense in $T$ for all $n_0 < \omega$, where $T\subseteq 2^{<\omega}$ is the well-pruned tree such that $[T]=2^\omega\menos \bigcup_{k<\omega}[\tau(k)]$.
\end{lemma}
\begin{proof}
    Note that $t\in T$ iff, for all $k<\omega$ such that $f(k)\leq |t|$, $t\frestr f(k)\neq \tau(k)$. Hence,~\eqref{eq:dense} implies that, for any $t\in T$, there are infinitely many $n<\omega$ such that $\sigma(n)\supseteq t$ and $\sigma(n)\in T$ (for the latter, we require that $t\in T$). Therefore, the conclusion follows.
\end{proof}

Now, to guarantee~\eqref{eq:dense} for the density argument, we define a Polish relational system that describes~\eqref{eq:dense} and, thanks to \autoref{Comgood}, we can conclude our proof when the iterands of the iteration are ccc and $\theta$-good with respect to this relational system (see details in \autoref{thm:Paddsn}). To guarantee goodness in the relevant cases (e.g.\ for random forcing), using a $\tau\in 2^f$ is not enough, so we need to expand to slaloms instead, and we also need further parameters $\la h_n:\, n<\omega\ra$ as indicated below.

\begin{definition}\label{def:Sf}
Fix an increasing function $f\in\omega^\omega$ with $f>0$. Let $\Gwf:=\set{h_n}{n<\omega}\subseteq\omega^\omega$ be a family of increasing functions which satisfies, for all $n, i<\omega$,
\begin{enumerate}[label= \rm (\roman*)]
    %\item $h_n(0)>0$, 
    \item $2^{i+1}h_n(i)\leq h_{n+1}(i)$, and 
    \item if $i\geq n$ then $h_n(i)<2^{f(i)-f(i-1)}$ (here $f(-1):=0$).
\end{enumerate}
Note that such a $\Gwf$ exists whenever $f(i) \geq f(i-1)+(i+1)^2$ for all $i<\omega$, 
e.g.\ $f(i) = \sum_{j\leq i+1}j^2 = \frac{(i+1)(i+2)(2i+3)}{6}$ and  $h_n(i) := (2^{i+1})^n\, 2^i$.

Define %$\Swf(2^f,\Gwf)=\bigcup_{n<\omega}\Swf(2^f,h_n)$ and  
$\Swf_f(\Gwf):=\set{\psi\in \bigcup_{g\in \Gwf} \prod_{i<\omega}[2^{f(i)}]^{\leq g(i)}}{\forall\, i<\omega\colon |\psi(i)|<2^{f(i)-f(i-1)}}$.

For $t\in\bigcup_{i<\omega}2^{f(i)}$, $\sigma\in(2^{<\omega})^\omega$ and $\psi\in \Swf_f (\Gwf)$, 
%such that $\exists^\infty n\colon t\subseteq \sigma(n)$, 
we define the relation
\[(t, \sigma) \sqsubset^f\psi\text{ iff }\forall^\infty\, n\colon \sigma(n)\supseteq t\imp \exists\, k<\omega\colon |t|<f(k)\leq|\sigma(n)| \text{ and } \sigma(n){\upharpoonright}f(k)\in\psi(k).\]
Put $\Rbf^f_\Gwf:=\la\bigcup_{i<\omega}2^{f(i)}\times(2^{<\omega})^\omega, \Swf_f(\Gwf),  \sqsubset^f\ra$, which is a relational system.
\end{definition}

The negation of $\sqsubset^f$ describes the situation in~\eqref{eq:dense} (for $\psi(k)=\{\tau(k)\})$:
\[(t, \sigma) \nsqsubset^f\psi\text{ iff }\exists^\infty\, n\colon \sigma(n)\supseteq t \text{ and } \forall\, k<\omega\colon |t|<f(k)\leq|\sigma(n)| \imp \sigma(n){\upharpoonright}f(k)\notin\psi(k).\]

\begin{lemma}\label{RfPrs}
Let $f$ and $\Gwf$ as in \autoref{def:Sf}. Then:
\begin{enumerate}[label= \rm (\alph*)]
    \item\label{RfPrsone} For any $\psi\in\Swf_f(\Gwf)$, $n_0,i,m<\omega$ with $m>f(i)$, and $t\in 2^{f(i)}$, there is some $t'\in 2^{m}$ such that $t'\supseteq t$ and $t'\frestr f(k)\notin\psi(k)$ whenever $|t|<f(k)\leq m$. 
    
    \item\label{RfPrstwo} The relational system $\Rbf^f_\Gwf$ is a Prs. 
    
    \item\label{RfPrsthree} If $t\in \bigcup_{i<\omega} 2^{f(i)}$ and $\sigma\in (2^{<\omega})^\omega$ is Cohen over a transitive model $M$ of $\thzfc$ with $f,\Gwf\in M$, then $(t,\sigma)\nsqsubset^f \psi$ for all $\psi\in\Swf_f(\Gwf)\cap M$.
\end{enumerate}
\end{lemma}
\begin{proof}
\ref{RfPrsone}: Let $j:=\min\set{k<\omega}{m\leq f(k)}$. We define $u\frestr f(k)$ by recursion on $k\in[i,j]$. We start with $u\frestr f(i):=t$.
Assume we have defined $u\frestr f(k-1)$ ($i<k\leq j$).  Since $|\psi(k)|<2^{f(k)-f(k-1)}$, we can find some $t_k\in 2^{f(k)-f(k-1)}$ such that $s^{\frown} t_k \notin \psi(k)$ for all $s\in 2^{f(k-1)}$, so we set $u\frestr f(k):=u\frestr f(k-1)^{\frown} t_k$. At the end, set $t':=u\frestr m$.

\ref{RfPrstwo}: Easy to check; condition~\ref{def:Prsc} of \autoref{def:Prs} follows from~\ref{RfPrsone}.

%To see~\ref{RfPrstwo}, note that the only condition in~\autoref{def:Prs} that needs some detailed proof is that the set
% \[\set{\sigma\in (2^{<\omega})^\omega}{\forall\,  n\geq n_0\colon \sigma(n)\supseteq t\imp \exists\, k<\omega\colon |t|<f(k)\leq|\sigma(n)| \text{ and } \sigma(n){\upharpoonright}f(k)\in\psi(k)}\]
% is nowhere dense for any $n_0<\omega$, $t\in \bigcup_{i<\omega}2^{f(i)}$ and $\psi\in\Scal_f(\Gwf)$. This follows from~\ref{RfPrsone}. 

\ref{RfPrsthree}: It follows directly from~\ref{RfPrstwo}. 
\end{proof}

We present the natural examples of $\Rbf^f_\Gwf$-good posets, which include random forcing. These are the Boolean algebras with a strictly positive (probability) finitely additive measure, 
as in the case of Kamburelis' result to force $\add(\Nwf)$ small (see \autoref{ExmPrs}~\ref{ExmPrsd}). The first step to prove this is to show the result for finitely additive measures with one additional property.

\begin{definition}
For a Boolean algebra $\Aor$, say that $\mu\colon \Aor\to[0,1]$ is a \textit{strictly positive probability finitely additive measure (pfam)} if it fulfills:
\begin{enumerate}[label= \rm (\roman*)]
    \item $\mu(\mathbf{1}_{\Aor}) = 1$, 
    \item $\mu(a\vee a') = \mu(a)+\mu(a')$ for all $a,a'\in\Aor$ such that $a\wedge a'=\zerobf_{\Aor}$, and  
    \item $\mu(a) = 0$ iff $a = \mathbf{0}_{\Aor}$.
\end{enumerate}
Note that any Boolean algebra with a pfam is ccc. 

Say that a pfam $\mu$ has the \emph{density property} if there is some countable $S\subseteq\Aor\smallsetminus\{\zerobf_\Aor\}$ such that 
\begin{equation}
\tag{$\spadesuit$}
\forall\, a\in\Aor\smallsetminus\{\zerobf_{\Aor}\}\ \forall\, \varepsilon>0\ \exists\, s\in S\colon \mu(a\wedge s)>\mu(s)(1-\varepsilon). \label{Lbdproperty}
\end{equation}
\end{definition}

\begin{lemma}\label{exm:densepfam}
The following Boolean algebras have a pfam with the density property.
\begin{enumerate}[label=\rm (\roman*)]
    \item\label{it:famrandom} Any subalgebra of random forcing containing all the clopen sets. 
    \item\label{it:famED} Any subalgebra of $\Bor^h_b$  containing all the clopen sets in $\prod b$ when $b,h\in\omega^\omega$ satisfy $\sum_{i\in\omega}\frac{h(i)}{|b(i)|}<\infty$ (see~\autoref{EDsepfam}). 
    \item\label{it:famsc} Any $\sigma$-centered Boolean algebra.
\end{enumerate}
\end{lemma}
\begin{proof}
\ref{it:famrandom}: Clear thanks to the Lebesgue density theorem.

\ref{it:famED}:  Immediate by \autoref{EDsepfam} and~\ref{it:famrandom}.

\ref{it:famsc}: Note that a Boolean algebra $\Aor$ is $\sigma$-centered iff there exists a sequence $\set{G_n}{n<\omega}$ of ultrafilters on $\Aor$ such that $\Aor\smallsetminus\{\zerobf_{\Aor}\}=\bigcup_{n<\omega}G_n$. For $a\in\Aor$, define 
\[\mu(a):=\sum\set{\frac{1}{2^{n+1}}}{a\in G_n,\,n<\omega},
\]
Clearly, this is a pfam on $\Aor$. We show that it has the density property. For each $s\in2^{<\omega}$ choose some $q_s\in\Aor\menos\{\zerobf_{\Aor}\}$ such that 
\[\text{for all $k<|s|$, $s(k)=1$ iff $q_s\in G_k$}\] 
if a $q_s$ satisfying the above exists, otherwise set $q_s:=\onebf_\Aor$. We show that $S:=\{q_s:\, s\in2^{<\omega}\}$ witnesses the density property. For $a\in\Aor\menos\{\zerobf_\Aor\}$ let $z_a\in 2^\omega$ such that $z_a(k)=1$ iff $a\in G_k$, which is not the constant zero function, i.e.\ $z_a(k_0)=1$ for some $k_0<\omega$. For $\varepsilon>0$, choose some $N>k_0$ such that $\frac{1}{2^{N-(k_0+1)}+1} < \varepsilon$, and let $s:=z_a\frestr N$. Then $q_s\in G_k$ iff $a\in G_k$ for all $k<N$, so $\mu(q_s\smallsetminus a)\leq \sum_{n\geq N}2^{-(n+1)}=2^{-N}$ and $\mu(q_s)\geq 2^{-(k_0+1)} + \mu(q_s\menos a)$. Therefore,
\[\frac{\mu(q_s \smallsetminus a)}{\mu(q_s)}\leq \frac{\mu(q_s\smallsetminus a)}{2^{-(k_0+1)} + \mu(q_s\smallsetminus a)}= \frac{1}{\frac{2^{-(k_0+1)}}{\mu(q_s\smallsetminus a)}+1} \leq \frac{1}{2^{N-(k_0+1)}+1}<\varepsilon.
%\sum_{n\geq N}\frac{1}{2^{n+1}}<\varepsilon.
\qedhere\]
\end{proof}

It is not hard to show that any Boolean algebra with a pfam satisfying the density property is $\sigma$-$n$-linked for all $2\leq n<\omega$. Therefore, any Boolean algebra of size ${>}\cfrak$ cannot have such a pfam. By results of Kamburelis~\cite[Prop.~2.6]{Ka}, the completion of the poset adding $\cfrak^+$-many Cohen reals has a pfam, but it cannot have one with the density property by the previous observation.
%the random algebra adding $\cfrak^+$-many random reals (side-by-side) cannot have such a pfam.

We now prove that we obtain $\Rbf^f_\Gwf$-good sets from Boolean algebras with a pfam satisfying the density property.

\begin{mainlemma}\label{mainlemma}
Let $f\in\omega^\omega$ and $\Gwf$ be as in \autoref{def:Sf} and let $\Aor$ be a Boolean algebra with a pfam $\mu$. If $\mu$ has the density property, then $\Aor$ is $\Rbf^f_\Gwf$-good.
\end{mainlemma}
\begin{proof}
Without loss of generality, we may assume that $\Aor$ is a complete Boolean algebra: if $\Aor'$ is a complete Boolean algebra such that $\Aor$ is a dense subalgebra of $\Aor'$, then we can extend $\mu$ to a finitely additive measure $\mu'$ on $\Aor'$ (see e.g.~\cite[Cor~3.3.6]{BRBR}). Since $\Aor$ is dense in $\Aor'$, $\mu'$ is a pfam and any set witnessing the density property of $\mu$ clearly witnesses the density property of $\mu'$.

So let $\Aor$ be a complete Boolean algebra and let $\mu$ be a pfam on $\Aor$ with the density property witnessed by some countable $S\subseteq\Aor\menos\{\zerobf_\Aor\}$. Assume that $\dot\psi$ is an $\Aor$-name for a member of $\Swf_f(\Gwf)$, wlog we can even assume that, for some $n^*<\omega$, $\Vdash_\Aor \forall\, n<\omega\colon |\dot\psi(n)|\leq h_{n^*}(n)$. 

Find a maximal antichain $A\subseteq\Aor\menos\{\zerobf_\Aor\}$ and some $g\colon A\to ([\omega]^{<\omega})^{n^*+1}$ such that, for all $a\in A$, \[a\Vdash\dot\psi{\upharpoonright}(n^*+1)=g_a.\]  
Fix $0<\varepsilon<1-2^{-(n^*+1)}$.
For $a\in A$, let $\la s_m^a:m<\omega\ra$ be an enumeration of the members of $S$ satisfying
\[\mu(a\wedge s_m^a)>\mu(s_m^a)(1-\varepsilon).\] 
This allows us to define a function $\psi_m^a$ with domain $\omega$ such that, for each $n<\omega$,
\[\psi_m^a(n)
   := 
    \begin{cases}
    \set{t\in 2^{f(n)}}{\mu(\| t\in\dot\psi(n)\|\wedge s_m^a)>\frac{1}{2^{n+1}}\mu(s_m^a)} & \textrm{if $n>n^*$,}\\
               %& \textrm{$b_0(m+1)<b_1(m+1)$ and $\dot I_m\subseteq I_n^*$,} \\ %&\textrm{  $(\dot I_m\subseteq I_n^*)$}\\
       g_a(n) & \textrm{if $n\leq n^*$.}\\
    \end{cases}
\]
We claim that $\psi_m^a\in\Swf_f(\Gwf)$ for any $a\in\Aor$ and $m<\omega$, in fact $\psi_m^a\in\Swf(2^f,h_{n^*+1})$. If $n\leq n^*$ then $\psi_m^a(n)=g_a(n)$ which is forced by $a$ to be  equal to $\dot\psi(n)$, so $|\psi_m^a(n)|\leq h_{n^*}(n)\leq h_{n^*+1}(n)$ and $|\psi_m^a(n)|<2^{f(n)-f(n-1)}$. Now assume $n\geq n^*+1$. We show that $|\psi^a_m(n)|\leq 2^{n+1}h_{n^*}(n)$, which implies $|\psi^a_m(n)|\leq h_{n^*+1}(n)<2^{f(n)-f(n-1)}$ by the properties of $h_{n^*+1}$ (see \autoref{def:Sf}). For each $t\in\psi^a_m(n)$ let $b_t:=\|t\in\dot\psi(n)\|\wedge s^a_m$. Since $2^{f(n)}$ is finite, we can find some $\delta>0$ such that $\mu(b_t)\geq2^{-(n+1)}\mu(s^a_m)+\delta$ for all $t\in\psi^a_m(n)$. Then, by~\cite[Prop.~1]{Kelley} applied to $\mu$ relative to $s^a_m$,
\[\inf\set{\frac{\iota(\Seq{b_{t}}{t \in c})}{|c|}}{c\in[\psi^a_m(n)]^{<\aleph_0}\menos \left\{\emptyset\right\}}\geq\frac{1}{\mu(s^a_m)}\left(\frac{1}{2^{n+1}}\mu(s^a_m)+\delta\right)>\frac{1}{2^{n+1}},\]
where $\iota(\la b_t:\, t\in c\ra)$ is the maximum size of a $d\subseteq c$ such that $\bigwedge_{t\in d} b_t\neq \zerobf_\Aor$. On the other hand, $\iota(\la b_t:\, t\in\psi^a_m(n)\ra)\leq h_{n^*}(n)$, otherwise some condition in $\Aor\menos\{\zerobf_\Aor\}$ would force $|\dot\psi(n)|>h_{n^*}(n)$, which is contradictory. Thus $2^{-(n+1)}|\psi^a_m(n)|<\iota(\la b_t:\, t\in\psi^a_m(n)\ra)\leq h_{n^*}(n)$, which implies the desired inequality.

Therefore, $H:=\set{\psi^a_m}{a\in A,\ m<\omega}$ is a countable subset of $S_f(\Gwf)$. We prove that this set witnesses the $\Rbf^f_\Gwf$-goodness for $\dot\psi$.
So let $t\in\bigcup_{m<\omega}2^{f(m)}$ and  $\sigma\in(2^{<\omega})^\omega$, and assume that $(t, \sigma)\nsqsubset^f\psi_m^a$ for all $a\in A$ and $m<\omega$. We prove that $\Vdash_\Aor (t, \sigma)\nsqsubset^f\dot\psi$, that is, 
\[\Vdash_\Aor\exists^\infty\, n\colon \sigma(n)\supseteq t \text{ and } \forall\, k<\omega\colon |t|<f(k)\leq|\sigma(n)| \imp \sigma(n){\upharpoonright}f(k)\notin \dot\psi(k).\]

Let $b\in \Aor\menos\{\zerobf_\Aor\}$ and $n_0<\omega$. It is enough to show that there are some $b'\leq b$ in $\Aor\menos\{\zerobf_\Aor\}$ and $n\geq n_0$ such that $b'$ forces the above for $n$.
Find $a\in A$ such that $a':=a\wedge b\neq \zerobf_\Aor$. By the density property of $\mu$, there is an $s\in S$ such that 
\[\mu(s\wedge a')>\mu(s)(1-\varepsilon).\]
Then $s=s_m^{a}$ for some $m<\omega$. Since $(t, \sigma)\nsqsubset^f\psi_m^{a}$, we can find some $n>n_0$  such that $\sigma(n)\supseteq t$ and, for all $k<\omega$, if $|t|<f(k)\leq|\sigma(n)|$, then  
\[\mu\Big(\|\sigma(n){\upharpoonright}f(k)\in\dot\psi(k)\|\wedge s\Big)\leq\frac{1}{2^{k+1}}\mu(s) \text{ when }k>n^*,\]
and $a\Vdash\sigma(n){\upharpoonright}f(k)\not\in\dot\psi(k)$ when $k\leq n^*$. 

Let $a^*:=s\wedge\bigvee \set{\|\sigma(n){\upharpoonright}f(k)\in\dot\psi(k)\|}{k>n^*,\, |t|<f(k)\leq|\sigma(n)|}$. Notice that
\[
    \mu(a^*) \leq\sum_{\mathclap{\substack{k>n^*\\ f(k)\leq|\sigma(n)|}}}\mu(s\wedge\|\sigma(n){\upharpoonright}f(k)\in \dot\psi(k)\|)
    \leq\sum_{\mathclap{\substack{k>n^*\\ f(k)\leq|\sigma(n)|}}}\frac{1}{2^{k+1}}\mu(s)
    <\sum_{k>n^*}\frac{1}{2^{k+1}}\mu(s)
    =\frac{\mu(s)}{2^{n^*+1}}.
\]
Hence $\mu(a^*)<\frac{\mu(s)}{2^{n^*+1}}$. Now consider $b':=s\wedge a'\smallsetminus a^*$ in $\Aor$. Observe that 
\[
    \mu(b') \geq\mu(s\wedge a')-\mu(a^*)
    >\mu(s)(1-\varepsilon)-\frac{\mu(s)}{2^{n^*+1}}
    =\mu(s)\bigg(1-\varepsilon-\frac{1}{2^{n^*+1}}\bigg)>0,
\]
so $b'>\zerobf_\Aor$, $b'\leq s\wedge a'$ and $b'\perp a^*$. Then $b'\Vdash\sigma(n){\upharpoonright}f(k)\not\in\dot\psi(k)$ for all $k>n^*$ such that $|t|<f(k)\leq|\sigma(n)|$. On the other hand, $b'\Vdash\sigma(n){\upharpoonright}f(k)\not\in \dot\psi(k)$ for all $k\leq n^*$ because $b'\leq a'$. Hence $b'\leq b$ is as desired.
\end{proof}

For an infinite cardinal $\kappa$, let $\Bor_\kappa$ be the complete Boolean algebra that adds $\kappa$-many random reals side by side. Although $\Bor_\kappa$ does not have a pfam with the density property when $\kappa \geq \aleph_1$, we have that: 

\begin{corollary}\label{cor:Bkappa}
    $\Bor_\kappa$ is $\Rbf^f_\Gwf$-good for any infinite cardinal $\kappa$.
\end{corollary}
\begin{proof}
    Any $\Bor_\kappa$-name of a slalom in $\Swf_f(\Gwf)$ depends on some countable support, so when restricting $\Bor_\kappa$ to this support (which is basically random forcing), we can use goodness and find the good countable set in the ground model.
\end{proof}

As a consequence, we have that the ``density property" is inessential for $\Rbf^f_\Gwf$-goodness.

\begin{theorem}\label{supermain}
    Let $f\in\omega^\omega$ and $\Gwf$ be as in \autoref{def:Sf}. Then, any Boolean algebra with a pfam is $\Rbf^f_\Gwf$-good.
\end{theorem}
\begin{proof}
    Let $\Aor$ be a Boolean algebra with a pfam. By \cite[Prop.~3.7]{Ka}, $\Aor$ can be completely embedded into the two step iteration $\Bor_\kappa\ast \Qnm$ for some infinite cardinal $\kappa$ and some $\Bor_\kappa$-name $\Qnm$ of a $\sigma$-centered poset. By \autoref{FSIgoodness},~\ref{cor:Bkappa} and \autoref{exm:densepfam}~\ref{it:famsc}, we have that $\Bor_\kappa\ast \Qnm$ is $\Rbf^f_\Gwf$-good. As a consequence, $\Aor$ is $\Rbf^f_\Gwf$-good.
\end{proof}

By \autoref{exm:densepfam} and \autoref{supermain}, we obtain:

\begin{corollary}\label{exm:Rfgood}
The following posets are $\Rbf^f_\Gwf$-good.
\begin{enumerate}[label=\rm (\roman*)]
    \item\label{exm:Rfgoodi} Any subalgebra of random forcing.
    \item\label{exm:Rfgoodii} $(\Eor^h_b)^N$ and $(\Bor^h_b)^N$ for any transitive model $N$ of (a large enough fragment of) $\thzfc$, when $b,h\in\omega^\omega\cap N$ satisfy $\sum_{i\in\omega}\frac{h(i)}{|b(i)|}<\infty$. 
    \item\label{exm:Rfgoodiii} Any $\sigma$-centered poset.
\end{enumerate}
\end{corollary}

We are now ready to prove the main technical device of this paper.

\begin{theorem}\label{thm:Paddsn}
Let $\theta_0\leq \theta$ be uncountable regular cardinals, $\lambda=\lambda^{<\theta_0}$ a cardinal and let $\pi=\lambda\delta$ (ordinal product) for some ordinal $0<\delta<\lambda^+$. Assume $\theta \leq \lambda$ and $\cf(\pi)\geq \theta_0$. If $\Por$ is a FS iteration of length $\pi$ of non-trivial $\theta_0$-cc $\theta$-$\Rbf^f_\Gwf$-good posets of size ${\leq}\lambda$, 
then $\Por$ forces $\Cbf_{[\lambda]^{<\theta}}\leqT\SNwf$, in particular, 
$\add(\SNwf)\leq\theta$ and $\lambda\leq\cof(\SNwf)$.
\end{theorem}
\begin{proof}
Let $\set{\gamma_\alpha}{\alpha<\pi_0}$ be an increasing enumeration of $0$ and all limit ordinals smaller than $\pi$, and let $I_\alpha:=[\gamma_\alpha,\gamma_{\alpha+1})$. Note that $\la I_\alpha:\, \alpha<\pi_0\ra$ is an interval partition of $\pi$ of sets with order type $\omega$. Even more, $\omega\xi<\pi$ for all $\xi<\pi$, so $\pi_0$ and $\pi$ have the same order type, i.e.\ $\pi_0=\pi$. We still use $\pi_0$ to distinguish indexes corresponding to (zero and) limit ordinals. Note that $|\pi|=\lambda$.

Let $\Seq{Z^0_\beta}{\beta<\lambda}$ be a partition of $\pi_0$ 
into cofinal subsets of $\pi_0$ of size $\lambda$. We can find this partition because $\pi$ is a multiple of $\lambda$. For $\beta<\lambda$, let $Z_\beta:=\bigcup_{\alpha\in Z^0_\beta}I_\alpha$.

For $\beta<\lambda$ find $\set{\dot f_\alpha}{\alpha\in Z^0_\beta}$ such that each $\dot f_\alpha$ is a $\Por_{\gamma_\alpha}$-name for an increasing function in $\omega^\omega$ and $\Por_\pi$ forces that $\set{\dot f_\alpha}{\alpha\in Z^0_\beta}$ is the collection of all increasing functions in $\omega^\omega$. This is possible by book-keeping or counting arguments because $\lambda^{<\theta_0}=\lambda$, $\Por$ is $\theta_0$-cc and $\cf(\pi)\geq \theta_0$, which in fact implies that $\cfrak$ is forced to be $\lambda$. Here $f_0 := f$.

For $\alpha< \pi_0$ denote by $\sigma_\alpha\in(2^{<\omega})^\omega\cap V_{\gamma_{\alpha+1}}$ the Cohen real over $V_{\gamma_\alpha}$ such that  $\hgt_{\sigma_\alpha}=f_\alpha$. Hence $2^\omega\cap V_{\gamma_\alpha}\subseteq[\sigma_\alpha]_\infty$.  Denote its $\Por_{\gamma_{\alpha+1}}$-name by $\dot\sigma_\alpha$.

In the final extension, for $\beta<\lambda$, define $X_\beta:=\bigcap_{\alpha\in Z^0_\beta}[\sigma_\alpha]_\infty$, which is clearly in $\SNwf$. As indicated in the discussion starting this section, it is enough to show in the final extension that, for any $\tau \in 2^f$, $\left|\set{\beta<\lambda}{X_\beta\subseteq\bigcup_{i<\omega}[\tau(i)]}\right|<\theta$.

Work in $V$. Let $\tau$ be a nice $\Por_\pi$-name of a member of $2^f$ and let $\dot\psi$ be a (nice) $\Por_\pi$-name of the slalom defined by $\dot\psi(k):=\{\tau(k)\}$. 
For every $t\in \bigcup_{m<\omega}2^{f(m)}$, $(t,\sigma_\alpha)$ is a Cohen real in $\left(\bigcup_{i<\omega}2^{f(i)}\right)\times 2^{f_\alpha}$ over $V_{\gamma_\alpha}$, so by \autoref{Comgood} $\Por_\pi$ forces that $\set{(t,\dot \sigma_\alpha)}{\alpha<\pi_0}$ is $\theta$-$\Rbf^f_\Gwf$-unbounded. It follows that $\Por_\pi$ forces that \[\set{(t,\dot\sigma_\alpha)}{t\in\bigcup_{m<\omega}2^{f(m)},\ \alpha<\pi_0} \text{ is $\theta$-$\Rbf^f_\Gwf$-unbounded,}\] 
which implies that $\left|\set{\alpha<\pi_0}{\exists\, t\in\bigcup_{m<\omega}2^{f(m)}\colon (t,\dot\sigma_\alpha)\sqsubset^f \dot\psi}\right| <\theta$. Since $\Por_\pi$ has the $\theta$-cc, we can find some $A\in[\pi_0]^{<\theta}$ in the ground model which is forced to contain this set. Therefore $B:=\set{\beta<\lambda}{Z^0_\beta\cap A\neq\emptyset}$ has size ${<}\theta$.

It suffices to prove that, for any $\beta\in\lambda\menos B$, $\Por_\pi$ forces that 
\[\dot X_\beta\nsubseteq\bigcup_{i<\omega}[\tau(i)].\]
Let $\beta\in \lambda\menos B$. Since $\tau$ is a name of a real, there is some $\xi\in \pi_0\menos Z^0_\beta$ such that $\tau$ is a $\Por_{\gamma_\xi}$-name. Let $\dot T$ be a $\Por_{\gamma_\xi}$-name of the well-pruned tree such that $[\dot T]=2^\omega\menos\bigcup_{k<\omega}[\tau(k)]$, and note that this tree is forced to be perfect. Since $\Por_{\gamma_{\xi+1}}$ adds a Cohen real over $V_{\gamma_\xi}$, we pick a $\Por_{\gamma_{\xi+1}}$-name $\dot c$ of a Cohen real in $[\dot T]$ over $V_{\gamma_\xi}$. To conclude the proof, it is enough to show that $\Por_\pi$ forces $\dot c\in \dot X_\beta = \bigcap_{\alpha\in Z^0_\beta}[\dot \sigma_\alpha]_\infty$.

Pick $\alpha\in Z^0_\beta$, so $\alpha\neq\xi$ by the choice of $\xi$. We split into two cases. If $\xi<\alpha$ then it is clear that $\Por_\pi$ forces $\dot c\in[\dot\sigma_\alpha]_\infty$ because $\dot\sigma_{\alpha}$ is forced to be Cohen over $V_{\gamma_\alpha}$, which contains $2^\omega \cap V_{\gamma_{\xi+1}}$. 

In the case $\alpha<\xi$, since $\beta\notin B$, $Z^0_\beta\cap A =\emptyset$, so $\alpha\notin A$, which implies that $\Por_\pi$ forces
\[\forall\, t\in \bigcup_{m<\omega}2^{f(m)}\colon (t,\dot\sigma_\alpha)\nsqsubset^f \dot\psi.\]
Therefore, by \autoref{magiclem}, $\Por_{\gamma_\xi}$ forces that $\set{\dot\sigma_\alpha(n)\in \dot T}{n\geq n_0}$ is dense in $\dot T$ for all $n_0<\omega$, which implies that $\Por_{\gamma_{\xi+1}}$ forces that $\dot c\in [\dot\sigma_\alpha]_\infty$.
\end{proof}

%\begin{remark}
    The proof of \autoref{thm:Paddsn} actually shows that $\Por$ forces $\Cbf_{[\lambda]^{<\theta}} \leqT \Sbf^f$ where
    %\begin{align*}
        $\Sbf^f  := \la \SNwf, 2^f, \sqsubseteq\ra$ 
    %    O^f  & := \set{\bigcup_{i<\omega}[t(i)]}{t\in 2^f}.
    %\end{align*} 
    and $A\sqsubseteq \tau$ iff $A\subseteq \bigcup_{i<\omega}[\tau(i)]$.
    It is clear that $\Sbf^f \leqT \SNwf$. Moreover, for Yorioka ideals, we obtain:
%\end{remark}

\begin{lemma}\label{SfYorio}
    Let $f,g\colon \omega\to\omega$ be increasing and assume that, for some $0<k<\omega$, $f(i)\leq g(i^k)$ for all but finitely many $i<\omega$. Then $\Sbf^f\leqT \Iwf_g$.
    
    In particular, if $f(i) \leq \frac{(i+1) (i+2) (2i+3)}{6}$ for all (but finitely many $i<\omega$) then $\Sbf^f\leqT \Iwf_\id$.
\end{lemma}
\begin{proof}
    First note that the hypothesis is equivalent to say that $g\ll g'$ implies $f\leq^* g'$ for any increasing $g'\in\omega^\omega$. Clearly $\Sbf^f\leqT \la\Iwf_g, 2^f,\sqsubseteq\ra$ so, to show $\la\Iwf_g, 2^f, \sqsubseteq \ra \leqT \Iwf_g$, it is enough to prove that any member of $\Iwf_g$ is contained in $\bigcup_{i<\omega}[\tau(i)]$ for some $\tau\in 2^f$. If $\sigma\in (2^{<\omega})^\omega$ and $\hgt_\sigma \gg g$, then $\hgt_\sigma \geq^* f$, so we can find some $\tau \in 2^f$ such that $\forall^\infty\, i<\omega\colon \tau(i) = \sigma_i\frestr f(i)$. It is clear that $[\sigma]_\infty \subseteq \bigcup_{i<\omega}[\tau(i)]$.
\end{proof}

\begin{corollary}\label{cor:addsn}
    With the assumptions of \autoref{thm:Paddsn}, $\Por$ forces that $\Cbf_{[\lambda]^{<\theta}} \leqT \Iwf_g$ for any increasing $g\in\omega^\omega$ such that, for some $0<k<\omega$, $f(i)\leq g(i^k)$ for all but finitely many $i<\omega$.

    In particular, if $f(i) = \frac{(i+1) (i+2) (2i+3)}{6}$ for any $i<\omega$ (which is allowed), then $\Por$ forces $\Cbf_{[\lambda]^{<\theta}} \leqT \Iwf_g$ for any $g\in\omega^\omega$, so $\add(\Iwf_g)\leq\theta$ and $\lambda\leq\cof(\Iwf_g)$.
\end{corollary}
\begin{proof}
    Note that any $g\in\omega^\omega$ in the generic extension lives in the intermediate extension, so we can apply \autoref{thm:Paddsn} and \autoref{SfYorio} to the remaining part of the iteration.
\end{proof}

\section{Applications~I}\label{SecConstr}
 
We prove~\autoref{4SN} (\autoref{thm4SN}), i.e.\ the consistency that the four cardinal characteristics associated with $\SNwf$ are pairwise different, along with some constellations of Cicho\'n's diagram.

We fix some notation. 
\begin{enumerate}[label = \rm (\arabic*)]
    \item For an infinite cardinal $\kappa$ denote by $\Fn_{<\kappa}(I,J)$ the poset of partial functions from $I$ into $J$ with domain of size ${<}\kappa$, ordered by $\supseteq$.
    \item We denote by $\Loc$ the standard $\sigma$-linked poset that adds a generic slalom $\varphi_{\gen} \in \Swf(\omega,\id)$ such that $x \in^* \varphi_{\gen}$ for any $x\in \omega^\omega$ in the ground model. (See e.g.~\cite[Def.~2.4]{BCM}.)
    \item $\Bor$ denotes random forcing.
    \item $\Dor$ denotes Hechler's poset for adding a dominating real over the ground model. Recall that it is $\sigma$-centered.
\end{enumerate}

Concerning some of the hypotheses that appear in our theorems, 
we use assumptions of the form $\cof([\lambda]^{<\theta}) = \lambda$, which implies $\cf(\lambda)\geq\cf(\theta)$ (because
$\cf(\cof([\lambda]^{<\theta})) \geq \add([\lambda]^{<\theta}) = \cf(\theta)$); another type of hypothesis has the form $\dfrak_{\cf(\lambda)} = \dfrak_\lambda = \cof\left(([\lambda]^{<\theta})^\lambda\right) = \nu$, which is easy to force. Indeed, ZFC proves that, for infinite cardinals $\theta\leq\lambda$ and $\mu$,
\[ \la\cf(\lambda)^{\cf(\mu)},\leq\ra \leqT \la\cf(\lambda)^{\mu},\leq\ra \eqT \la \lambda^\mu, \leq\ra \leqT \Cbf_{[\lambda]^{<\theta}}^\mu \leqT ([\lambda]^{<\theta})^\mu,\]
so
\[\dfrak(\cf(\lambda)^{\cf(\mu)}) \leq \dfrak(\lambda^\mu) \leq \cov\left(([\lambda]^{<\theta})^{\mu}\right) \leq \cof\left(([\lambda]^{<\theta})^{\mu}\right) \leq (\cof([\lambda]^{<\theta}))^\mu \leq  (\lambda^{<\theta})^\mu \leq (2^\lambda)^\mu.\]
%Also note that $\lambda\leq\mu$ implies $\cof\left(([\lambda]^{<\theta})^{\mu}\right) = (2^\lambda)^\mu = 2^\mu$.
Hence, if $\cf(\lambda)^{<\cf(\lambda)} = \cf(\lambda)$ and $\nu^\lambda=\nu$
then, after forcing with $\Fn_{<\cf(\lambda)}(\nu,2)$,
\[ \dfrak_{\cf(\lambda)}=\dfrak_{\lambda}=
\cof\left(([\lambda]^{<\theta})^{\lambda}\right)=2^{\lambda}=\nu.\]
We can further force $\cof\left(([\lambda]^{<\theta})^{\lambda}\right) < 2^{\lambda}$, but there is no known natural way to do this when $\theta<\lambda$.

\begin{lemma}\label{lem:hyp}
    Let $\theta \leq \lambda < \nu \leq\vartheta$  be infinite cardinals. Assume that $\cf(\lambda)^{<\cf(\lambda)} = \cf(\lambda)$ and $\vartheta^{\lambda} = \vartheta$.
    \begin{enumerate}[label = \rm  (\alph*)]
        \item\label{hyp:a} Assume that there is some cardinal $\theta_0<\cf(\theta)$ such that $\theta_0<\cf(\lambda)$ and $\theta_0^{<\theta_0} = \theta_0$.  If $\nu^\lambda = \nu$ then $\Fn_{<\cf(\lambda)}(\nu,2)\times \Fn_{<\theta_0}(\vartheta,2)$ preserves cofinalities and forces
    \[\dfrak_{\cf(\lambda)}=\dfrak_{\lambda}=
      \cof\left(([\lambda]^{<\theta})^{\lambda}\right)=\nu \leq 2^{\theta_0} = 2^\lambda =\vartheta.\]

      \item\label{hyp:b} If $\theta = \lambda$ and $\nu$ are regular, then  there is a cofinality preserving poset forcing
      \[\dfrak_{\lambda}=
      \cof\left(([\lambda]^{<\lambda})^{\lambda}\right)=\nu \leq 2^\lambda =\vartheta.\]
    \end{enumerate}
\end{lemma}
\begin{proof}
\ref{hyp:a}: We only explain why $\dfrak_{\cf(\lambda)}=\dfrak_{\lambda}=
      \cof\left(([\lambda]^{<\theta})^{\lambda}\right)=\nu$ is forced. It is clear that this is forced by $\Fn_{<\cf(\lambda)}(\nu,2)$. On the other hand, $\cf(\lambda)$-cc posets preserve the value of $\dfrak_{\cf(\lambda)}$, and $\theta_0^+$-cc posets preserve the values of both $\cof([\lambda]^{<\theta})$ and $\cof\left(([\lambda]^{<\theta})^{\lambda}\right)$ (see e.g.~\cite[Lemma~6.6]{CM23}). Since $\Fn_{<\theta_0}(\vartheta,2)$ has the $\theta_0^+$-cc (in the $\Fn_{<\cf(\lambda)}(\nu,2)$-extension), the result follows.

\ref{hyp:b}: In ZFC, when $\lambda$ is regular, $[\lambda]^{<\lambda} \eqT \lambda$, so $([\lambda]^{<\lambda})^\lambda \eqT \la \lambda^\lambda, \leq\ra$, hence $\dfrak_{\lambda} =     \cof\left(([\lambda]^{<\lambda})^{\lambda}\right)$. Observe that the ${<}\lambda$ support iteration of length $\vartheta+\nu$ (ordinal sum) of the Hechler poset adding a dominating function in $\lambda^\lambda$ does the job. See details in~\cite[Subsec.~4.2]{BBSM}.
\end{proof}

Hence, the hypothesis of the form $\dfrak_{\cf(\lambda)}=\dfrak_{\lambda}=
\cof\left(([\lambda]^{<\theta})^{\lambda}\right)=\nu$ is feasible for the results of this section. Although we do not indicate the value forced to $2^\dfrak$, it equals the ground-model value of $2^\lambda$ (also with $\lambda=\theta_6$ or $\theta_7$). This value can be assumed larger thanks to the previous lemma, with the possible exception that $\theta_0 = \aleph_0$ in the case of~\ref{hyp:a}: we will have $\cfrak = 2^\lambda =\vartheta$, so we cannot force $\cfrak = \lambda$ afterward with a ccc poset (which is the intention in all our results). This limitation is present when we intend to force $\add(\Nwf) = \aleph_1$ (i.e.\ $\theta_1 = \aleph_1$).

We proceed to present and prove the theorems of this section. From now on, fix $f_0\in\baireincr$ defined by $f_0(i) = \frac{(i+1) (i+2) (2i+3)}{6}$.
The following result shows the simpler forcing construction we could find to force that the four cardinal characteristics associated with $\SNwf$ are pairwise different.

\begin{theorem}\label{thm4SN}
Let $\kappa\leq\lambda$ be uncountable cardinals with $\kappa$ regular and assume that $\lambda^{\aleph_0}=\lambda$ and $\cof([\lambda]^{<\kappa}) = \lambda$.  
Then, there is a ccc poset forcing 
\begin{multicols}{2}
\begin{enumerate}[label = \rm (\alph*)]
    \item $\cfrak=\lambda < \cof(\SNwf)$, 
    \item $\Nwf \eqT \Iwf_f \eqT\Cbf_{[\lambda]^{<\aleph_1}} \leqT \SNwf$ for all $f\in\baireincr$,
    \item $\Cbf^\perp_\Nwf\eqT\Cbf^\perp_{\SNwf}\eqT\Cbf_\Mwf \eqT\Cbf_{[\lambda]^{<\kappa}}$,
    \item $\la\cf(\lambda)^{\cf(\lambda)},\leq\ra \leqT \SNwf \leqT \Cbf^{\lambda}_{[\lambda]^{<\aleph_1}}$.
\end{enumerate}
\end{multicols}
 In particular, it is forced that 
\begin{align*}
    \add(\Nwf)=\add(\SNwf)= \add(\Iwf_\id)=\aleph_1 \leq \cov(\Nwf)=\non(\Mwf)=\cov(\SNwf)=\kappa \leq \\
    \leq \cov(\Mwf)=\non(\SNwf)=\cfrak=\lambda < \cof(\SNwf).
\end{align*}
If, additionally, we assume in the ground model that %$\cf(\lambda)^{<\cf(\lambda)} = \cf(\lambda)$ and $\nu^\lambda=\nu$, 
$\dfrak_{\cf(\lambda)} = \cof\left(([\lambda]^{<\aleph_1})^{\lambda}\right)=\nu$ 
then %the two-step iteration of $\Fn_{<\cf(\lambda)}(\nu,2)$ with 
the previous ccc poset 
%preserves cofinalities and 
forces, in addition, $ \dfrak_{\cf(\lambda)} = \dfrak_\lambda = \cof(\SNwf) = \cof\left(([\lambda]^{<\aleph_1})^{\lambda}\right) =\nu$. (See~\autoref{Figthm4SN}.)
\end{theorem}

\begin{figure}[ht]
\begin{center}
  \includegraphics[width=\linewidth]{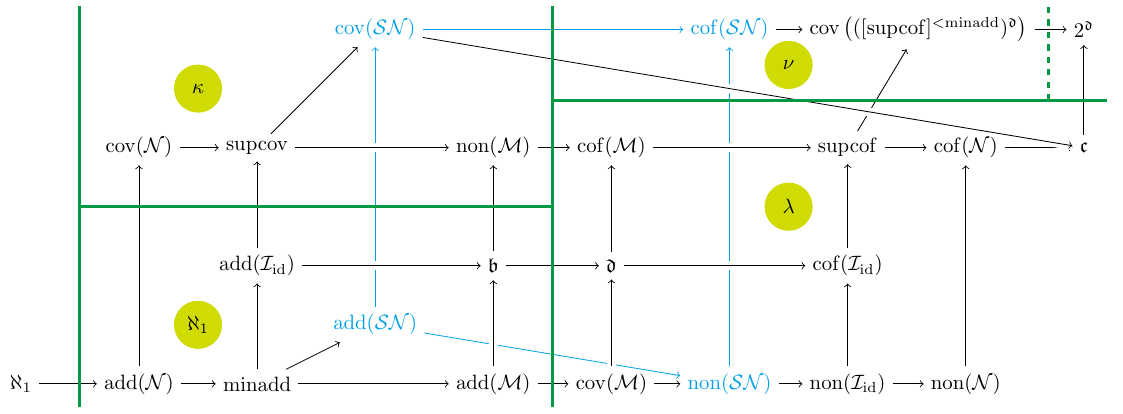}
  \caption{Constellation forced in~\autoref{thm4SN}. The dotted green line indicates that $2^\dfrak$ may be larger (depending on the value of $2^{\lambda}$ in the ground model).}
  \label{Figthm4SN}
\end{center}
\end{figure}

\begin{proof}
Let $\Por:=\Por_\lambda=\la \Por_\alpha,\Qnm_\alpha:\, \alpha<\lambda\ra$ be a FS iteration such that any $\Qnm_\alpha$ is a $\Por_\alpha$-name of a poset of the form $\Bor^{\dot N_\alpha}$, where $\dot N_\alpha$ is a $\Por_\alpha$-name of a transitive model of (a large enough fragment of) ZFC of size ${<}\kappa$. We also ensure, by a book-keeping argument, that in the final extension any set of size ${<}\kappa$ of Borel codes of Borel measure zero sets is included in some $N_\alpha$ (by using that $\cof([\lambda]^{<\kappa})=\lambda$, which is preserved in any ccc forcing extension). %, which is preserved in any intermediate extension). 
It is clear that $\Por$ forces $\cfrak=\lambda$.

%It is well-known~\cite{KST} that $\Por$ forces $\Lc^*\eqT\Cbf_{[\lambda]^{<{\aleph_1}}}$, and $\Cbf^\perp_\Nwf\eqT\Cbf_\Mwf\eqT\Ed\eqT\Cbf_{[\lambda]^{<\kappa}}$ 

Notice that the iterands are $\Rbf^{f_0}_\Gwf$-good (for any suitable choice of $\Gwf$) and $\Lc^*$-good by~\autoref{exm:Rfgood}~\ref{exm:Rfgoodi} and~\autoref{ExmPrs}~\ref{ExmPrsd}, respectively. Hence, by applying~\autoref{thm:Paddsn},~\autoref{cor:addsn} and~\autoref{Comgood}, $\Por$ forces that $\Cbf_{[\lambda]^{<\aleph_1}}$ is Tukey-below $\SNwf$, $\Iwf_f$ for all $f\in\baireincr$, and below $\Lc^* \eqT\Nwf$. Moreover, all iterands are of size ${<}\kappa$, so by~\autoref{smallgoodness} they are $\kappa$-$\Rbf^*$-good, where $\Rbf^*$ is a Prs Tukey equivalent with $\Cbf_\Mwf$ (see e.g.~\cite[Ex.~1.21~(1)]{CM22}) so $\Por$ forces $\Cbf_{[\lambda]^{<\kappa}} \leqT \Cbf_\Mwf$ by \autoref{Comgood}.

On the other hand, it is clear that $\Nwf$ and $\Iwf_f$ (for all $f\in\baireincr$) are Tukey-below $\Cbf_{[\R]^{<\aleph_1}}$ (by \autoref{lem:TukeyCtheta}~\ref{Ctheta:a} because both have Borel bases and uncountable additivities), so $\Nwf\eqT\Iwf_f \eqT \Cbf_{[\lambda]^{<\aleph_1}}$ in the $\Por$-extension. The hypothesis of \autoref{itsmallsets} holds for $\theta=\kappa$ and $\Rbf = \Cbf^\perp_{\Bwf(2^\omega)\cap\Nwf}$ (due to the random reals added along the iteration), so
$\Por$ forces $\Cbf^\perp_\Nwf\leqT\Cbf_{[\lambda]^{<\kappa}}$ (note that $\Cbf_{\Bwf(2^\omega)\cap\Nwf}\eqT \Cbf_{\Nwf}$). Therefore $\Cbf_\Mwf\eqT\Cbf^\perp_\Nwf\eqT\Cbf_{[\lambda]^{<\kappa}}$ is forced (recall that $\Cbf_\Mwf\leqT\Cbf^\perp_\Nwf$ in ZFC).

As the iterands of the FS iteration that determine $\Por$ have  
precaliber $\kappa$, by~\autoref{thm:precaliber} $\Por$ forces $\Cbf_{[\lambda]^{<{\kappa}}}\leqT\Cbf_\SNwf^\perp$, and thus $\Cbf_{[\lambda]^{<{\kappa}}}\eqT\Cbf_\SNwf^\perp$ because
$\Cbf^\perp_{\SNwf}\leqT\Cbf^\perp_{\Nwf}$ (in ZFC). 
Next, %notice that $\Por$ forces that $\minadd=\aleph_1$ and $\supcof=\dfrak=\lambda$ because $\add(\Nwf)\leq\minadd\leq\add(\SNwf)$ and $\cov(\Mwf)\leq\non(\SNwf)\leq\supcof\leq\cof(\Nwf)$ (see~\autoref{Cichonwith_SN}). 
%Then, 
by employing~\autoref{new_upperb}, $\Por$ forces that $\SNwf\leqT\Cbf_{[\supcof]^{<\minadd}}^\dfrak=\Cbf_{[\lambda]^{<\aleph_1}}^{\lambda}$. On the other hand, by \autoref{cor:lowSN}, $\Por$ forces $\la\cf(\lambda)^{\cf(\lambda)},\leq\ra \leqT \SNwf$ and $\lambda<\cof(\SNwf)$.

For the second part of the theorem, recall that %, in the ground model, 
% \[ \la\cf(\lambda)^{\cf(\lambda)},\leq\ra \leqT \la\cf(\lambda)^{\lambda},\leq\ra \eqT \la \lambda^\lambda, \leq\ra \leqT \Cbf_{[\lambda]^{<\aleph_1}}^\lambda \leqT ([\lambda]^{<\aleph_1})^\lambda,\]
% so
% \[\dfrak_{\cf(\lambda)} \leq \dfrak_\lambda \leq \cov\left(([\lambda]^{<\aleph_1})^{\lambda}\right) \leq \cof\left(([\lambda]^{<\aleph_1})^{\lambda}\right) \leq \cof([\lambda]^{<\aleph_1})^\lambda = 2^\lambda.\]
% Then, after forcing with $\Fn_{<\cf(\lambda)}(\nu,2)$,
%\[\dfrak_{\cf(\lambda)}=\dfrak_{\lambda}=\cof\left(([\lambda]^{<\aleph_1})^{\lambda}\right)= \nu,\]
the values of $\dfrak_{\cf(\lambda)}$ and $\cof\left(([\lambda]^{<\aleph_1})^{\lambda}\right)$ are preserved after forcing with any ccc poset. Therefore, after forcing with $\Por$, $\cof(\SNwf)=\nu$. %Hence, we are done. 
\end{proof}

It is unclear what is the value of $\bfrak$ in the previous construction. Still, modifications allow two possibilities: $\bfrak=\kappa$ can be forced in addition by including, in the iterations, posets of the form $\Dor^N$ where $N$ is a transitive model of ZFC of size ${<}\kappa$, with a book-keeping argument similar to the one used for random forcing; and $\bfrak=\aleph_1$ can be forced in addition by iterating with fams as in~\cite{ShCov,KST}, see also~\cite{uribethesis,CMU} and \autoref{KST4SN}.

For the results that follow, %not only do we derive both cases above from a single and less complicated construction,\footnote{For $\theta_4 = \kappa$ and $\theta_4 = \aleph_1$ in \autoref{thm:4SN}, respectively.} but 
we force more different simultaneous values in Cicho\'n's diagram. For this purpose, we employ the method of matrix iterations with ultrafilters from~\cite{BCM},
%Before proving this result (\autoref{thm:4SN}), 
which we review as follows.

\begin{definition}[{\cite[Def.~2.10]{BCM}}]\label{Defmatsimp}
A~\emph{simple matrix iteration} of ccc posets is composed of the following objects: 
\begin{enumerate}[label=\rm (\Roman*)]
    \item ordinals $\gamma$ (height) and $\pi$ (length);
    \item a function $\Delta\colon \pi\to\gamma$; 
    \item a sequence of posets $\la \Por_{\alpha,\xi}:\, \alpha\leq \gamma,\ \xi\leq \pi\ra$ where $\Por_{\alpha,0}$ is the trivial poset for any $\alpha\leq \gamma$;
    \item for each $\xi<\pi$, $\dot{\Qor}_\xi$ is a 
     $\Por_{\Delta(\xi),\xi}$-name of a poset such that $\Por_{\gamma,\,\xi}$ forces it to be ccc;
    \item $\Por_{\alpha,\xi+1}=\Por_{\alpha,\,\xi}\ast\Qnm_{\alpha,\xi}$, where  
\[\dot{\mathbb{Q}}_{\alpha,\xi}:=
\begin{cases}
    \Qnm_\xi & \textrm{if  $\alpha\geq\Delta(\xi)$,}\\
    \{0\}         & \textrm{otherwise;}
\end{cases}\]
    \item for $\xi$ limit, $\Por_{\alpha,\xi}=\limdir_{\eta<\xi}\Por_{\alpha,\eta}$. 
\end{enumerate}

It is known that $\alpha\leq\beta\leq\gamma$ and $\xi\leq\eta\leq\pi$ imply $\Por_{\alpha,\xi}\subsetdot\Por_{\beta,\eta}$, see e.g.~\cite{B1S} and \cite[Cor.~4.31]{CM}.  If $G$ is $\Por_{\gamma,\pi}$-generic
over $V$, we denote $V_{\alpha,\xi}= [G\cap\Por_{\alpha,\xi}]$ for all $\alpha\leq\gamma$ and $\xi\leq\pi$. 
%We often refer to $\Por_{\gamma, \pi}$ as the simple matrix iteration $\Seq{\Por_{\alpha, \xi}, \Qnm_{\alpha, \xi}}{ \alpha\leq\gamma, -1\leq\xi\leq\pi}$. 
\end{definition}

% \begin{definition}[{Blass and Shelah~\cite{B1S}}]\label{Defmatsimp}
% A~\emph{simple matrix iteration} of ccc posets is composed of the following objects: 
% \begin{enumerate}[label=\rm (\Roman*)]
%     \item ordinals $\gamma$ (height) and $\pi$ (length)
      
%     \item for $\alpha\leq\gamma$, a FS iteration $\Por_{\alpha,\pi}=\Seq{\Por_{\alpha,\xi}, \Qnm_{\alpha,\pi}}{\xi<\pi}$ such that 
    
%     \begin{enumerate}[label=\rm (III-\arabic*)]
%         \item $\Qnm_{\alpha, -1}=\Por_{\alpha, 0}=\Cor_{\alpha}$ ($\alpha$-many
% Cohen reals);
        
%         \item for $1\leq\xi<\pi$ there is a $\Qnm_{\xi}^*\in V_{\Delta(\xi),\xi}$ such that $\Por_{\alpha, \xi+1}=\Por_{\alpha, \xi}\ast\Qnm_{\alpha, \xi}$ where 
%     \[\Qnm_{\alpha,\xi}=\left\{\begin{array}{ll}
%           \Qnm_\xi^*  & \text{if $\alpha\geq\Delta(\xi)$,}\\
%           \{0\} & \text{otherwise.}
%           \end{array}\right.\]
%     \end{enumerate}
%     \item for $\xi$ limit, $\Por_{\alpha, \xi}=\limdir_{\eta<\xi}\Por_{\alpha, \eta}$.
% \end{enumerate}  
% We often refer to $\Por_{\gamma, \pi}$ as the simple matrix iteration $\Seq{\Por_{\alpha, \xi}, \Qnm_{\alpha, \xi}}{ \alpha\leq\gamma, -1\leq\xi\leq\pi}$. 
% \end{definition}

\begin{lemma}[{\cite[Lemma~5]{BrF}, see also~\cite[Cor.~2.6]{mejiavert}}]\label{realint}
 Assume that $\Por_{\gamma, \pi}$ is a simple matrix iteration as in~\autoref{Defmatsimp} with $\cf(\gamma)>\omega$. 
Then, for any $\xi\leq\pi$,
\begin{enumerate}[label=\rm (\alph*)]
    \item  $\Por_{\gamma,\xi}$ is the direct limit of $\Seq{\Por_{\alpha,\xi}}{\alpha<\gamma}$, and
    \item if $\dot{f}$ is a $\Por_{\gamma,\xi}$-name of a function from $\omega$ into $\bigcup_{\alpha<\gamma}V_{\alpha,\xi}$ then $\dot{f}$ is forced to be equal to a $\Por_{\alpha,\xi}$-name for some $\alpha<\gamma$.  In particular, the reals in $V_{\gamma,\xi}$ are precisely the reals in $\bigcup_{\alpha<\gamma}V_{\alpha,\xi}$.
\end{enumerate}  
\end{lemma}

\begin{theorem}[{\cite[Thm.~5.4]{CM}}]\label{matsizebd}
%Let $\Rbf$ be a gPrs coded in $V$ and 
Let $\Por_{\gamma, \pi}$ be a simple matrix iteration as in~\autoref{Defmatsimp}. Assume that, for any $\alpha<\gamma$, there is some $\xi_\alpha<\pi$ such that $\Por_{\alpha+1,\xi_\alpha}$ adds a Cohen real $\dot{c}_\alpha\in X$ %that is $\Rbf$-unbounded 
over $V_{\alpha,\xi_\alpha}$. 
   Then, for any $\alpha<\gamma$, $\Por_{\alpha+1,\pi}$ forces that $\dot{c}_{\alpha}$ is Cohen %$\Rbf$-unbounded 
   over $V_{\alpha,\pi}$. 

   In addition, if $\cf(\gamma)>\omega_1$ and $f\colon \cf(\gamma)\to\gamma$ is increasing and cofinal, then  $\Por_{\gamma,\pi}$ forces that $\{\dot{c}_{f(\zeta)}:\, \zeta<\cf(\gamma)\}$ is a $\cf(\gamma)$-$\Cbf_\Mwf$-unbounded family. In particular, $\Por_{\gamma,\pi}$ forces $\gamma\leqT \Cbf_\Mwf$ and $\non(\Mwf)\leq\cf(\gamma)\leq\cov(\Mwf)$. 
\end{theorem}

In~\cite{mejiavert}, the third author introduced the notion of \emph{ultrafilter-linkedness} (abbreviated uf-linkedness). %(see~\cite[Def.~3.24]{mejiavert}). 
He proved that no $\sigma$-uf-linked poset adds dominating reals and that such a poset preserves a certain type of mad (maximal almost disjoint) families. %He also proved that Fr-linked preserve $\Dbf$-unbounded families (see~\cite[Thm.~3.30]{mejiavert}). 
These results where improved in~\cite{BCM}, which motivated the construction of matrix iterations of ${<}\theta$-uf-linked posets to improve the separation of the left-hand side of Cicho\'n's diagram from~\cite{GMS} by including
$\cov(\Mwf)<\dfrak=\non(\Nwf)=\cfrak$.

The following notion formalizes the matrix iterations with ultrafilters from~\cite{BCM}. The property ``${<}\theta$-uf-linked" is used as a black box, i.e.\ there is no need to review its definition, but it is enough to present the relevant examples and facts (with proper citation).

\begin{definition}[{\cite[Def.~4.2]{BCM}}]\label{Defmatuf}
Let $\theta\geq\aleph_1$ and let $\Por_{\gamma, \pi}$ be a simple matrix iteration as in~\autoref{Defmatsimp}. Say that $\Por_{\gamma, \pi}$ is a~\emph{${<}\theta$-uf-extendable matrix iteration} if for each $\xi<\pi$, $\Por_{\Delta(\xi),\xi}$ forces that $\Qnm_\xi$ is a ${<}\theta$-uf-linked poset.
\end{definition}

\begin{example}\label{exm:ufl}
The following are the instances of ${<}\theta$-uf-linked posets that we use in our applications.
\begin{enumerate}[label=\rm (\arabic*)]

    \item Any poset of size $\mu<\theta$ is ${<}\theta$-uf-linked. In particular, Cohen
forcing is $\sigma$-uf-linked (i.e.\ ${<}\aleph_1$-uf-linked), see~\cite[Rem.~3.3~(5)]{BCM}.
    
    \item Random forcing is $\sigma$-uf-linked~\cite[Lem.~3.29 \& Lem.~5.5]{mejiamatrix}. 

    \item Let $b$ and $h$ be as in~\autoref{DefEDforcing}. Then $\Eor_b^h$ is $\sigma$-uf-linked~\cite[Lem.~3.8]{BCM}.
\end{enumerate}    
\end{example}

\begin{theorem}[{\cite[Thm.~4.4]{BCM}}]\label{mainpres}
Assume that $\theta\leq\mu$ are uncountable cardinals with $\theta$ regular. Let $\Por_{\gamma,\pi}$ be a ${<}\theta$-uf-extendable matrix iteration as in~\autoref{Defmatuf} such that
\begin{enumerate}[label = \rm (\roman*)]
    \item $\gamma\geq\mu$ and $\pi\geq\mu$,
    \item for each $\alpha<\mu$, $\Delta(\alpha)=\alpha+1$ and $\Qnm_\alpha$ is Cohen forcing, and
    \item $\dot c_\alpha$ is a $\Por_{\alpha+1,\alpha+1}$-name of the Cohen real in $\omega^\omega$ added by $\Qnm_\alpha$.
\end{enumerate}
Then $\Por_{\alpha,\pi}$ forces that $\set{\dot c_\alpha}{\alpha<\mu}$ is $\theta$-$\Dbf$-unbounded, in particular, $\Cbf_{[\mu]^{<\theta}}\leqT \Dbf$.
\end{theorem}
\begin{proof}
    Although the conclusion of~\cite[Thm.~4.4]{BCM} is different, the same proof works.
\end{proof}

% Before proving~\autoref{Main4SN}, we state the following that is very useful to decide the value of $\cof(\SNwf)$ by preserving the values of cardinals of the form $\dfrak_\kappa$ and $\cof\left(([\lambda]^{<\kappa})^w\right)$ from the ground model.

% \begin{lemma}[{\cite[Lem.~6.6]{CM23}}]\label{lem:prcof}
%     Assume that $\kappa$ is an uncountable regular cardinal and that $\Por$ is a $\kappa$-cc poset. If $\lambda\geq\kappa$ is a cardinal with cofinality ${\geq}\kappa$ and $w$ is a set, then $\Por$ preserves the ground-model value of $\cof([\lambda]^{<\kappa})$, $\cof\left(([\lambda]^{<\kappa})^w\right)$ and $\dfrak_\kappa$.
% \end{lemma}

We are finally ready to proceed with the next application. Here,
we denote the relational systems (some introduced in~\autoref{ExmPrs}) $\Rbf_1:=\Lc^*$, $\Rbf_2:=\Cbf_{\Nwf}^{\perp}$, $\Rbf_3:=\Cbf_{\SNwf}^{\perp}$ and $\Rbf_4:=\Dbf$.

\begin{theorem}\label{thm:4SN}
Let $\theta_1\leq\theta_2\leq\theta_3\leq\theta_4\leq\theta_5\leq \theta_6 = \theta_6^{\aleph_0}$ be uncountable regular cardinals and assume that $\cof([\theta_6]^{<\theta_i}) = \theta_6$ for $1 \leq i \leq 4$. %, and let $\nu$ be a cardinal such that $\nu^{\theta_6}=\nu$. 
Then there is a ccc poset $\Por$ forcing:
\begin{enumerate}[label=\rm (\alph*)]
    \item\label{thm:c} $\cfrak = \theta_6$;
    \item\label{thm:one} $\Iwf_f \eqT \Cbf_{[\theta_6]^{<\theta_1}}\leqT\SNwf$ for all $f\in \baireincr$;
    \item \label{thm:two} $\Rbf_i\eqT\Cbf_{[\theta_6]^{<{\theta_i}}}$ for $1\leq i\leq4$ and $\Cbf^\perp_{\Iwf_g}\eqT\Cbf_{[\theta_6]^{<{\theta_3}}}$ for some $g$ in the ground model;
    \item \label{thm:three}  $\theta_5\leqT\Cbf_\Mwf$, $\theta_6\leqT \Cbf_\Mwf$, and $\Ed\leqT \theta_6\times\theta_5$;
    \item \label{thm:four} $\theta_6^{\theta_6}\leqT\SNwf\leqT\Cbf_{[\theta_6]^{<\theta_1}}^{\theta_6}$. % and $\dfrak_{\theta_6}=2^{\theta_6} = \nu$.
\end{enumerate}
In particular, $\Por$ forces:
\begin{align*}
\add(\Nwf) = \add(\Iwf_\id)=\add(\SNwf)=\theta_1\leq\cov(\Nwf)=\theta_2\leq \cov(\SNwf)= \supcov = \theta_3 \leq\\
\leq \bfrak=\theta_4
\leq\non(\Mwf)=\theta_5 
\leq\cov(\Mwf)=\non(\SNwf)=\dfrak=\cfrak=\theta_6 < \dfrak_{\theta_6} \leq \cof(\SNwf).
\end{align*}
In addition, if we assume in the ground model that $\dfrak_{\theta_6} = \cof\left(([\theta_6]^{<\theta_1})^{\theta_6}\right) = \nu$, then the previous ccc poset forces $\dfrak_{\theta_6} = \cof(\SNwf) = \cof\left(([\theta_6]^{<\theta_1})^{\theta_6}\right) = \nu$. (See \autoref{Fig4SN}.)
\end{theorem}

\begin{figure}[ht]
\begin{center}
  \includegraphics[width=\linewidth]{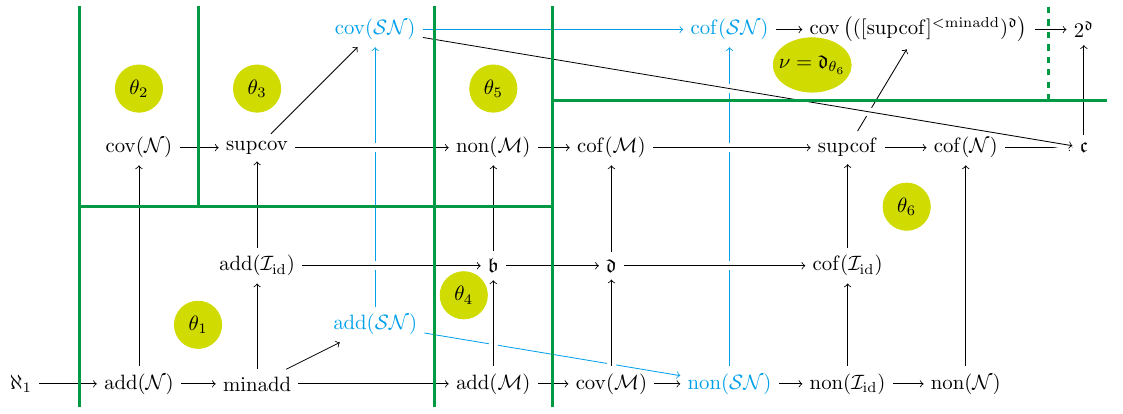}
  \caption{Constellation forced in \autoref{thm:4SN}.}
  \label{Fig4SN}
\end{center}
\end{figure}

\begin{proof}
 %Let $\Omega_0$ and $\Omega_1$ be a partition of $\theta_6$ such that $|\Omega_0|=|\Omega_1|=\theta_6$. 
 For each $\rho<\theta_6\theta_5$ denote $\eta_\rho:=\theta_6\rho$. Fix a bijection $g=(g_0, g_1,g_2):\theta_6\to\{0, 1,2,3\}\times\theta_6\times\theta_6$ and  fix a function $t\colon\theta_6\theta_5\to\theta_6$ such that, for any $\alpha<\theta_6$, the set $\set{\rho<\theta_6\theta_5}{t(\rho)=\alpha}$ is cofinal in $\theta_6\theta_5$. 

%We first force with $\Por_0:=\Fn_{<\theta_6}(\nu,2)$ to obtain, in the generic extension,\[\dfrak_{\theta_6}=\cof\left(([\theta_6]^{<\theta_1})^{\theta_6}\right)=2^{\theta_6}=\nu.\]
%Recall that $\Por_0$ is $\theta_6^+$-cc and ${<}\theta_6$-closed.

% \azul{Afterward, force with $\Cor_{\theta_6}$. Then, in the generic extension, $\cfrak = \theta_6$ and, by~\autoref{lem:prcof}, we get $\dfrak_{\theta_6}=\cof\left(([\theta_6]^{<\theta_1})^{\theta_6}\right)=\nu$.}\smallskip

%Work in $V_{0,0}:=V^{\Por_0}$ and force with 
We construct a ${<}\theta_4$-uf-extendable matrix iteration $\Por$ as follows. 
%
% \azul{
% The ccc poset required is of the form $\Por_0\ast\Por$ where the posets are constructed through the following step:} \smallskip
%
% \noindent\textbf{Step 1:} Note that $\Por_0$ is $\theta_6^+$-cc and ${<}\theta_6$-closed, so it preserves cofinalities, and $\Por_0$ forces  \smallskip

%\noindent\textbf{Step 2:} 
First, construct (by recursion) increasing functions $\rho, \varrho, b \in \omega^\omega$ such that, for all $k<\omega$,
\begin{enumerate}[label=(\roman*)]
    \item $k^{k+1} \leq \rho(k)$,
    \item $\sum_{i<\omega}\frac{\rho(i)^i}{\varrho(i)}<\infty$ and
    \item $k \varrho(k)^{\rho(k)^k} < b(k)$.
\end{enumerate}
By~\autoref{genlink}, we obtain that $\Eor_b=\Eor^1_b$ is $(\rho, \varrho^{\rho^\id})$-linked, and so it is $\aLc^*(\varrho,\rho)$-good.

We now construct the ${<}\theta_4$-uf-extendable matrix iteration $\Por_{\gamma,\,\pi}$ with $\gamma=\theta_6$ and $\pi=\theta_6\theta_6\theta_5$. First set,
\begin{enumerate}[label=\rm (C\arabic*)]
    \item $\Delta(\alpha):=\alpha+1$ and $\Qnm_\alpha$ is Cohen forcing for $\alpha<\theta_6$. 
\end{enumerate}
Define the matrix iteration at each $\xi=\eta_\rho+\varepsilon$ for $0<\rho<\theta_6\theta_5$ and $\varepsilon<\theta_6$ as follows. Denote\footnote{We think of $X_1$ as the set of Borel codes of Borel sets with measure zero.} 
\begin{align*}
\Qor^*_0 & := \Loc, & \Qor^*_1 &:= \Bor, & \Qor^*_2 & := \Eor_b, & \Qor^*_3 &:=\Dor,\\
X_0 & := \omega^\omega, & X_1 & := \Bwf(2^\omega)\cap \Nwf, & X_2 &:= \prod b, & X_3 & := \omega^\omega. 
\end{align*}
For $j<4$, $0<\rho<\theta_6\theta_5$ and $\alpha<\theta_6$, choose 
\begin{enumerate}[label=(E$j$)]
   \item\label{Ej} a collection $\{\Qnm_{j,\alpha,\zeta}^\rho:\zeta<\theta_6\}$ of nice $\Por_{\alpha,\eta_\rho}$-names for posets of the form $(\Qor^*_j)^N$ for some transitive model $N$ of ZFC with $|N|<\theta_{j+1}$
   such that, for any $\Por_{\alpha,\eta_\rho}$-name $\dot F$ of a subset of $X_j$ of size ${<}\theta_{j+1}$, there is some $\zeta<\theta_6$ such that, in $V_{\alpha,\eta_\rho}$, $\Qnm^\rho_{j,\alpha,\zeta} = (\Qor^*_j)^N$ for some $N$ containing $\dot F$ (we explain later why this is possible),
\end{enumerate} 
and set 
\begin{enumerate}[label=\rm (C\arabic*)]
\setcounter{enumi}{1}
    \item $\Delta(\xi):=t(\rho)$ and $\Qnm_{\xi}:=\Eor^{V_{\Delta(\xi),\xi}}$ when $\xi=\eta_\rho$; 
    %when $\xi=\eta_\rho+1+\varepsilon$ for some $\varepsilon\in\Omega_1$

    \item   $\Delta(\xi):=g_1(\varepsilon)$ and $\Qnm_{\xi}:=\Qnm^{\rho}_{g(\varepsilon)}$ when $\xi=\eta_\rho+1+\varepsilon$ for some $\varepsilon< \theta_6$. 
\end{enumerate}

Why is~\ref{Ej} possible? Since $\Por_{\alpha,\eta_\rho}$ has size ${\leq\theta_6}$, $\Por_{\alpha,\eta_\rho}$ forces $\cfrak\leq\theta_6$ and, as $\cof([\theta_6]^{<\theta_{j+1}}) = \theta_6$ is preserved in any ccc forcing extension, in the ground model $V_{0,0}$ we can find a set $\set{\dot F^\rho_{j,\alpha,\zeta}}{\zeta<\theta_6}$ of $\Por_{\alpha,\rho_\eta}$-names of members of $[X_j]^{<\theta_{j+1}}$ which are forced to be cofinal in $[X_j]^{<\theta_{j+1}}$, moreover, it can be found satisfying that, for any $\Por_{\alpha,\eta_\rho}$-name $\dot F$ of a subset of $X_j$ of size ${<}\theta_{j+1}$, there is some $\zeta<\theta_6$ (in the ground model) such that $\Por_{\alpha,\eta_\rho}$ forces $\dot F \subseteq \dot F^\rho_{j,\alpha,\zeta}$. For each $\zeta<\theta_6$ choose some $\Por_{\alpha,\eta_\rho}$-name $\dot N$ of a transitive model of ZFC of size ${<}\theta_{j+1}$ that is forced to contain $\dot F^\rho_{j,\alpha,\zeta}$, and so we let $\Qnm^\rho_{j,\alpha,\zeta}$ be a $\Por_{\alpha,\eta_\rho}$-name for $(\Qor^*_j)^{\dot N}$.
    
Clearly, $\Por:=\Por_{\theta_6,\pi}$ is ccc. We can now show that $\Por$ forces what we want. Note that $\Por$ can be obtained by the FS iteration $\la\Por_{\theta_6,\xi},\Qnm_{\theta_6,\xi}:\xi<\pi\ra$. It should be clear that $\Por$ forces $\cfrak=\theta_6$.%\smallskip

\ref{thm:one}:
To force that $\Cbf_{[\theta_6]^{<\theta_1}}$ is Tukey-below $\Sbf^{f_0}$ and $\Iwf_f$ for all $f$, thanks to \autoref{thm:Paddsn} and \autoref{cor:addsn} it suffices to check that, for each $\xi<\pi$, $\Por_{\gamma, \xi}$ forces that $\Qnm_{\gamma, \xi}$ is $\theta_1$-$\Rbf^{f_0}_\Gwf$-good (for any suitable choice of $\Gwf$). The cases $\xi<\theta_6$ and $\xi=\eta_\rho$ for $\rho>0$ follow by~\autoref{exm:Rfgood}~\ref{exm:Rfgoodiii}; when $\xi=\eta_\rho+1+\varepsilon$ for some $\rho>0$ and $\varepsilon<\theta_6$,  we split into four subcases: the case $g_0(\varepsilon)=0$ is clear by~\autoref{smallgoodness}; when $g_0(\varepsilon)=1$ it follows by~\autoref{exm:Rfgood}~\ref{exm:Rfgoodi}; when $g_0(\varepsilon)=2$, it follows by~\autoref{exm:Rfgood}~\ref{exm:Rfgoodii}; and when $g_0(\varepsilon)=3$, it follows by~\autoref{exm:Rfgood}~\ref{exm:Rfgoodiii}. The Tukey connection $\Iwf_f \leqT \Cbf_{[\theta_6]^{<\theta_1}}$ follows by \autoref{lem:TukeyCtheta}~\ref{Ctheta:a} because $\add(\Iwf_f) \geq \add(\Nwf) \geq \theta_1$ (which is guaranteed in the next item) and $|\R|=\theta_6$.

\ref{thm:two}: 
In a similar way to the previous argument, it can be checked that all iterands are $\theta_1$-$\Rbf_1$-good (see~\autoref{ExmPrs}~\ref{ExmPrsd}, and note that for the case $g_0(\varepsilon)=2$ we use~\autoref{EDsepfam}) so, by~\autoref{Comgood}, $\Por$ forces $\Cbf_{[\theta_6]^{<{\theta_1}}}\leqT\Rbf_1$. 
Now we show that $\Por$ forces $\Rbf_1\leqT\Cbf_{[\theta_6]^{<{\theta_1}}}$. Let $\dot A$ be a $\Por$-name for a subset of $\omega^\omega$ of size ${<}\theta_1$. By employing~\autoref{realint} we can can find $\alpha<\theta_6$ and $\rho<\theta_6\theta_5$ such
that $\dot A$ is $\Por_{\alpha,\eta_\rho}$-name. By~(E0), we can find a $\zeta<\theta_6$ and a $\Por_{\alpha,\eta_\rho}$-name $\dot N$ of a transitive model
of ZFC of size ${<}\theta_1$ such that $\Por_{\alpha,\eta_\rho}$ forces that $\dot N$ contains $\dot A$ as a subset and $\Loc^{\dot N}=\Qnm_{0,\alpha,\zeta}^\rho$, so the
generic slalom added by $\Qnm_{\xi}=\Qnm_{g(\varepsilon)}^\rho$ localizes all the reals in $\dot A$ where $\varepsilon:=g^{-1}(0,\alpha,\zeta)$ and $\xi=\eta_\rho+1+\varepsilon$. 
%It is clear that, in $V_{\alpha+1,\xi+1}$, there is a $\Qnm_{g(\varepsilon)}^\rho$-generic slalom \red{$\dot \varphi\in\Swf(\omega,\id)$ over $V_{\alpha+1,\xi}$}, so this generic slalom localizes all the reals of $A$.
%so $\Por_{\Delta(\xi),\xi+1}$ forces that $\dot \varphi_\rho\in\Swf(\omega,\Hwf)\cap V_{\Delta(\xi),\xi+1}$, the generic slalom over $V_{\Delta(\xi),\xi+1}$ added by $\Qnm_{\Delta(\xi),\xi}$, is localizating over $\dot A\cap\omega^\omega$. 
Then, by applying~\autoref{itsmallsets}, $\Por$ forces that $\Rbf_1\leqT\Cbf_{[\theta_6]^{<{\theta_1}}}$ since $|\omega^\omega|=|\pi|=\theta_6$. %\smallskip

Notice that $\Por$ forces $\Rbf_i\leqT\Cbf_{[\theta_6]^{<\theta_i}}$ for $i\in\{2,3,4\}$: this is basically the same argument as before but for $i=3$ the forcing $\Eor_b$ is used to show that $\aLc(b,1)\leqT \Cbf_{[\theta_6]^{<\theta_3}}$, and the fact that ZFC proves that $\Cbf_{\SNwf}^\perp \leqT \Cbf^\perp_{\Iwf_g} \leqT \aLc(b,1)$ for some $g\in\baireincr$ (see~\cite[Lem.~2.5]{KM21}). 
%The cited references actually indicate that $\Cbf^\perp_{\Iwf_g} \leqT \aLc(b,h)$ for some increasing $g\in \omega^\omega$, 
%So it is forced that $\theta_3\leq \balc_{b,1} \leq \cov(\Iwf_g) \leq \supcov \leq \cov(\SNwf)$.
On the other hand, for $i=2$, since $\Por$ can be obtained by the FS iteration $\la\Por_{\theta_6,\xi},\Qnm_{\theta_6,\xi}:\xi<\pi\ra$ and 
all its iterands are $\theta_2$-$\aLc^*(\varrho,\rho)$-good (see~\autoref{KOpre}),  $\Por$ forces $\Cbf_{[\theta_6]^{<{\theta_2}}}\leqT \aLc^*(\varrho,\rho) \leqT \Rbf_2$ by applying~\autoref{Comgood}; and for $i=3$, since $\Por$ is obtained by the FS iteration of precaliber $\theta_3$ posets, by~\autoref{thm:precaliber} $\Por$ forces $\Cbf_{[\theta_6]^{<{\theta_3}}}\leqT\Rbf_3$.

Note that, by \autoref{exm:ufl}, the matrix iteration is ${<}\theta_4$-uf-extendable. Therefore, by~\autoref{mainpres}, $\Por$ forces $\Cbf_{[\theta_6]^{<{\theta_4}}}\leqT\Rbf_4$.

% \red{Since $\Por$ forces $\theta\leqT\Rbf_4$ for any $\theta$ regular, $\Por$ preserves $\theta\leqT\Rbf_4$ for any regular cardinal $\theta\in[\theta_4,\theta_6]$ because $\Por$ is a $\theta_4$-uf-Knaster. Therefore, we conclude~\ref{thm:two}.} 

%Now we prove that $\theta_1\leq\add(\Nwf)$, $\theta_2\leq\cov(\Nwf)$, $\theta_3\leq\cov(\SNwf)$, and $\theta_4\leq\bfrak$. We just prove that $\Por$ forces $\textrm{MA}_{<\theta_1}$ (which implies that $\add(\Nwf)\geq\theta_1$) because similarly can be obtained the rest. To this end, let $\Snm$ be a $\Por$-name of a ccc poset of size ${<}\theta_1$ and $D$ a family of size ${<}\theta_1$ of dense subsets of $\Snm$. By~\autoref{realint} there is some $\alpha<\theta_6$ and $\rho<\theta_5$ such that both $\Snm$ and $D$ are $\Por_{\alpha,\,\theta_6\rho}$-names. Therefore, there is some  $\zeta<\theta_6$ such that $\Snm=\Qnm_{0,\alpha,\zeta}$, so the generic set added by $\Qnm_{g(\varepsilon)}=\Qnm$ intersects all the dense sets in $D$ where $\varepsilon:=g^{-1}(0,\alpha,\zeta)$ and $\xi=\theta_\rho+2+\varepsilon$. Therefore $\Por$ forces $\add(\Nwf)=\theta_1$. Thus it is forced that $\cov(\Nwf)=\theta_2$, $\cov(\SNwf)=\theta_3$ and $\bfrak=\theta_4$.

\noindent\ref{thm:three}: Since $\cf(\pi)=\theta_5$, by applying~\autoref{lem:strongCohen}, $\Por$ forces that $\theta_5\leqT\Cbf_\Mwf$ and, by~\autoref{matsizebd}, $\Por$ forces $\theta_6\leqT\Cbf_\Mwf$. It remains to prove that $\Por$ forces that $\Ed\leqT \theta_6\times\theta_6\theta_5$ (because $\theta_6\theta_5\eqT \theta_5$). For this purpose, for each $\rho<\theta_6\theta_5$ denote by $\dot e_\rho$ the $\Por_{\Delta(\eta_\rho),\eta_\rho+1}$-name of the eventually different real over $V_{t(\rho),\eta_\rho}$ added by $\Qnm_{t(\rho),\eta_\rho}$. In $V_{\gamma, \pi}$, we are going to define maps $\Psi_-:\omega^\omega\to\theta_6\times\theta_6\theta_5$ and $\Psi_+:\theta_6\times\theta_6\theta_5\to\omega^\omega$ such that, for any $x\in\omega^\omega$ and for any $(\alpha,\rho)\in\theta_6\times\theta_6\theta_5$, if $\Psi_-(x)\leq(\alpha,\rho)$ then $x\neq^\infty\Psi_+(\alpha,\rho)$.

For $x\in V_{\theta_6,\pi}\cap\omega^\omega$, we can find $\alpha_x<\theta_6$ and $\rho_x<\theta_6\theta_5$ such that $x\in V_{\alpha_x,\eta_{\rho_x}}$, so put $\Psi_-(x):=(\alpha_x,\rho_x)$; for $(\alpha, \rho)\in\theta_6\times\theta_6\theta_5$, find some $\rho'<\theta_6\theta_5$ such that $\rho'\geq\rho$ and $t(\rho')=\alpha$, and define $\Psi_+(\alpha,\rho):=\dot e_{\rho'}$. It is clear that $(\Psi_-,\Psi_+)$ is
the required Tukey connection.

%As the FS iteration that determines $\Por$ has cofinality $\theta_5$, $\non(\Mwf)\leq\theta_5$ and $\theta_6\leq\cov(\Mwf)$ is forced by~\autoref{matsizebd} applied to $\Rbf_4$. Hence, $\Por$ forces that $\non(\Mwf)=\theta_5$. On the other hand, since $\Vdash_{\Por}\,\cfrak\leq\theta_6$,  along with $\cov(\Mwf)\leq\cfrak$ it is forced that $\cov(\Mwf)=\cfrak=\theta_6$. 
%Even more, $\Por$ forces that $\supcof=\theta_6$. In addition, $\Por$ forces: 

%\noindent\underline{$\add(\SNwf)=\theta_1:$} This follows because $\theta_1\leq\add(\Nwf)\leq\add(\SNwf)\leq\theta_1$. 

%\noindent\underline{$\non(\SNwf)=\theta_6:$} This follows because $\theta_6=\cov(\Mwf)\leq\non(\SNwf)\leq\non(\Nwf)\leq\theta_6$,   

\noindent\ref{thm:four}: 
Note that $\Por$ forces that $\minadd=\theta_1$  and $\supcof=\theta_6$. %because $\add(\Nwf)\leq\minadd\leq\add(\SNwf)$ and $\non(\SNwf)\leq\supcof\leq\cof(\Nwf)$ (see~\autoref{Cichonwith_SN}).
Since $\Por$ forces $\cov(\Mwf)=\non(\SNwf)=\dfrak=\cfrak=\theta_6$,  $\Por$ forces $\theta_6^{\theta_6}\leqT\SNwf$ by using~\autoref{cor:lowSN}, and $\SNwf\leqT\Cbf_{[\supcof]^{<\minadd}}^\dfrak=\Cbf_{[\theta_6]^{<\theta_1}}^{\theta_6}$ by~\autoref{new_upperb}.

For the final part of the theorem, since $\dfrak_{\theta_6}=\cof\left(([\theta_6]^{<\theta_1})^{\theta_6}\right)=\nu$ in the ground model and $\Por$ is ccc, the same holds in the final extension, so $\cof(\SNwf)=\nu$ by~\ref{thm:four}. %This finishes the proof. 
%
%fix $A_\varnothing^f\in\Iwf_f$ dense $G_\delta$ set and fix $\set{A_\alpha^f}{\alpha<\theta_6}$ cofinal in $\Iwf_f$. Now for $Z\in[\theta_6]^{<\theta_4}$ define \[A_Z^f:=A_\varnothing^f\in\Iwf_f\cup\bigcup_{\alpha\in Z}A_\alpha^f.\] 
%\rojo{Since $\add(\Iwf_f)\geq\theta_4$}, we get that $A_Z^f\in\Iwf_f$. Thus, for any $Z, Z'$ in $[\theta_6]^{<\theta_4}$, if $Z\subseteq Z'$, then $A_Z^f\subseteq A_{Z'}^f$. We claim that $\Seq{ A_Z^f}{Z\in[\theta_6]^{<\theta_4}}$ is a $\Iwf_f$-directed system on $\la[\theta_6]^{<\theta_4},\subseteq\ra$.
%
%By~\autoref{lem:DomSys} there is a $\lambda$-dominating system $\Seq{ A^{f_\gamma}}{\gamma<\lambda}$ on $[\lambda]^{<\kappa}$ because $\cov(\Mwf)=\dfrak=\lambda$. Hence,
\end{proof}

A different forcing construction as in the previous theorem allows to force the same constellation with $\theta_6$ possibly singular and by replacing~\ref{thm:three} by $\Cbf_\Mwf \eqT \Ed \eqT \Cbf_{[\theta_6]^{<\theta_5}}$, while only getting $\cf(\theta_6)^{\cf(\theta_6)} \leqT \SNwf \leqT \Cbf^{\theta_6}_{[\theta_6]^{<\theta_1}}$ in~\ref{thm:four}, but it requires more hypothesis. Concretely, perform a forcing construction as in~\cite[Thm.~1.35]{GKScicmax} under the same assumptions of \autoref{thm:4SN} but without demanding $\theta_6$ regular, and additionally assuming that $\cof([\theta_6]^{<\theta_5}) = \theta_6$ and either {\rm (i)} $\theta_4 = \theta_5$ or {\rm (ii)} $\theta_5$ is $\aleph_1$-inaccessible\footnote{This means that $\mu^{\aleph_0} < \theta_5$ for any cardinal $\mu<\theta_5$.} and there is some cardinal $\theta<\theta_5$ such that $\forall\, \alpha<\theta_4\colon |\alpha| \leq \theta$,\footnote{The latter just means that, if $\theta_4$ is a successor cardinal, then its predecessor is ${\leq}\theta$, or else $\theta_4 \leq \theta$.} and whenever $2^\theta < \theta_6$, $\theta^{<\theta} = \theta$. Although the cited theorem uses GCH (and more restrictions), the presented assumptions are enough. See the proof of \autoref{KST4SN} for further explanation.

The previous result can be modified to force, in addition, that $\cov(\Mwf)<\dfrak$. However, $\non(\SNwf)<\cof(\SNwf)$ is sacrificed unless we include a (quite undesirable) assumption about the cofinality of (the value we want to force to) the continuum.

\begin{theorem}\label{thm4SNtwo}
    Let $\theta_1\leq\theta_2\leq\theta_3\leq\theta_4 \leq\theta_5\leq \theta_6$ be uncountable regular cardinals, and $\theta_7=\theta_7^{\aleph_0}\geq \theta_6$ satisfying $\cof([\theta_7]^{<\theta_i}) = \theta_7$ for $1\leq i\leq 4$. %\footnote{Instead of $\theta_7^{<\theta_4} = \theta_7$, it could be assumed that $\theta_7^{\aleph_0} = \theta_7$ and $\cof([\theta_7]^{<\theta_i}) = \theta_7$ for $1\leq i\leq 4$.}  
     Then there is a ccc poset forcing:
\begin{enumerate}[label=\rm (\alph*)]
    \item\label{thm:a'} $\cfrak = \theta_7$;
    \item\label{thm:one'} $\Iwf_f \eqT\Cbf_{[\theta_7]^{<\theta_1}}\leqT\SNwf$ for all $f\in\baireincr$;
    \item \label{thm:two'} $\Rbf_i\eqT\Cbf_{[\theta_7]^{<{\theta_i}}}$ for $1\leq i\leq4$ and $\Cbf^\perp_{\Iwf_g} \eqT \Cbf_{[\theta_7]^{<{\theta_3}}}$ for some $g$ in the ground model;
    \item \label{thm:three'}  $\theta_5\leqT\Cbf_\Mwf$, $\theta_6\leqT \Cbf_\Mwf$, and $\Ed\leqT \theta_6\times\theta_5$;
    \item \label{thm:four'} $\SNwf\leqT\Cbf_{[\theta_7]^{<\theta_1}}^{\theta_7}$; and
    \item\label{fin:f} if $\cf(\theta_7) = \theta_5$ then $\theta_5^{\theta_5} \leqT \SNwf$.
\end{enumerate}
In particular,
\begin{align*}
\add(\Nwf) = \add(\Iwf_\id)=\add(\SNwf)=\theta_1\leq\cov(\Nwf)=\theta_2\leq \cov(\SNwf)= \supcov = \theta_3 \leq\\
\leq \bfrak=\theta_4
\leq\non(\Mwf)=\theta_5 
\leq\cov(\Mwf)=\theta_6 \leq \non(\SNwf)=\dfrak=\cfrak=\theta_7.%\text{ and } \dfrak_{\theta_6}=\cof(\SNwf)=\theta_7.
\end{align*}
In addition, if we assume in the ground model that $\cf(\theta_7) = \theta_5$ and $\dfrak_{\theta_5} = \cof\left(([\theta_7]^{<\theta_1})^{\theta_7}\right) = \nu$, then the previous ccc poset forces $\dfrak_{\theta_5}=\dfrak_{\theta_7} = \cof(\SNwf) = \cof\left(([\theta_7]^{<\theta_1})^{\theta_7}\right) = \nu$. (See~\autoref{Fig4SNtwo}.)
%If, in addition, $\theta_5^{<\theta_5}=\theta_5$, $\cf(\theta_7) = \theta_5$ and $\nu$ is a cardinal such that $\nu^{\theta_7}=\nu$, then the two-step iteration of $\Fn_{<\theta_5}(\nu,2)$ with the ccc poset for the above forces the same and $\dfrak_{\theta_7}=\cof(\SNwf)=\nu$.
\end{theorem}

\begin{figure}[ht]
\begin{center}
  \includegraphics[width=\linewidth]{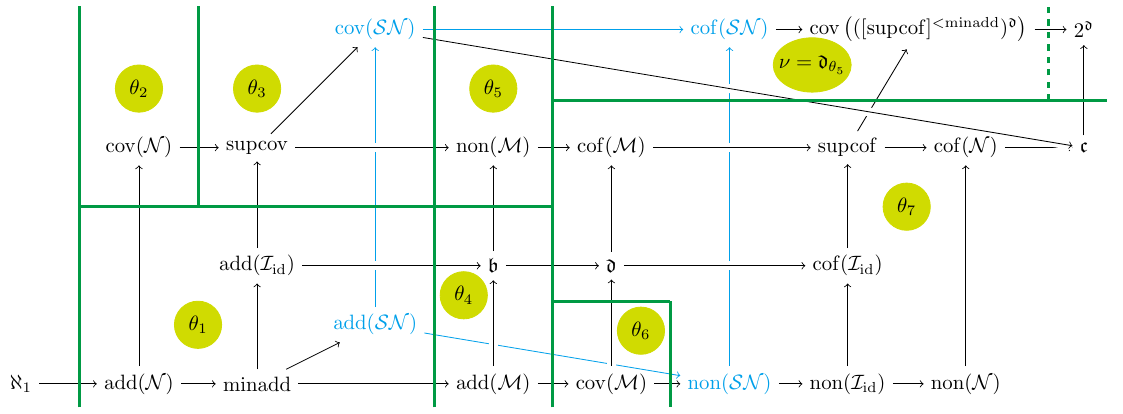}
  \caption{Constellation forced in~\autoref{thm4SNtwo}.}
  \label{Fig4SNtwo}
\end{center}
\end{figure}

\begin{proof}
    Construct a ${<}\theta_4$-uf-extendable matrix iteration $\Por$ as in \autoref{thm:4SN}, but use $\eta_\rho:=\theta_7 \rho$ for $\rho<\theta_6\theta_5$, a bijection $g\colon \theta_7\to \{0,1,2,3\}\times \theta_6\times \theta_7$, and a function $t\colon \theta_6\theta_5\to \theta_6$ such that $\{\rho<\theta_6\theta_5\colon t(\rho)=\alpha\}$ is cofinal in $\theta_6\theta_5$ for all $\alpha<\theta_6$. Then, $\Por$ forces~\ref{thm:a'}--\ref{thm:four'}.

    To show~\ref{fin:f}, we use notions and results from~\cite{CM23} that we do not fully review. In this reference, we define a principle we denote by $\DS(\delta)$, with the parameter an ordinal $\delta$, which has a profound effect on $\cof(\SNwf)$. According to~\cite[Cor.~6.2]{CM23}, $\Por$ forces $\DS(\theta_7\theta_5)$. On the other hand,~\cite[Thm.~4.24]{CM23} states that, if $\DS(\delta)$ holds and $\non(\SNwf)=\supcof$ has the same cofinality as $\delta$, say $\theta$, then $\la\theta^\theta,\leq\ra\leqT \SNwf$. Since $\Por$ forces $\non(\SNwf)=\supcof=\theta_7$ and $\cf(\theta_7) = \cf(\theta_7 \theta_5) = \theta_5$, we conclude~\ref{fin:f}.

    For the final part of the theorem, we proceed as in the previous proofs. %noticing that, in the ground model, 
    %$\dfrak_{\theta_5} = \dfrak_{\theta_7} = \cof\left(([\theta_7]^{<\theta_1})^{\theta_7}\right) = \nu$.
\end{proof}

In the previous two theorems, we have obtained constellations that can include $\cov(\Nwf)<\cov(\SNwf)<\bfrak$. But, how about constellations including $\bfrak<\cov(\Nwf)<\cov(\SNwf)$ and $\cov(\Nwf)< \bfrak <\cov(\SNwf)$? The model constructed in~\cite[Thm.~2.43]{KST} can be used for the first one, but there $\cov(\SNwf)=\non(\Mwf)$.

\begin{theorem}\label{KST4SN}
Let $\theta_1\leq\theta_2\leq\theta_3\leq\theta_5$  be uncountable regular cardinals, and let $\theta_6 = \theta_6^{\aleph_0} \geq \theta_5$ be a cardinal such that $\cof([\theta_6]^{<\theta_i}) = \theta_6$ for $1 \leq i \leq 5$. 
Assume that either {\rm (i)} $\theta_2=\theta_3$, or {\rm (ii)} both $\theta_3$ and $\theta_5$ are $\aleph_1$-inaccessible, 
and there is some cardinal $\theta<\theta_3$ such that  $\forall\, \alpha<\theta_2\colon |\alpha|\leq\theta$, and whenever $2^\theta<\theta_6$, $\theta^{<\theta} = \theta$.
Then, for some suitable parameters $b,h\in\omega^\omega$, there is a ccc poset forcing:
\begin{multicols}{2}
\begin{enumerate}[label=\rm (\alph*)]
    \item $\cfrak = \theta_6 < \cof(\SNwf)$;
    \item $\Cbf_{[\theta_6]^{<\theta_1}}\leqT\SNwf$;
    \item $\Nwf \eqT \Iwf_f \eqT \Cbf_{[\theta_6]^{<\theta_1}}$ for all $f$;
    \item $\Dbf \eqT \Cbf_{[\theta_6]^{<\theta_2}}$;
    \item $\Cbf^\perp_\Nwf\eqT \Cbf_{[\theta_6]^{<\theta_3}}$;
    \item $\Cbf^\perp_\SNwf \eqT \aLc(b,h) \eqT \Cbf_{[\theta_6]^{<\theta_5}}$; 
    \item $\Cbf_\Mwf \eqT \Cbf_{[\theta_6]^{<\theta_5}}$;
    \item $\cf(\theta_6)^{\cf(\theta_6)}\leqT\SNwf\leqT\Cbf_{[\theta_6]^{<\theta_1}}^{\theta_6}$. % and $\dfrak_{\theta_6}=2^{\theta_6} = \nu$.
\end{enumerate}
\end{multicols}
In particular,
\begin{align*}
\add(\Nwf) & = \add(\Iwf_\id) =\add(\SNwf)=\theta_1\leq \bfrak =\theta_2 \leq \cov(\Nwf)=\theta_3\leq\\ 
 & \leq \cov(\SNwf)= \supcov = \non(\Mwf) = \theta_5
\leq\cov(\Mwf)=\cfrak=\theta_6 <  \cof(\SNwf).
\end{align*}
In addition, if we assume in the ground model that $\dfrak_{\cf(\theta_6)} = \cof\left(([\theta_6]^{<\theta_1})^{\theta_6}\right) = \nu$, then the previous ccc poset forces $\dfrak_{\theta_6} = \cof(\SNwf) = \cof\left(([\theta_6]^{<\theta_1})^{\theta_6}\right) = \nu$. (See~\autoref{FigKST4SN}.)
\end{theorem}

\begin{figure}[ht]
\begin{center}
  \includegraphics[width=\linewidth]{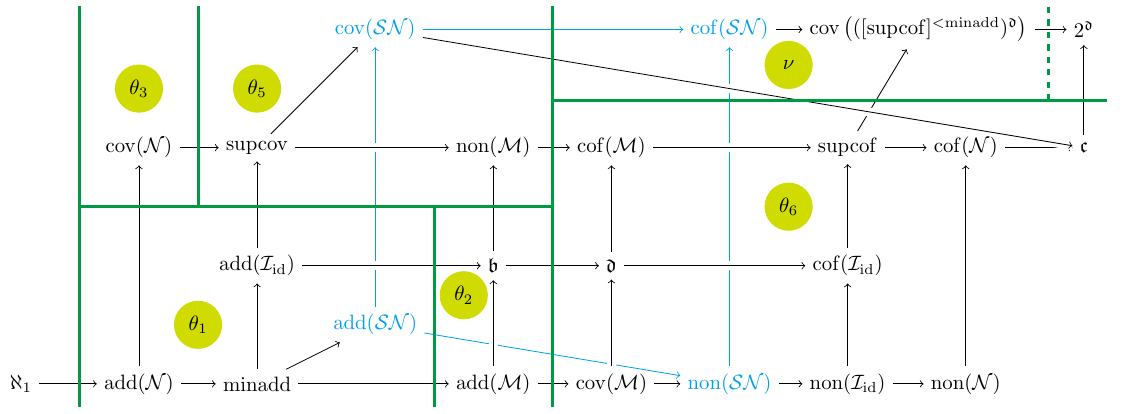}
  \caption{Constellation forced in~\autoref{KST4SN}.}
  \label{FigKST4SN}
\end{center}
\end{figure}

\begin{proof}
  The forcing construction for this theorem uses a FS iteration of length $\theta_5+\theta_5$ with fam-limits and a ccc poset they denote $\Tilde{\Eor}$ to increase $\non(\Mwf)$ (see also~\cite{M24Anatomy}). This poset, like $\Eor^h_b$, increases $\balc_{b,h}$ for some (fast increasing) parameters $b,h\in\omega^\omega$ (also recall that $\Cbf^\perp_{\Iwf_g} \leqT \aLc(b,h)$ for some parameter $g\in\baireincr$), and it is forcing equivalent to a subalgebra of random forcing. Therefore, the same arguments as the previous theorems can be used to understand its effect on $\SNwf$.

  We do not develop the details of this construction, but just explain the reason we can weaken GCH (originally used in~\cite{KST}) to our hypothesis, as similarly done in~\cite{modKST,CMU}. To make sense of them, we suggest to use small transitive models as in \autoref{thm4SN} for the cardinal characteristics corresponding to $\theta_1$ and $\theta_2$, instead of restricting to $V^{\Por'_{\alpha}}$ for some $\Por'_\alpha \subsetdot \Por_\alpha$.\footnote{As done in~\cite{modKST}.} 
  In the case $\theta_2 = \theta_3$ no fam-limits are needed and we can iterate with $\tilde\Eor$ (and even with $\Eor^h_b$) without restricting it to any small model $N$, which is fine because it is $\sigma$-uf-linked (see~\cite[Cor.~4.3.11 \& Thm.~4.2.19]{uribethesis}) and hence $\Dor$-good (see \autoref{ExmPrs}~\ref{ExmPrsb}). For the second case, when $\theta_5 < 2^\theta$ and $\theta^{<\theta} = \theta$, proceed as in the proof of~\cite[Thm.~2.43]{KST} and \cite[Thm.~11.4]{CMU} by first constructing the ccc poset in the $\Fn_{<\theta}(\theta_5,2)$-extension (to use a required assumption $\theta_5\leq 2^\theta$), and then pulling it down to the ground model; when $\theta_5\leq 2^\theta$, there is no need to step in the $\Fn_{<\theta}(\theta_5,2)$-extension.
\end{proof}

The case $\cov(\Nwf)<\bfrak<\cov(\SNwf) = \non(\Mwf)$ can be obtained by a modification of the poset for \autoref{thm:4SN}.

\begin{theorem}\label{thm:4SN3}
Let $\theta_1\leq\theta_2\leq\theta_3\leq \theta_5\leq \theta_6 = \theta_6^{\aleph_0}$ be uncountable regular cardinals %\footnote{Note that there is no $\theta_4$.} 
and assume that $\cof([\theta_6]^{<\theta_i}) = \theta_6$ for $1 \leq i \leq 3$. %, and let $\nu$ be a cardinal such that $\nu^{\theta_6}=\nu$. 
Then, for some suitable parameters $b,h\in\omega^\omega$, there is a ccc poset forcing:
\begin{multicols}{2}
\begin{enumerate}[label=\rm (\alph*)]
    \item $\cfrak = \theta_6$;
    \item $\Cbf_{[\theta_6]^{<\theta_1}}\leqT\SNwf$;
    \item $\Nwf \eqT \Iwf_f\eqT \Cbf_{[\theta_5]^{<\theta_1}}$ for all $f$;
    \item $\Cbf^\perp_\Nwf \eqT \Cbf_{[\theta_5]^{<\theta_2}}$;
    \item $\Dbf\eqT \Cbf_{[\theta_5]^{<\theta_3}}$;
    \item\label{f:Snp} $\theta_5\leqT\Cbf^\perp_\SNwf$ and $\theta_6\leqT\Cbf^\perp_\SNwf$;
    \item  $\theta_5\leqT\Cbf_\Mwf$ and $\theta_6\leqT \Cbf_\Mwf$;
    \item\label{h:Snp} $\Cbf^\perp_{\Iwf_g}\leqT \theta_6\times\theta_5$ for some $g$ in the ground model;
    \item $\theta_6^{\theta_6}\leqT\SNwf\leqT\Cbf_{[\theta_6]^{<\theta_1}}^{\theta_6}$. % and $\dfrak_{\theta_6}=2^{\theta_6} = \nu$.
\end{enumerate}
\end{multicols}
In particular, 
\begin{align*}
\add(\Nwf)= \add(\Iwf_\id) =\add(\SNwf)=\theta_1\leq\cov(\Nwf)=\theta_2 \leq\bfrak=\theta_3\leq \cov(\SNwf)= \supcov =\\
=\non(\Mwf)=\theta_5 
\leq \cov(\Mwf)=\non(\SNwf)=\dfrak=\cfrak=\theta_6 < \dfrak_{\theta_6} \leq \cof(\SNwf).
\end{align*}
In addition, if we assume in the ground model that $\dfrak_{\theta_6} = \cof\left(([\theta_6]^{<\theta_1})^{\theta_6}\right) = \nu$, then the previous ccc poset forces $\dfrak_{\theta_6} = \cof(\SNwf) = \cof\left(([\theta_6]^{<\theta_1})^{\theta_6}\right) = \nu$. (See~\autoref{Figthm4SN3}.)
\end{theorem}
\begin{figure}[ht]
\begin{center}
  \includegraphics[width=\linewidth]{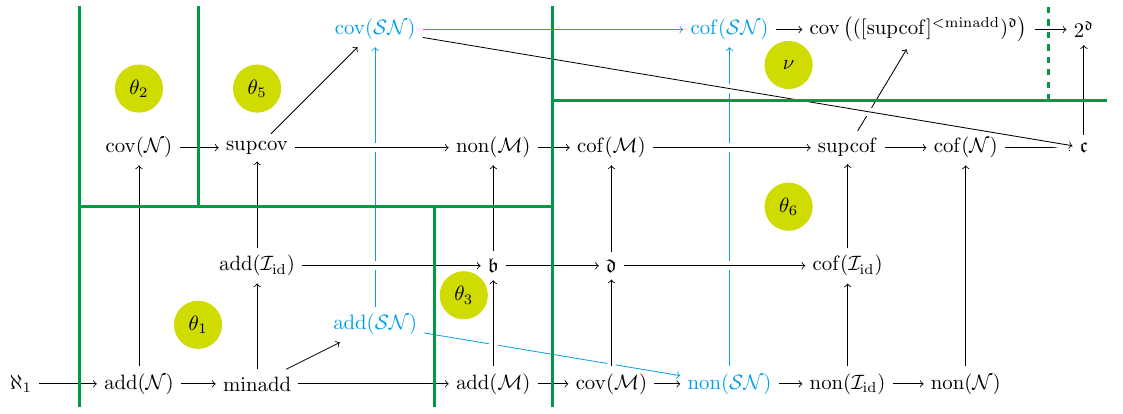}
  \caption{Constellation forced in~\autoref{thm:4SN3}.}
  \label{Figthm4SN3}
\end{center}
\end{figure}
\begin{proof}
    Proceed as in the proof of \autoref{thm:4SN}, but neglect $\Eor^N_b$ and replace $\Eor$ by $\Eor_b$ (which is $\sigma$-$\Fr$-linked, see~\autoref{exm:ufl}), which is the responsible for forcing~\ref{h:Snp}. Item~\ref{f:Snp} follows by \autoref{thm:cfcovSN} applied to $\la \Por_{\theta_6,\xi}:\, \xi\leq \pi\ra$ and
    $\la \Por_{\alpha,\pi}:\, \alpha\leq \theta_6\ra$, respectively.
\end{proof}
    
%Although it is unclear how to modify \autoref{KST4SN} to additionally force $\cof(\SNwf)<\non(\Mwf)$, this is possible for \autoref{thm:4SN3} by an unpublished iteration construction with ultrafilter limits for non-${<}\theta_3$-centered posets (see~\cite{GKMSeva}).

It is unclear how to modify \autoref{KST4SN} and~\ref{thm:4SN3} to additionally force $\cof(\SNwf)<\non(\Mwf)$.

\begin{remark}\label{rem:covIf}
    In all the constellations proved consistent in this section, we have that the additivities of the Yorioka ideals are the same, likewise for the uniformities and cofinalities. However, we cannot ensure the same about the coverings: while some of them are equal to $\supcov=\cof(\SNwf)$, some others may be equal to $\cov(\Nwf)$. For instance, in \autoref{thm:4SN} the iteration is $\theta_2$-$\aLc^*(\varrho,\rho)$-good, so it forces $\Cbf_{[\theta_6]^{<\theta_2}}\leqT \aLc^*(\varrho,\rho)$. On the other hand, by \autoref{KOpre}, $\aLc^*(\varrho,\rho) \leqT \aLc(\varrho,\rho^\id) \leqT \Cbf^\perp_\Nwf$, but if $\varrho$ is very fast increasing, ZFC proves that there is some $g'\in\baireincr$ such that $\aLc(\varrho,\rho^\id) \leqT \Cbf^\perp_{\Iwf_{g'}}$ (see~\cite[Lem.~2.4]{KM21}), so it is forced that $\cov(\Iwf_{g'}) = \theta_2$. It is unclear to us whether all coverings are forced to be the same when $\varrho$ is not that ``fast-increasing".
\end{remark}

\section{Applications~II: Cicho\'n's maximum}\label{sec:subm}

In the following theorem, we force Cicho\'n's maximum simultaneously with the four cardinal characteristics associated with $\SNwf$ pairwise different. Note that, in the same generic extension, the four cardinal characteristics associated with  $\Iwf_f$ are pairwise different for any $f\in\baireincr$. Most of this section is dedicated to proving this theorem. Note that we do not force a value of $\cof(\SNwf)$, but
only force that $\cof(\Nwf) \leq \cof(\SNwf)$. However, there are some cases where we can force a value of $\cof(\SNwf)$, which we present at the end of this section.

\begin{theorem}\label{thm:4SNmax}
    Let $\lambda = \lambda^{\aleph_0}$ be a cardinal, and for $1\leq i\leq 5$ be $\lambda^\bfrak_i$ and $\lambda^\dfrak_i$ be regular uncountable cardinals such that $\lambda^\bfrak_i\leq \lambda^\bfrak_j \leq \lambda^\dfrak_j \leq \lambda^\dfrak_i\leq\lambda$ for any $i\leq j$. Then there is a ccc poset forcing the constellation in~\autoref{Figthm4SNmax}:
    \begin{align*}
    \aleph_1 \leq & \add(\Nwf)=\add(\SNwf) = \add(\Iwf_\id) = \lambda^\bfrak_1 \leq \cov(\Nwf) = \lambda^\bfrak_2 \leq\\
    \leq & \cov(\SNwf) = \supcov =\lambda^\bfrak_3 \leq \bfrak = \lambda^\bfrak_4 \leq \non(\Mwf) = \lambda^\bfrak_5 \leq\\
    \leq & \cov(\Mwf) = \lambda^\dfrak_5 \leq \dfrak = \lambda^\dfrak_4 \leq \non(\SNwf) = \lambda^\dfrak_3 \leq\\
    \leq & \non(\Nwf) = \lambda^\dfrak_2 \leq \cof(\Nwf) = \cof(\Iwf_\id) = \lambda^\dfrak_1 \leq \cfrak =\lambda \text{ and } \lambda^\dfrak_1 \leq \cof(\SNwf).
    \end{align*}
\end{theorem}

\begin{figure}[ht!]
\begin{center}
  \includegraphics[width=\linewidth]{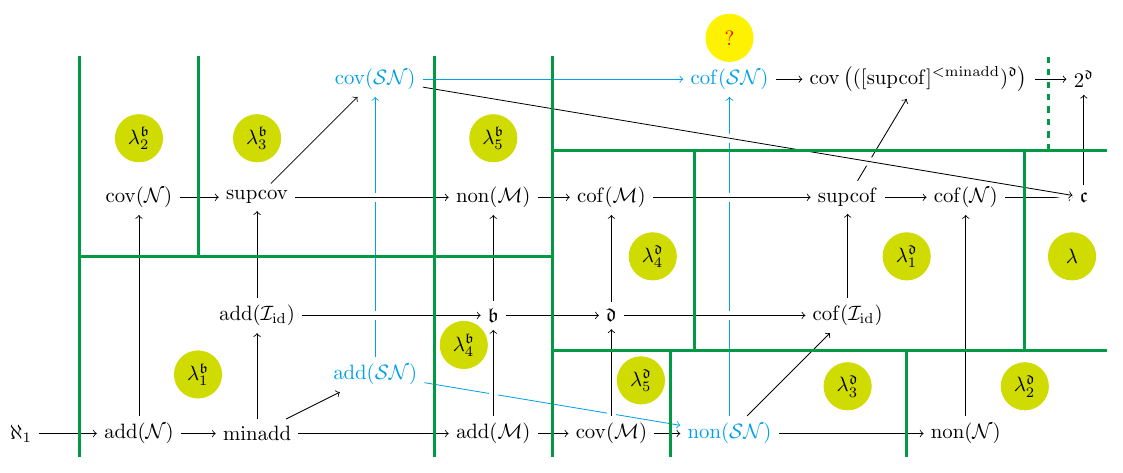}
 \caption{The constellation forced in \autoref{thm:4SNmax}. It is unclear what values are forced to $\cof(\SNwf)$ and $\cov\left(([\supcof]^{<\minadd})^\dfrak\right)$, but $2^\dfrak$ is forced to be the ground-model value of $\lambda^{\lambda^\dfrak_4}$.}
  \label{Figthm4SNmax}
\end{center}
\end{figure}

To prove the theorem, we apply the method of intersections with elementary submodels from~\cite{GKMS}, but we follow the presentation from~\cite[Sec.~4 \&~5]{CM22}. Fix a large enough regular cardinal $\chi$. The goal is to find some $N\preceq H_\chi$, closed under countable sequences, such that $\Por\cap N$ is the desired ccc poset, where $\Por$ is the ccc poset constructed in \autoref{thm:4SN} for large enough $\theta_i$.

However, we do not explicitly construct this model $N$ since it follows the same dynamics as in~\cite[Sec.~3]{GKMS} and~\cite[Sec.~5]{CM22}. More important than the construction of the model $N$ is the effect of $\Por\cap N$ on the Tukey connections forced by $\Por$. Concretely, when $\Rbf$ is a definable relational system of the reals  (see \autoref{def:defrel}) or just sufficiently absolute, with parameters in $N$, and $K = \la A,B,\leqtr\ra\in N$ is a relational system, then $\Vdash_\Por \Rbf\leqT K$ implies $\Vdash_{\Por\cap N} \Rbf\leqT K\cap N$ (see \autoref{KcapN}). Here, $K\cap N:=\la A\cap N, B\cap N,\leqtr\ra$, and both $K$ and $K\cap N$ are \emph{fixed} in the ground model, i.e.\ they are not re-interpreted in generic extensions. A similar result holds for $K\leqT \Rbf$ as well. Therefore, the values that $\Por\cap N$ force to $\bfrak(\Rbf)$ and $\dfrak(\Rbf)$ are determined by the values of $\bfrak(K\cap N)$ and $\dfrak(K\cap N)$, which do not change in ccc forcing extensions for the $K$ of our interest. In~\cite[Sec.~4 \&~5]{CM22}, it is explained in detail how such $N$ should be constructed to get the desired values of $\bfrak(K\cap N)$ and $\dfrak(K\cap N)$, so we consider that there is no need to repeat the whole process to prove our theorem.

To make sense of the above, we must explain the effect of $\Por\cap N$ on the Tukey connections that $\Por$ forces for $\SNwf$ and $\Cbf^\perp_\SNwf$ because these relational systems are \underline{not} ``sufficiently absolute", so we do not necessarily have the same effect explained previously. We prove some partial results for $\SNwf$ in this sense, which happens to be enough for applying the same construction of $N$ as presented in~\cite[Sec.~5]{CM22}.

Generally speaking, let $\kappa$ be an uncountable regular cardinal, $\chi$ a large enough regular cardinal, $N\preceq H_\chi$ a \emph{${<}\kappa$ closed model}, i.e.\ $N^{<\kappa}\subseteq N$, and let $\Por\in N$ be a $\kappa$-cc poset. Then $\Por\cap N\subsetdot \Por$, so $\Por\cap N$ is $\kappa$-cc, too. Even more, there is a one-to-one correspondence between the (nice) $\Por\cap N$-names of members of $H_\kappa$, and the $\Por$-names $\tau\in N$ of members of $H_\kappa$, in particular, we have the same correspondence between the nice names of reals. Moreover, if $G$ is $\Por$-generic over $V$, then $H_\kappa^{N[G]} = H_\kappa^{V[G]}\cap N[G]$. We also have absoluteness results: if $p\in\Por\cap N$, $\varphi(\bar x)$ is a sufficiently absolute formula and $\bar \tau\in N$ is a finite sequence of $\Por$-names of members of $H_\kappa$, then
\[p\Vdash_\Por\varphi(\bar\tau) \text{ iff }p\Vdash_{\Por\cap N}\varphi(\bar\tau).\]

Let $\Rbf=\la X,Y,\sqsubset\ra$ and $K=\la A,B,\leqtr\ra$ be relational systems. Note that
\begin{align*}
    K\leqT \Rbf \text{ iff } & \text{there is a sequence $\la x_a:\, a\in A\ra$ in $X$ such that}\\
     & \forall\, y\in Y\ \exists\, b_y\in B\ \forall\, a\in A\colon a \nleqtr b_y \imp x_a \nsqsubset y;\\[1ex]
    \Rbf \leqT K \text{ iff } & \text{there is a sequence $\la y_b:\, b\in B\ra$ in $Y$ such that}\\
    & \forall\, x\in X\ \exists\, a_x\in A\ \forall\, b\in B\colon a_x \leqtr b \imp x\sqsubset y_b.
\end{align*}

In the first case, we say that the sequence $\la x_a:\, a\in A\ra$ \emph{witnesses $K\leqT\Rbf$}, while in the second case we say that the sequence $\la y_b:\, b\in B\ra$ \emph{witnesses $\Rbf\leqT K$}.

We are interested in the case when $K$ is a fixed relational system in the ground model (i.e.\ it is interpreted in any model of ZFC as the same relational system), while $\Rbf$ may change depending on the model it is interpreted, e.g.\ when $\Rbf$ is a gPrs. If $\Por$ is a poset, then it is clear that $\Vdash_\Por K\leqT \Rbf$ iff there is a sequence $\la \dot x_a :\, a\in A\ra$ of members of $X$ such that $\Por$ forces that $\la \dot x_a :\, a\in A\ra$ witnesses $K\leqT\Rbf$. Here, we say that $\la \dot x_a :\, a\in A\ra$ \emph{witnesses $\Vdash_\Por K\leqT \Rbf$}. In a similar way, we can define ``$\la \dot y_b:\, b\in B\ra$ \emph{witnesses $\Vdash_\Por\Rbf\leqT K$}".

In this context, we know the following result.

\begin{lemma}[{\cite[Lem.~4.5]{CM22}}]\label{KcapN}
    Let $\kappa$ be an uncountable regular cardinal, $\chi$ a large enough regular cardinal, and let $N\preceq H_\chi$ be a ${<}\kappa$-closed model. Assume that $\Por\in N$ is a $\kappa$-cc poset, $K=\la A,B,\leqtr\ra\in N$ is a relational system, and $\Rbf$ is a definable relational system of the reals with parameters in $N$. %Denote $K\cap N:=\la A\cap N, B\cap N, \leqtr\ra$. 
    Then:
    \begin{enumerate}[label = \rm (\alph*)]
        \item If $\la \dot y_b:\, b\in B\ra \in N$ witnesses $\Vdash_{\Por} \Rbf\leqT K$ then $\la \dot y_b:\, b\in B\cap N\ra$ witnesses $\Vdash_{\Por\cap N} \Rbf \leqT K\cap N$.
        \item If $\la \dot x_a:\, a\in A\ra \in N$ witnesses $\Vdash_{\Por} K\leqT \Rbf$ then $\la \dot x_a:\, a\in A\cap N\ra$ witnesses $\Vdash_{\Por\cap N} K\cap N\leqT \Rbf$.
    \end{enumerate}
\end{lemma}

However, the previous lemma cannot be applied directly to $\SNwf$ and $\Cbf_{\SNwf}$ because the ideal $\SNwf$ is not absolute, in general. So we need to look deeper at $\SNwf$ if we want to prove similar results.

We can code the members (of a cofinal subset) of $\SNwf$ by sequences $\la \sigma^f:\, f\in\baireincr\ra$ such that $\sigma^f\in 2^f$, noting that $\bigcap_{f\in\baireincr}[\sigma^f]_\infty\in\SNwf$.\footnote{The domain of such a sequence can be a dominating family $D\subseteq\baireincr$ instead.} In the context of forcing, we can define nice names of members of $\SNwf$ in the following way. Given a poset $\Por$, we say that a collection of $\Por$-names $\baireincr_\Por$ is a \emph{$\Por$-core of $\baireincr$} if it satisfies:
\begin{enumerate}[label = \rm (\arabic*)]
    \item $\baireincr_\Por$ is a set of $\Por$-nice names of members of $\baireincr$, and
    \item If $\dot x$ is a $\Por$-name and $\Vdash_\Por \dot x\in\baireincr$, then there is a \underline{unique} $\dot f\in \baireincr_\Por$ such that $\Vdash_\Por \dot x = \dot f$.
\end{enumerate}
A $\Por$-core of $\baireincr$ always exists, and we use the notation $\baireincr_\Por$ to refer to one of them. We then say that a sequence $\bar\sigma = \la\sigma^{\dot f} :\, \dot f\in \baireincr_\Por\ra$ is a \emph{nice $\Por$-name of a member of $\SNwf$} if each $\sigma^{\dot f}$ is a nice $\Por$-name of a member of $2^{\dot f}$ (i.e.\ a nice name of a member of $(2^{<\omega})^\omega$ which is forced to be in $2^{\dot f}$).

Within the assumptions of \autoref{KcapN}, when $\bar\sigma = \la\sigma^{\dot f} :\, \dot f\in \baireincr_\Por\ra$ is a nice $\Por$-name of a member of $\SNwf$ and $\bar\sigma\in N$, we can define $\bar\sigma|_N:= \la\sigma^{\dot f} :\, \dot f\in \baireincr_\Por\cap N\ra$ where $\sigma^{\dot f}$ refers to the equivalent $\Por\cap N$-nice name. Since $\baireincr_{\Por\cap N}:=\baireincr_\Por\cap N$ is a $\Por\cap N$-core of $\baireincr$, we have that $\bar\sigma|_N$ is a nice $\Por\cap N$-name of a member of $\SNwf$.

For $\SNwf$, we have the following version of \autoref{KcapN}.

\begin{lemma}\label{SN_N}
    Under the assumptions of \autoref{KcapN}, let $\la\bar\sigma_a:\, a\in A\ra\in N$ be a sequence of nice $\Por$-names of members of $\SNwf$ and $f_0\in\baireincr\cap N$. Then:
    \begin{enumerate}[label = \rm (\alph*)]
        \item\label{SN_Na} If $\la\bar\sigma_a:\, a\in A\ra$ witnesses $\Vdash_\Por K\leqT \Sbf^{f_0}$, then $\la\bar\sigma_a|_N:\, a\in A\cap N\ra$ witnesses $\Vdash_{\Por\cap N} K\cap N\leqT \Sbf^{f_0}$.

        \item\label{SN_Nb} If $\la\bar\sigma_a:\, a\in A\ra$ witnesses $\Vdash_\Por K\leqT \Cbf^\perp_\SNwf$, then $\la\bar\sigma_a|_N:\, a\in A\cap N\ra$ witnesses $\Vdash_{\Por\cap N} K\cap N\leqT \Cbf^\perp_\SNwf$.
    \end{enumerate}
\end{lemma}
\begin{proof}
    We prove~\ref{SN_Na} (\ref{SN_Nb} is similar). Let $p\in\Por\cap N$ and let $\dot \tau$ be a nice $\Por\cap N$-name of a member of $2^{f_0}$, which can be seen as a $\Por$-name that belongs to $N$. Since $\Por$ forces that $\la\bar\sigma_a:\, a\in A\ra$ witnesses $K\leqT \Sbf^{f_0}$,
    \[N\models \exists\, q\leq p\ \exists\, b\in B\ \forall\, a\in A\colon a \nleqtr b \imp q\Vdash_\Por \bigcap_{\dot f\in \baireincr_\Por}[\sigma^{\dot f}_a]_\infty \nsqsubseteq \dot \tau,\]
    so pick $q\in\Por\cap N$ stronger than $p$ and $b\in N$ such that
    \[N\models \forall\, a\in A\colon a \nleqtr b \imp q\Vdash_\Por \bigcap_{\dot f\in \baireincr_\Por}[\sigma^{\dot f}_a]_\infty \nsqsubseteq \dot \tau.\]
    It remains to show that $q\Vdash_{\Por\cap N}``\bigcap_{\dot f\in \baireincr_\Por\cap N}[\sigma^{\dot f}_a]_\infty \nsqsubseteq \dot \tau"$ for any $a\in A\cap N$ with $a\nleqtr b$. For such an $a$, we can find a nice $\Por$-name $\dot x\in N$ of a real in $2^\omega$ such that
    \[N\models \text{``} q \Vdash_\Por \dot x\in \bigcap_{\dot f\in \baireincr_\Por}[\sigma^{\dot f}_a]_\infty \text{ and }q \Vdash_\Por \dot x\notin \bigcup_{i<\omega}[\dot \tau(i)]\text{''},\]
    so, for any $\dot f\in \baireincr_\Por \cap N$, $N\models q\Vdash_\Por \dot x\in [\sigma^{\dot f}_a]_\infty$. Therefore, $q\Vdash_{\Por\cap N}\dot x\in [\sigma^{\dot f}_a]_\infty$ and $q \Vdash_{\Por\cap N} \dot x\notin \bigcup_{i<\omega}[\dot \tau(i)]$, which allows us to conclude that $q\Vdash_{\Por\cap N} \bigcap_{\dot f\in \baireincr_\Por\cap N}[\sigma^{\dot f}_a]_\infty \nsqsubseteq \dot \tau$
\end{proof}

The difficulty for dealing with $K\leqT \SNwf$, $\SNwf\leqT K$ and $\Cbf^\perp_\SNwf\leqT K$ lies in the fact that we cannot guarantee that any nice $\Por\cap N$-name of a member of $\SNwf$ has the form $\bar\sigma|_N$ for some nice $\Por$-name $\bar\sigma\in N$ of a member of $\SNwf$.

\begin{proof}[Proof of \autoref{thm:4SNmax}]
    We start with the following observation from Elliot Glazer (private communication): Although Cicho\'n's maximum is originally forced in~\cite{GKMS} by assuming \emph{eventual GCH} (i.e.\ that GCH holds above some cardinal) in the ground model, the latter assumption is inessential for the main result. The same is true for \autoref{thm:4SNmax}, that is, this theorem is proved assuming eventual GCH, but this assumption can be removed as follows:\footnote{Although most set theorists are happy to force statements by just assuming GCH in the ground model, we are more pedantic and minimize, successfully, the strength of the assumptions in the ground model.} \emph{Without assuming eventual GCH}, let $\kappa^*$ be a large enough regular cardinal, and let $W$ be a set of ordinals coding $R_{\kappa^*}$, the $\kappa^*$-th level of the Von Neumann hierarchy of the universe of sets. Since $L[W]$ models ZFC with eventual GCH, \emph{working inside $L[W]$} we can find a ccc poset $\Qor$ as in \autoref{thm:4SNmax} (using the proof under eventual GCH sketched from the next paragraph), which can be constructed of size $\lambda$, and hence, inside $R_{\alpha}$ for some $\alpha$ relatively small with respect to $\kappa^*$. As $\kappa^*$ is \emph{large enough}, we actually have that the collection of nice $\Qor$-names of reals and members of $\SNwf$ are in $R_{\kappa^*}$, so $R_{\kappa^*}$, and hence $V$ (where eventual GCH was not assumed), satisfies that $\Q$ is as required.

    Hence, without loss of generality, we assume that GCH holds above some regular cardinal $\theta^-_1>\lambda$. Let $\theta^-_i$ and $\theta_i$ be ordered as in~\autoref{fig:cichoncoll}, where the inequalities between them are strict.

\begin{figure}[ht]
\centering
\begin{tikzpicture}[xscale=2/1]
\footnotesize{
\node (addn) at (0,3){$\theta_1$};
\node (covn) at (0,6){$\theta_2$};
\node (covsn) at (0.7,7.5){$\theta_3$};
\node (nonn) at (3.8,3) {$\theta_6$} ;
\node (nonsn) at (3.2,3) {$\theta_6$} ;
\node (cfn) at (3.8,6) {$\theta_6$} ;
\node (addm) at (1.4,3) {$\bullet$} ;
\node (covm) at (2.6,3) {$\theta_6$} ;
\node (nonm) at (1.4,6) {$\theta_5$} ;
\node (cfm) at (2.6,6) {$\bullet$} ;
\node (b) at (1.4,4.5) {$\theta_4$};
\node (d) at (2.6,4.5) {$\theta_6$};
\node (c) at (5,6) {$\theta_6$};

\draw[gray]
      (covn) edge [->] (nonm)
      (nonm)edge [->] (cfm)
      (cfm)edge [->] (cfn)
      (addn) edge [->]  (addm)
      (addm) edge [->]  (covm)
      (addm) edge [->] (b)
      (d)  edge[->] (cfm)
      (b) edge [->] (d)
      (nonn) edge [->]  (cfn)
      (cfn) edge[->] (c)
      (covm) edge [->] (d)
      (covm) edge [->]  (nonsn)
      (nonsn) edge [->]  (nonn);

\draw[sub3,line width=.05cm](addn) edge[->] node[left] {$\theta^-_2$} (covn);

\draw[sub3,line width=.05cm] (covn) edge[->] node[above left] {$\theta^-_3$} (covsn);
\draw[sub3,line width=.05cm] (covsn) edge[->] node[below left] {$\theta^-_4$} (b);

\draw[sub3,line width=.05cm] (b)  edge [->] node[midway, right] {$\theta^-_5$} (nonm);

\draw[sub3,line width=.05cm,dashed] (nonm) edge[->] (covm);

\node (aleph1) at (-1,-3) {$\aleph_1$};
\node (addn-f) at (0,-3){$\lambda_1^\bfrak$};
\node (covn-f) at (0,0){$\lambda_2^\bfrak$};
\node (covsn-f) at (0.7,1.5){$\lambda_3^\bfrak$};
\node (nonn-f) at (3.8,-3) {$\lambda_2^\dfrak$} ;
\node (nonsn-f) at (3.2,-3) {$\lambda_3^\dfrak$} ;
\node (cfn-f) at (3.8,0) {$\lambda_1^\dfrak$} ;
\node (addm-f) at (1.4,-3) {$\bullet$} ;
\node (covm-f) at (2.6,-3) {$\lambda_5^\dfrak$} ;
\node (nonm-f) at (1.4,0) {$\lambda_5^\bfrak$} ;
\node (cfm-f) at (2.6,0) {$\bullet$} ;
\node (b-f) at (1.4,-1.5) {$\lambda_4^\bfrak$};
\node (d-f) at (2.6,-1.5) {$\lambda_4^\dfrak$};
\node (c-f) at (5,0) {$\lambda$};

%%%%  diagram 1a
\draw[sub3,line width=.05cm,dashed] (c-f) edge[->] node[above] {$\theta^-_1$} (addn);

\draw[gray]
   (covn-f) edge [->] (nonm-f)
   (nonm-f)edge [->] (cfm-f)
   (cfm-f)edge [->] (cfn-f)
   (addn-f) edge [->]  (addm-f)
   (addm-f) edge [->]  (covm-f)
   (covm-f) edge [->]  (nonsn-f)
   
   (addm-f) edge [->] (b-f)
   (d-f)  edge[->] (cfm-f)
   (b-f) edge [->] (d-f);

\draw[,sub3,line width=.05cm,dashed] 
             (covsn-f) edge [->] (b-f)
              (nonm-f) edge [->] (covm-f)
              (d-f) edge [->] (nonsn-f);

\draw[sub3,line width=.05cm](aleph1) edge[->] (addn-f)
(addn-f) edge[->] (covn-f)
(covn-f) edge [->] (covsn-f)
(b-f)  edge [->] (nonm-f)
(covm-f) edge [->] (d-f)
(nonsn-f) edge [->]  (nonn-f)
(nonn-f) edge [->]  (cfn-f)
(cfn-f) edge[->] (c-f);

}
\end{tikzpicture}
\caption{Strategy to force Cicho\'n's maximum: we construct a ccc poset $\Por$ forcing the constellation at the top, and find a $\sigma$-closed model $N$ such that $\Por\cap N$ forces the constellation at the bottom.}\label{fig:cichoncoll}
\end{figure}
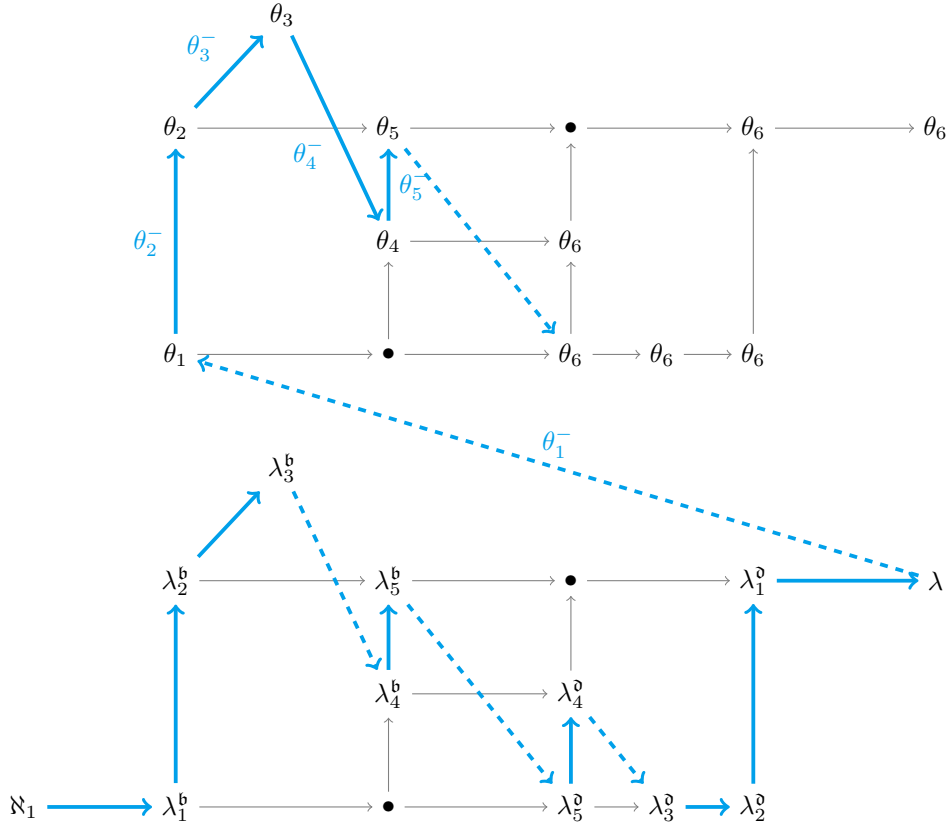

So let $\Por$ be a ccc poset as in \autoref{thm:4SN}. According to the proof, $\Por$ forces more things:
    \begin{enumerate}[label = \rm (\alph*)]
        \item\label{maxa} $\cfrak = \theta_6 = |\Por|$;
        \item\label{maxb} $\Cbf_{[\theta_6]^{<\theta_1}} \leqT \Sbf^{f_0}$;
        \item\label{maxc} $\Rbf_i \eqT \Cbf_{[\theta_6]^{<\theta_i}}$ for $1\leq i\leq 4$, even more,
        \item\label{maxd} for $i=3$, $\aLc(b,1) \eqT \Iwf_g \eqT \Cbf_{[\theta_6]^{<\theta_3}}$;
        \item\label{maxe} $\theta_5 \leqT \Cbf_\Mwf$, $\theta_6 \leqT \Cbf_\Mwf$ and $\Ed\leqT\theta_6\times\theta_5$;
        \item\label{maxf} $\theta_6^{\theta_6}\leqT\SNwf\leqT\Cbf_{[\theta_6]^{<\theta_1}}^{\theta_6}$
    \end{enumerate}
    The parameters $f_0,b,g\in\baireincr$ can be found in the ground model such that $f_0(i)=\frac{(i+1)(i+2)(2i+3)}{6}$ for any $i<\omega$ and such that ZFC proves $\Cbf_{\Iwf_g}^\perp \leqT \aLc(b,1)$. Recall that ZFC proves $\Cbf^\perp_\SNwf \leqT \Cbf^\perp_{\Iwf_g}$ and $\Cbf_\Mwf \leqT \Ed$. 
    
    For each $1\leq i\leq 4$ let $S_i:=[\theta_6]^{<\theta_i}\cap V$, which is a directed preorder. By GCH above $\theta^-_1$, we obtain that $S_i\eqT \Cbf_{[\theta_6]^{<\theta_i}}\cap V$ and that any ccc poset forces $S_i\eqT \Cbf_{[\theta_6]^{<\theta_i}} \eqT [\theta_6]^{<\theta_i}$. Therefore, $\Cbf_{[\theta_6]^{<\theta_i}}$ can be replaced by $S_i$ in~\ref{maxb}--\ref{maxd}.
    
    Following the same dynamics as in~\cite[Sec.~5]{CM22}, we can construct a $\sigma$-closed $N\preceq H_\chi$ of size $\lambda$ such that, for any $1\leq i\leq 4$,
    \begin{enumerate}[label = \rm (\alph*')]
        \item $S_i\cap N \leqT \prod_{j=i}^5 \lambda^\dfrak_j\times \lambda^\bfrak_j$,
        \item $\lambda^\bfrak_i\leqT S_i\cap N$ and $\lambda^\dfrak_i\leqT S_i\cap N$,
        \item $\theta_5\cap N\eqT \lambda^\bfrak_5$ and $\theta_6\cap N \eqT \lambda^\dfrak_5$.
    \end{enumerate}
    Then, by~\ref{maxa}--\ref{maxe}, \autoref{KcapN},~\ref{SN_N} and~\ref{SfYorio}, $\Por\cap N$ is a required.
\end{proof}

We do not know the effect of $\Por\cap N$ on~\cref{maxf}, particularly on $\theta_6^{\theta_6}\leqT\SNwf$. However, we have the following cases where we can force a value to $\cof(\SNwf)$.

\begin{theorem}\label{cor1:4SNmax}
    Under the assumptions of \autoref{thm:4SNmax}, if 
    $\cof\left(([\lambda^\dfrak_1]^{<\lambda^\bfrak_1})^{\lambda^\dfrak_4}\right) = \lambda^\dfrak_1$ 
    then the ccc poset constructed in \autoref{thm:4SNmax} forces $\cof(\SNwf) = \lambda^\dfrak_1$ (see~\autoref{cichonmaxSN2}). %$\cof(\SNwf)=\lambda^\dfrak_1$.
\end{theorem}

Note that the additional assumption implies $\lambda^\dfrak_4 < \lambda^\dfrak_1$, see~\cite[Lem.~3.10]{CM23}.

\begin{proof}
Let $\Qor$ be the poset obtained in~\autoref{thm:4SNmax}. Since $\Qor$ forces $\lambda_1^\dfrak\leq\cof(\SNwf)$ along with $\cof\left(([\lambda^\dfrak_1]^{<\lambda^\bfrak_1})^{\lambda^\dfrak_4}\right) = \lambda^\dfrak_1$ (which is preserved by ccc posets), it is forced that $\cof(\SNwf)=\lambda_1^\dfrak$ by~\autoref{new_upperb}.
\end{proof}

We do not know how to force $\cof(\SNwf)$ larger in this context, unless we restrict to $\cov(\Mwf) = \cof(\Nwf)$:

\begin{theorem}\label{cor2:4SNmax}
    Under the assumptions of \autoref{thm:4SNmax}, assume in addition that $\lambda^\dfrak_5=\lambda^\dfrak_1 = \kappa$ and $\dfrak_{\kappa} = \cof\left(([\kappa]^{<\lambda^\bfrak_1})^{\kappa}\right) = \nu$. Then the ccc poset constructed in \autoref{thm:4SNmax} forces $\cof(\SNwf) = \nu$ (see~\autoref{Figcor24SNmax}).
\end{theorem}

The assumption implies that $\kappa<\nu$, but any of the cases $\nu<\lambda$, $\nu=\lambda$ and $\lambda<\nu$ is allowed. This result strengthens~\cite[Thm.~4.6]{cardona} and \autoref{thm:4SN}.

\begin{figure}[ht!]
\begin{center}
  \includegraphics[width=\linewidth]{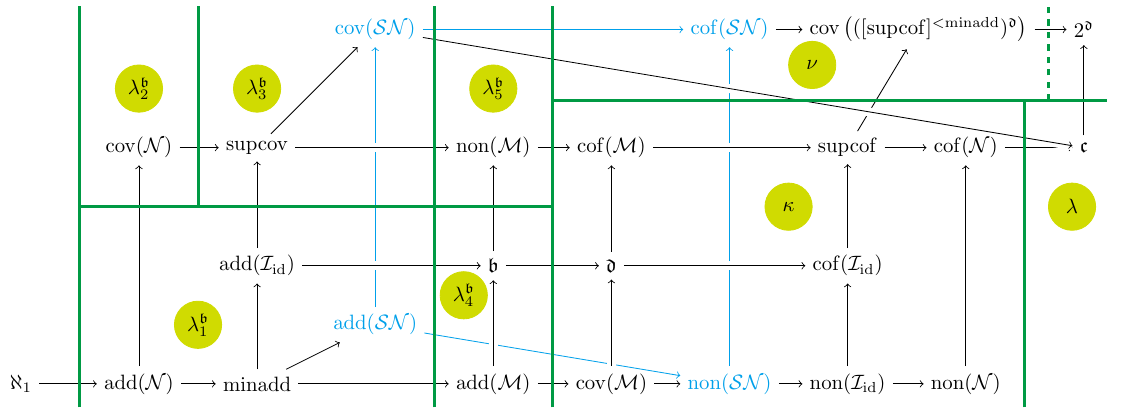}
 \caption{The constellation forced in~\autoref{cor2:4SNmax}.} 
  \label{Figcor24SNmax}
\end{center}
\end{figure}

\begin{proof}
As in~\autoref{cor1:4SNmax}, let $\Qor$ be the poset obtained in~\autoref{thm:4SNmax}. Since $\Qor$
forces $\cov(\Mwf)=\dfrak=\kappa$, together with $\dfrak_{\kappa} = \cof\left(([\kappa]^{<\lambda^\bfrak_1})^{\kappa}\right) = \nu$ (which is preserved by ccc posets), it is forced that $\cof(\SNwf)=\nu$ by~\autoref{new_upperb} and~\ref{cor:lowSN}.   
\end{proof}

\begin{figure}[ht!]
\begin{center}
  \includegraphics[width=\linewidth]{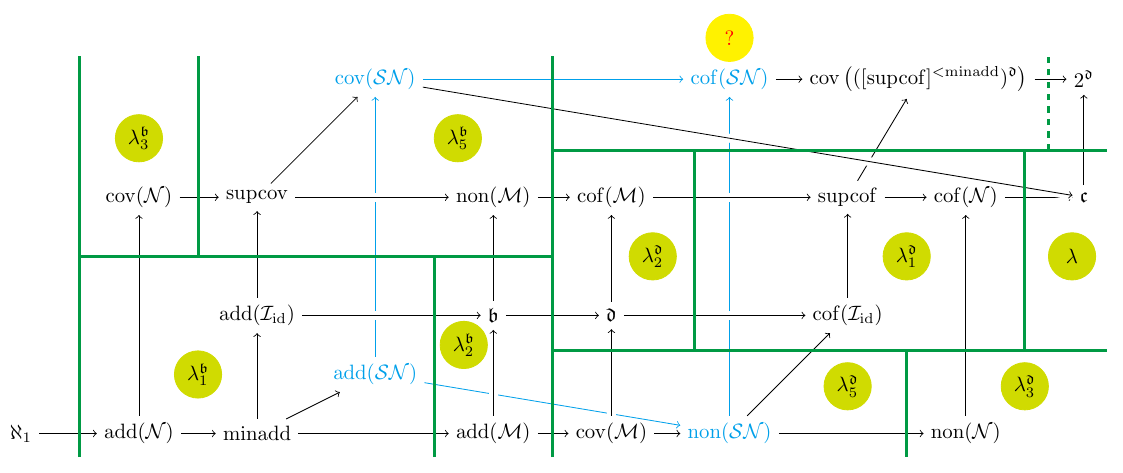}
 \caption{Constellation forced in~\autoref{KST4SNmax}.} 
  \label{FigKST4SNmax}
\end{center}
\end{figure}

In a similar way, \autoref{KST4SN} and~\ref{thm:4SN3} can be expanded as follows.

\begin{theorem}\label{KST4SNmax}
    Under the same assumptions of \autoref{thm:4SNmax}, there is a ccc poset forcing the constellation in~\autoref{FigKST4SNmax}.
    \begin{align*}
    \aleph_1 \leq & \add(\Nwf)=\add(\SNwf) = \add(\Iwf_\id) = \lambda^\bfrak_1 \leq \bfrak = \lambda^\bfrak_2 \leq \cov(\Nwf) = \lambda^\bfrak_3 \leq\\
    \leq & \cov(\SNwf) = \supcov = \non(\Mwf) = \lambda^\bfrak_5 \leq \cov(\Mwf) = \non(\SNwf) = \lambda^\dfrak_5 \leq\\ 
    \leq & \non(\Nwf) = \lambda^\dfrak_3 \leq
    \dfrak = \lambda^\dfrak_2 \leq \cof(\Nwf) = \cof(\Iwf_\id) = \lambda^\dfrak_1 \leq \cfrak =\lambda \text{ and } \lambda^\dfrak_1 \leq \cof(\SNwf).
    \end{align*}
    Furthermore:
    \begin{enumerate}[label = \rm (\alph*)]
        \item If (in the ground model) $\cof\left(([\lambda^\dfrak_1]^{<\lambda^\bfrak_1})^{\lambda^\dfrak_4}\right) = \lambda^\dfrak_1$ then the same ccc poset forces $\cof(\SNwf) = \lambda^\dfrak_1$.

        \item If $\lambda^\dfrak_5=\lambda^\dfrak_1 = \kappa$ and $\dfrak_{\kappa} = \cof\left(([\kappa]^{<\lambda^\bfrak_1})^{\kappa}\right) = \nu$ then the same ccc poset forces $\cof(\SNwf) = \nu$.
    \end{enumerate}
\end{theorem}

\begin{theorem}\label{4SN3max}
    Under the same assumptions of \autoref{thm:4SNmax}, there is a ccc poset forcing the constellation in~\autoref{Fig4SN3max}.
    \begin{align*}
    \aleph_1 \leq & \add(\Nwf)=\add(\SNwf) = \add(\Iwf_\id) = \lambda^\bfrak_1 \leq \cov(\Nwf) = \lambda^\bfrak_2 \leq \bfrak = \lambda^\bfrak_3 \leq\\
    \leq & \cov(\SNwf) = \supcov =  \non(\Mwf) = \lambda^\bfrak_5 \leq
    \cov(\Mwf) = \non(\SNwf) = \lambda^\dfrak_5 \leq\\ 
    \leq  & \dfrak = \lambda^\dfrak_3 \leq  \non(\Nwf) = \lambda^\dfrak_2 \leq \cof(\Nwf) = \cof(\Iwf_\id) = \lambda^\dfrak_1 \leq \cfrak =\lambda \text{ and } \lambda^\dfrak_1 \leq \cof(\SNwf).
    \end{align*}
    Furthermore:
    \begin{enumerate}[label = \rm (\alph*)]
        \item If (in the ground model) $\cof\left(([\lambda^\dfrak_1]^{<\lambda^\bfrak_1})^{\lambda^\dfrak_4}\right) = \lambda^\dfrak_1$ then the same ccc poset forces $\cof(\SNwf) = \lambda^\dfrak_1$.

        \item If $\lambda^\dfrak_5=\lambda^\dfrak_1 = \kappa$ and $\dfrak_{\kappa} = \cof\left(([\kappa]^{<\lambda^\bfrak_1})^{\kappa}\right) = \nu$ then the same ccc poset forces $\cof(\SNwf) = \nu$.
    \end{enumerate}
\end{theorem}
\begin{figure}[ht!]
\begin{center}
  \includegraphics[width=\linewidth]{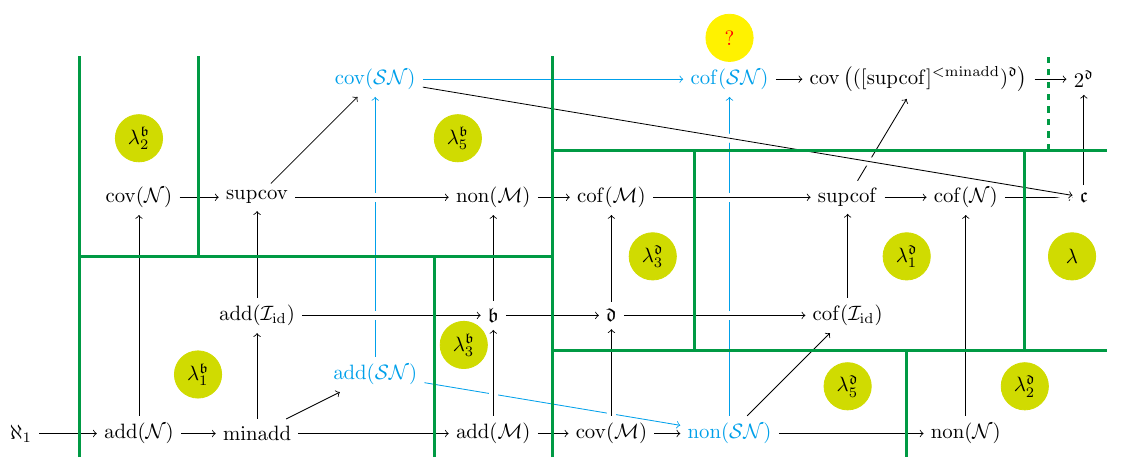}
 \caption{Constellation forced in~\autoref{4SN3max}.} 
  \label{Fig4SN3max}
\end{center}
\end{figure}

\begin{remark}\label{remcovnon}
    For the same reasons explained in \autoref{rem:covIf}, it is unclear whether all coverings of Yorioka ideals can be forced to be the same in all the results of this section. Moreover, after intersecting with submodels, it could happen that at least two uniformities of Yorioka ideals are different, but it is unclear how to modify the parameters of the construction so that they are all forced to be the same.
\end{remark}

\section{Discussions and open questions}\label{sec:Q}

We forced the four cardinal characteristics associated with $\SNwf$ pairwise different along with Cicho\'n's maximum, but we still do not know how to separate additional values in~\autoref{Cichonwith_SN}. For instance, we could ask the following:

\begin{question}\label{Oaddsn} 
Are each of the following statements consistent with $\thzfc$? 
\begin{enumerate}[label=\rm(\arabic*)]
    \item\label{Oaddsnone}  $\add(\Nwf)<\add(\SNwf)<\bfrak$.
    
    \item\label{Oaddsntwo} $\add(\Nwf)<\bfrak<\add(\SNwf)$.   
\end{enumerate}
\end{question}

One natural way to solve~\autoref{Oaddsn} would be to find a poset that increases $\add(\SNwf)$ while preserving $\add(\Nwf)$ (and $\bfrak$) small. 
It is known from Pawlikowski~\cite{paw85} that $\minLc \leq \add(\SNwf)$ where 
\[\minLc := \min\set{|F|}{\exists\, b\in\omega^\omega\colon F\subseteq \prod b \text{ and }\neg\exists\, \varphi\in \Swf(b,\id)\ \forall\, x\in F\colon x\in^* \varphi},\]
which he used to force $\bfrak<\minLc=\cfrak$.
However,  $\add(\Nwf) = \min\{\bfrak,\minLc\}$ (see e.g.~\cite[Lem.~3.11]{CM}), so we cannot increase $\bfrak$ and $\minLc$ simultaneously, while forcing $\add(\Nwf)$ small. 

%During the referee review of this paper, the second author with Miroslav Repick\'y and Saharon Shelah solved~\autoref{Oaddsn} (see~\cite{CRS}).

Concerning Yorioka ideals, it is known that many coverings and uniformities can be forced to be different (even with continuum many different values, see~\cite{CKM}). However, the consistency of two different additivity numbers is unknown, likewise for the cofinality. Although the consistency of $\add(\Nwf)<\add(\Iwf_f)<\cof(\Iwf_f)<\cof(\Nwf)$ for some $f$ is known~\cite[Thm.~5.12]{CM}, it is unknown whether ZFC proves $\minadd = \add(\Nwf)$ and $\supcof = \cof(\Nwf)$, or $\minadd =\add(\Iwf_\id)$ and $\supcof = \cof(\Iwf_\id)$. These questions are addressed in~\cite[Sec.~6]{CM}.

With respect to \autoref{rem:covIf} and~\ref{remcovnon}, we wonder about the consistency of $\cov(\Nwf)<\cov(\Iwf_\id)$ and $\non(\Iwf_\id) < \non(\Nwf)$. This was indeed proved by Goto~\cite{goto}, but not in the context of FS iterations. Solving this consistency for finite support iterations would give tools to solve the problems addressed in these remarks, namely, forcing that all coverings of Yorioka ideals are equal in the models constructed in \autoref{SecConstr} and~\ref{sec:subm}, likewise for the uniformities.

As mentioned in the introduction, the first proof of Cicho\'n's maximum used strongly compact cardinals~\cite{GKScicmax}. Concretely, a ccc poset forcing that separates the values of the left side of Cicho\'n's diagram is produced, and Boolean ultrapowers are applied to additionally separate the right side. The dynamics are similar to the proofs in \autoref{sec:subm}: given a strongly compact cardinal $\kappa$ and an elementary embedding $j\colon V\to N$ corresponding to a Boolean ultrapower, if $\Por$ is a ccc poset that forces $K\leqT \Rbf$, where $\Rbf$ is a definable relational system of the reals, then $j(\Por)$ forces $j(K) \leqT \Rbf$ (with respect to $V$), and the same applies to $\Rbf\leqT K$ (note the similarity with \autoref{KcapN}). So, to understand the values $j(\Por)$ forces to $\bfrak(\Rbf)$ and $\dfrak(\Rbf)$, it is enough to calculate $\bfrak(j(K))$ and $\dfrak(j(K))$. 

Concerning $\SNwf$, we can show similar results as \autoref{SN_N}, namely, $\Vdash_\Por K\leqT\Sbf^{f_0}$ implies $\Vdash_{j(\Por)} j(K) \leqT \Sbf^{f_0}$, and likewise when replacing $\Sbf^{f_0}$ by $\Cbf^\perp_\SNwf$. As a consequence, Boolean ultrapowers can be used to obtain constellations as in \autoref{sec:subm}. However, we do not know the effect of $j$ to Tukey connections for, e.g.\ $K\leqT \SNwf$ and $\Cbf^\perp_{\SNwf} \leqT K$. The reason is that, although any nice $j(\Por)$-name of a real is always in $N$, we cannot state the same about nice $j(\Por)$-names of strong measure zero sets.

Many questions related to ~\autoref{4SN} remain open.  We mention a few:

\begin{question}\label{Question4I}
Is each one of the following statements consistent with ZFC?
\begin{enumerate}[label=\rm(\arabic*)]
    \item\label{Q:4E} $\add(\Ewf)<\cov(\Ewf)<\non(\Ewf)<\cof(\Ewf)$.
    
    \item\label{Q:4M} $\add(\Mwf)<\cov(\Mwf)<\non(\Mwf)<\cof(\Mwf)$.
    
    \item $\add(\SNwf)<\non(\SNwf)<\cov(\SNwf)<\cof(\SNwf)$.
     
    \item\label{Q4I-4} $\add(\Iwf_f)<\non(\Iwf_f)<\cov(\Iwf_f)<\cof(\Iwf_f)$ for any $f\in\omega^\omega$.%\footnote{Since the writing of this paper, this problem has been solved by David Valderrama who showed that consistency indeed holds in the shattered iteration model from~\cite{Brshatt}.}
\end{enumerate}
\end{question}

FS iterations of ccc forcings will not work to solve~\autoref{Question4I} because any such iteration whose length has uncountable cofinality $\pi$ forces $\non(\Mwf)\leq\cf(\pi)\leq\cov(\Mwf)$ and $\cov(\SNwf)\leq \cf(\pi)$ (see \autoref{thm:cfcovSN}), so alternative approaches are required. Roughly speaking, two approaches could be used to solve~\autoref{Question4I}: the first approach is a creature-forcing method based on the notion of decisiveness \cite{KS09,KS12}, developed in \cite{FGKS,CKM}, but this method is restricted to $\omega^\omega$-bounding forcings, that means, results in $\dfrak=\aleph_1$. So this method does not work to solve~\ref{Q:4M}. On the other hand, $\dfrak =\aleph_1$ implies $\add(\Ewf)=\aleph_1$ and $\cov(\Ewf)=\cov(\Nwf)$, but this method tends to use a device called ``rapid reading", which makes the creature construction force $\cov(\Nwf)=\aleph_1$. Therefore, the known creature methods cannot be used to solve~\ref{Q:4E}.

The second approach is the first author's new method of \emph{shattered iterations}~\cite{Brshatt}, but further research is needed to know how to use it to solve these problems. Since the writing of this paper, \autoref{Question4I}~\ref{Q4I-4} has been solved by David Valderrama with this method.

\subsection*{Acknowledgments} The authors would like to thank: Elliot Glazer, for letting us know that eventual~GCH is not essential to prove Cicho\'n's maximum; Andr\'es Uribe-Zapata, for revision and suggestions to this paper, which helped us deliver a more presentable final version; and the anonymous referee for the valuable comments.

{\small
%\addcontentsline{toc}{section}{References}
\bibliography{bibli}
\bibliographystyle{alpha}
}

\end{document}